\documentclass[twoside,a4paper,reqno,11pt]{amsart} 
\usepackage[top=28mm,right=28mm,bottom=28mm,left=28mm]{geometry}

\usepackage{amsfonts, amsmath, amssymb, mathrsfs, bm, latexsym, stmaryrd, array, hyperref, mathtools, bold-extra}

\renewcommand{\a}{\alpha}
\renewcommand{\b}{\beta}

\newcommand{\e}{\varepsilon}
\renewcommand{\l}{\lambda} 

\newcommand{\s}{\sigma}
\renewcommand{\O}{\Omega}

\newcommand{\C}{\mathcal{C}}

\newcommand{\la}{\langle}
\newcommand{\ra}{\rangle}

\renewcommand{\to}{\rightarrow}

\newcommand{\leqs}{\leqslant}
\newcommand{\geqs}{\geqslant}

\newcommand{\vs}{\vspace{2mm}}

\makeatletter
\newcommand{\imod}[1]{\allowbreak\mkern4mu({\operator@font mod}\,\,#1)}
\makeatother

\newtheorem{theorem}{Theorem} 
\newtheorem{theoremm}{Theorem}

\newtheorem*{conj*}{Conjecture}

\newtheorem{conj}[theorem]{Conjecture}

\newtheorem{thm}{Theorem}[section] 
\newtheorem{prop}[thm]{Proposition} 
\newtheorem{lem}[thm]{Lemma}
 
\newtheorem{con}[thm]{Conjecture}

\theoremstyle{definition}
\newtheorem{rem}[thm]{Remark}
\newtheorem{remk}[theorem]{Remark}

\begin{document}

\author{Timothy C. Burness}
\address{T.C. Burness, School of Mathematics, University of Bristol, Bristol BS8 1UG, UK}
\email{t.burness@bristol.ac.uk}

\author{Marco Fusari}
\address{M. Fusari, Dipartimento di Matematica ``Felice Casorati”, University of Pavia, Via Ferrata 5, 27100 Pavia, Italy}
\email{lucamarcofusari@gmail.com}
 
\title[On derangements in simple permutation groups]{On derangements in simple permutation groups}

\dedicatory{\rm Dedicated to Martin Liebeck on the occasion of his 70th birthday} 

\begin{abstract}
Let $G \leqslant {\rm Sym}(\Omega)$ be a finite transitive permutation group and recall that an element in $G$ is a derangement if it has no fixed points on $\O$. Let $\Delta(G)$ be the set of derangements in $G$ and define $\delta(G) = |\Delta(G)|/|G|$ and $\Delta(G)^2 = \{ xy \,:\, x,y \in \Delta(G)\}$. In recent years, there has been a focus on studying derangements in simple groups, leading to several remarkable results. For example, by combining a theorem of Fulman and Guralnick with recent work by Larsen, Shalev and Tiep, it follows that $\delta(G) \geqs 0.016$ and $G = \Delta(G)^2$ for all sufficiently large simple transitive groups $G$. In this paper, we extend these results in several directions. For example, we prove that $\delta(G) \geqs 89/325$ and $G = \Delta(G)^2$ for all finite simple primitive groups with soluble point stabilisers, without any order assumptions, and we show that the given lower bound on $\delta(G)$ is best possible. We also prove that every finite simple transitive group can be generated by two conjugate derangements, and we present several new results on derangements in arbitrary primitive permutation groups.
\end{abstract}

\date{\today}

\maketitle

\setcounter{tocdepth}{1}
\tableofcontents

\section{Introduction}\label{s:intro}

Let $G$ be a finite transitive permutation group on a set $\Omega$ with $|\O| \geqs 2$ and point stabiliser $H = G_{\a}$. Recall that an element of $G$ is a \emph{derangement} if it has no fixed points on $\O$. We write $\Delta(G)$ for the set of derangements in $G$ (sometimes we will use $\Delta(G,\O)$ or $\Delta(G,H)$, if we need to specify $\O$ or $H$). By a classical theorem of Jordan \cite{Jordan2}, published in 1872, we know that $\Delta(G)$ is non-empty. This elementary observation leads naturally to a wide range of problems and applications that have been intensively studied in recent years (for instance, see Serre's article \cite{Serre} for interesting applications in number theory and topology).

In one direction, there is an extensive literature concerning the existence of derangements with specified properties. A well known theorem of Fein, Kantor and Schacher \cite{FKS}, which relies on the Classification of Finite Simple Groups (CFSG), shows that every group $G$ as above contains a derangement of prime power order, which in turn has important number-theoretic applications concerning the structure of Brauer groups of global field extensions (see \cite{FKS}). However, transitive groups with no derangements of prime order do exist. For example, the smallest Mathieu group ${\rm M}_{11}$, viewed as a primitive permutation group of degree $12$, does not contain a prime order derangement. These so called \emph{elusive groups} have been the subject of numerous papers and they are closely related to some interesting open problems, such as the \emph{Polycirculant Conjecture} in algebraic graph theory (see the survey article \cite{AAS} and \cite[Section 1.3]{BGiu} for further details).

In a different direction, it is natural to consider the proportion of derangements in $G$, 
\[
\delta(G) = \delta(G,\O) = \delta(G,H) = \frac{|\Delta(G)|}{|G|},
\]
which one can view as the probability that a uniformly random element in $G$ has no fixed points. This has been widely studied since the 1990s and there has been a special interest in determining lower bounds. For example, a theorem of Cameron and Cohen \cite{CC} shows that $\delta(G) \geqs |\O|^{-1}$, with equality if and only if $G$ is sharply $2$-transitive. And by applying CFSG, the groups with $\delta(G) < 2|\O|^{-1}$ have been determined by Guralnick and Wan \cite{GW}, motivated by applications to curves over finite fields in arithmetic geometry. The latter two results have been generalised in a recent preprint of Garzoni \cite{Garzoni}, where the  main theorem states that if $G$ is primitive and $|\O|$ is sufficiently large, then either $G$ is a Frobenius group or 
\[
\delta(G) \geqs \frac{|\O|^{\frac{1}{2}}+1}{2|\O|}.
\]
This lower bound is best possible and it settles a conjecture of Guralnick and Tiep \cite[p.272]{GT03} on primitive affine groups.

Perhaps the most striking result on the proportion of derangements is the following deep theorem of Fulman and Guralnick, which establishes a conjecture of Boston and Shalev from the 1990s (see \cite{Boston}). The proof is presented in the sequence of papers \cite{FG1, FG2, FG3, FG4}.

\begin{theorem}[Fulman \& Guralnick]\label{t:FG0}
There is an absolute constant $\e > 0$ such that $\delta(G) \geqs \e$ for every finite simple transitive permutation group $G$.
\end{theorem}

The constant $\e$ is undetermined, although \cite[Theorem 1.1]{FG4} states that one can take $\e = 0.016$ for all sufficiently large simple groups. 

Since $\Delta(G)$ is a normal subset, it is also natural to consider the analogous problem for conjugacy classes. Here recent work of Eberhard and Garzoni \cite{EG} shows that the proportion of conjugacy classes consisting of derangements in simple transitive groups of Lie type is also bounded away from zero (it is easy to see that  the conclusion is false for alternating groups). It is interesting to note that the latter result extends to almost simple groups of Lie type, whereas examples can be constructed to show that $\delta(G)$ can be arbitrarily small in the almost simple setting (we will return to this below).

In order to state our first result, it will be convenient to define
\[
\a(G) = \min\{\delta(G,H) \,:\, \mbox{$H < G$ core-free}\},
\]
where we recall that a subgroup $H$ of $G$ is \emph{core-free} if $\bigcap_{g \in G}H^g = 1$ (equivalently, the natural transitive action of $G$ on $G/H$ is faithful). Then Theorem \ref{t:FG0} implies that there is an absolute (and undetermined) constant $\e>0$ such that $\a(G) \geqs \e$ for every simple group $G$. In Theorem \ref{t:main1} below, we show that $\e = 1/e$ is asymptotically the best possible constant for alternating groups, and we determine the optimal constant for every sporadic group. Note that Table \ref{tab:spor} is presented in Section \ref{ss:prop_spor}. Throughout this paper, whenever we use the term ``simple group", we implicitly assume the group is non-abelian (and we adopt the notation for simple groups used in \cite{KL}).

\begin{theoremm}\label{t:main1}
Let $G$ be a simple alternating or sporadic group. 

\vspace{1mm}

\begin{itemize}\addtolength{\itemsep}{0.2\baselineskip}
\item[{\rm (i)}] If $G = A_n$, then $\a(G) \to 1/e$ as $n \to \infty$.
\item[{\rm (ii)}] If $G$ is a sporadic group, then $\a(G)$ is recorded in Table \ref{tab:spor}. In particular, we have $\a(G) \geqs 2197/7425$, with equality if and only if $G$ is the McLaughlin group ${\rm McL}$.
\end{itemize}
\end{theoremm}

\begin{remk}
Let us record some comments on the statement of Theorem \ref{t:main1}.

\vspace{1mm}

\begin{itemize}\addtolength{\itemsep}{0.2\baselineskip}
\item[{\rm (a)}] To prove part (i), we will show that  
\[
\a(A_n) = \sum_{j=2}^{n}\frac{(-1)^j}{j!} - \frac{(-1)^n(n-1)}{n!} = \frac{[n!/e]}{n!} - \frac{(-1)^n(n-1)}{n!}
\]
for all sufficiently large $n$, where $[x]$ denotes the nearest integer to $x$. 
Here the key step is to show that if we take the natural action of $G = A_n$ on the set of $k$-element subsets of $\{1, \ldots, n\}$, then $\delta(G)$ is minimal when $k=1$ or $n-1$ (see Proposition \ref{p:alt_intrans}). 

\item[{\rm (b)}] We conjecture that the formula for $\a(A_n)$ in (a) holds for all $n \geqs 9$, which would imply that $\e = 13/45$ is the optimal constant in Theorem \ref{t:FG0}  as we range over all simple alternating groups (see Conjecture \ref{c:1}).

\item[{\rm (c)}] The proof for sporadic groups relies on computational methods, working closely with the information on character tables and fusion maps available in the \textsf{GAP} Character Table Library \cite{GAPCTL}. 
\end{itemize}
\end{remk}

Our next result gives an effective and best possible version of Theorem \ref{t:FG0} for simple primitive groups with soluble point stabilisers. Here the possibilities for $G$ and $H$ have been determined by Li and Zhang \cite{LZ} (see \cite[Tables 14-20]{LZ}). For example, if $G = A_n$ is an alternating group and $n \geqs 17$, then $n = p$ is a prime and $H = {\rm AGL}_1(p) \cap G$ is the only possibility (up to conjugacy in $G$).

\begin{theoremm}\label{t:main2}
Let $G$ be a finite simple primitive permutation group with soluble point stabiliser $H$. Then  
\[
\delta(G) \geqs \frac{89}{325},
\]
with equality if and only if $G = {}^2F_4(2)'$ and $H = 2^2.[2^8].S_3$.
\end{theoremm}

\begin{remk}\label{r:main2}
Some comments on the statement of Theorem \ref{t:main2} are in order.

\vspace{1mm}

\begin{itemize}\addtolength{\itemsep}{0.2\baselineskip}
\item[{\rm (a)}] For sporadic groups, the given lower bound is valid for all transitive actions with soluble point stabilisers. In fact, we prove that $\delta(G) \geqs 21/55$, with equality if and only if $G = {\rm M}_{11}$ and $H = {\rm U}_{3}(2).2$ or $2.S_4$ (see Proposition \ref{p:spor_sol}).

\item[{\rm (b)}] In order to establish Theorem \ref{t:main3} below, we require a slightly more general version of Theorem \ref{t:main2}. So we will actually prove that $\delta(G) \geqs 89/325$ for every finite simple transitive group $G$ with soluble point stabiliser $H$, where $H = G \cap M$ for some maximal subgroup $M$ of an almost simple group with socle $G$.

\item[{\rm (c)}] We conjecture that $\e = 89/325$ is the optimal constant for all transitive groups with soluble point stabilisers. In fact, we speculate that this is the best possible constant in Theorem \ref{t:FG0}, without any additional assumptions.
\end{itemize}
\end{remk}

As noted above, Theorem \ref{t:FG0} does not extend to almost simple groups. For example, as explained in \cite[Section 6]{FG1}, if we take $G = {\rm Aut}({\rm L}_2(p^r)) = {\rm PGL}_2(p^r){:}\la \varphi \ra$ and $\O = \varphi^G$, where $p$ and $r$ are primes with $(r,p(p^2-1))=1$ and $\varphi$ is a field automorphism of order $r$, then every element in $G \setminus {\rm PGL}_2(p^r)$ has a fixed point and thus 
\[
\delta(G) \leqs \frac{|{\rm PGL}_2(p^r)|}{|G|}  = \frac{1}{r}
\]
can be arbitrarily small. 

In this example, notice that $\delta(G) < 3/\log n$ if $p=2$ and $r \geqs 5$, where $n = |\O|$ is the degree of $G$ and $\log$ is the natural logarithm. This shows that \cite[Theorem 1.5]{FG1} is essentially best possible since it states that there exists an absolute constant $\gamma>0$ such that $\delta(G) \geqs \gamma/\log n$ for all almost simple primitive permutation groups $G$ of degree $n$ (moreover, this extends to all non-affine primitive groups). Here we use Theorem \ref{t:FG0} to establish another natural extension to primitive permutation groups (recall that the \emph{socle} of $G$ is the product of its minimal normal subgroups).

\begin{theoremm}\label{t:main3}
Let $G$ be a finite primitive group with socle $N$ and point stabiliser $H$. Then the following hold:

\vspace{1mm}

\begin{itemize}\addtolength{\itemsep}{0.2\baselineskip}
\item[{\rm (i)}] $\delta(N) \geqs \e$, where $\e$ is the constant in Theorem \ref{t:FG0}. 

\item[{\rm (ii)}] If $H$ is soluble, then $\delta(N) \geqs 89/325$, with equality if and only if $N = {}^2F_4(2)'$ and $H \cap N = 2^2.[2^8].S_3$.
\end{itemize}
\end{theoremm}

Further motivation for studying derangements for simple groups stems from a theorem of Larsen, Shalev and Tiep \cite{LST}. In order to set the scene, let $G$ be a finite transitive permutation group and recall that $\Delta(G)$ is a normal subset of $G$. We define the \emph{derangement width} of $G$, denoted $w(G)$, to be the minimal positive integer $k$ such that $G = \Delta(G)^k$, where
\[
\Delta(G)^k = \{ x_1 \cdots x_k \,: \, \mbox{$x_i \in \Delta(G)$ for all $i$} \},
\]
setting $w(G) = \infty$ if $G \ne \Delta(G)^k$ for all $k$. Notice that if $G$ is simple then $G = \la \Delta(G) \ra$ and $w(G)$ is simply the diameter of the corresponding Cayley graph ${\rm Cay}(G,\Delta(G))$. 

There is an extensive literature on so-called \emph{width problems} for finite groups, and for finite simple groups and normal subsets in particular (we refer the reader to Liebeck's excellent survey article \cite{Lie_sur}). One of the main open problems in this area is a conjecture from the 1980s attributed to John Thompson, which asserts that every finite simple group $G$ has a conjugacy class $C$ such that $G = C^2$. This has turned out to be a very difficult problem, but there has been significant progress in recent years. For instance, through the work of several authors, the problem has been reduced to groups of Lie type defined over fields with at most $8$ elements. It is also worth noting that Thompson's conjecture immediately implies a famous conjecture of Ore from 1951, which asserts that every element in a finite simple group is a commutator. The proof of the latter was completed by Liebeck et al. in \cite{LOST} using character-theoretic methods.

It is therefore natural to study the derangement width of finite simple groups and the main result here is the following theorem from \cite{LST}.

\begin{theorem}[Larsen, Shalev \& Tiep]
Let $G$ be a finite simple transitive group.

\vspace{1mm}

\begin{itemize}\addtolength{\itemsep}{0.2\baselineskip}
\item[{\rm (i)}] If $|G|$ is sufficiently large, then $G = \Delta(G)^2$.
\item[{\rm (ii)}] If $G = A_n$, then $G = \Delta(G)^2$ for all $n$.
\end{itemize}
\end{theorem}

It is conjectured in \cite{LST} that the condition on $|G|$ in part (i) is not needed.

\begin{conj}[Larsen, Shalev \& Tiep]\label{c:lst00}
We have $G = \Delta(G)^2$ for every finite simple transitive group $G$.
\end{conj}

It seems difficult to approach this conjecture for groups of Lie type with the character-theoretic methods adopted in \cite{LST}, so we focus on some special cases of interest in Theorems \ref{t:main4} and \ref{t:main_new} below. Note that $\Delta(G)$ is inverse-closed, so $\Delta(G)^2$ always contains the identity element. In part (ii) of the next result, recall that the rank one groups of Lie type are the following:
\begin{equation}\label{e:rank1}
{\rm L}_2(q), \; q \geqs 4; \;\; {\rm U}_3(q),\; q \geqs 3; \;\; {}^2B_2(q), \; q \geqs 8; \;\; {}^2G_2(q)', \; q \geqs 3
\end{equation}

\begin{theoremm}\label{t:main4}
Let $G \leqs {\rm Sym}(\O)$ be a finite simple transitive group with point stabiliser $H$.

\vspace{1mm}

\begin{itemize}\addtolength{\itemsep}{0.2\baselineskip}
\item[{\rm (i)}] If $G$ is an alternating group, a sporadic group, or a rank one group of Lie type, then there exist conjugacy classes $C,D$ of derangements such that 
\begin{equation}\label{e:CDD}
G = \left\{ \begin{array}{ll}
C^2 \cup CD & \mbox{if $G = {\rm L}_2(7)$ and $H = S_4$} \\
\{1\} \cup CD & \mbox{otherwise.}
\end{array}\right.
\end{equation}

\item[{\rm (ii)}] If $G$ is primitive and $H$ is soluble, then $G = \Delta(G)^2$.
\end{itemize}
\end{theoremm}

\begin{remk}\label{r:main4}
Let us record some comments on Theorem \ref{t:main4}.

\vspace{1mm}

\begin{itemize}\addtolength{\itemsep}{0.2\baselineskip}
\item[{\rm (a)}] It is easy to show that $G = \Delta(G)^2$ if $\delta(G) > 1/2$ (see Lemma \ref{l:easy2}), so there is a natural connection between Theorem \ref{t:main2} and part (ii) of Theorem \ref{t:main4}. This observation will be useful in several places (for example, see the proofs of  Theorems \ref{t:main_new} and \ref{t:width_gen} in Sections \ref{ss:width_lie} and \ref{ss:prim_width}, respectively, and also the proof of Proposition \ref{p:class_sol}(ii) in Section \ref{ss:class_sol}).

\item[{\rm (b)}] For groups of Lie type, we rely heavily on earlier work of Guralnick and Malle \cite{GM,GM2}, where character-theoretic methods are used to show that $G = \{1\} \cup CD$ for certain conjugacy classes $C$ and $D$. For example, in proving part (ii) for exceptional groups, we show that there exist classes $C,D$ of derangements with $G = \{1\} \cup CD$.

\item[{\rm (c)}] In part (i), if $G = A_n$ and $n \geqs 9$ then we combine work of Bertram \cite{Bertram}, Brenner \cite{Brenner1} and Larsen and Tiep \cite{LT} to prove that either $G = C^2$ for some conjugacy class $C$ of derangements, or $n \equiv 3 \imod{4}$ and $\O = \{1, \ldots, n\}$ is the natural permutation domain. In the latter case, \cite[Theorem 1]{LT} gives $G = \{1\} \cup D^2$ where $D$ is a class of $n$-cycles, and we conjecture that $G = C^2$ for the class $C$ of elements with cycle-type $(n-4,2^2)$. The latter assertion has been checked computationally for $n \leqs 23$.

\item[{\rm (d)}] If $G$ is a sporadic group, then we use a computational approach to show  that $G = CD$ for conjugacy classes $C,D$ of derangements. Moreover, we can take $C = D$ unless $(G,H)$ is one of the cases in Table \ref{tab:spor1} (see Section \ref{ss:width_spor}), and in the latter cases it is easy to check that $G = \{1\} \cup C^2$, where $C$ is the conjugacy class recorded in the third column of the table.

\item[{\rm (e)}] Suppose $G = {\rm L}_2(7)$ and $H$ is a maximal subgroup isomorphic to $S_4$ (there are two conjugacy classes of such subgroups). Then $\Delta(G) = C \cup D$, where $C$ and $D$ are the two conjugacy classes of elements of order $7$, and one can check that 
\[
C^2 = D^2 = \{ x \in G \,:\, |x| \ne 1,4\},\;\; CD = \{ x \in G \,:\, |x| \ne 2\},
\]
which means that $G \ne \{1\} \cup C^2$, $G \ne \{1\} \cup CD$ and $G = C^2 \cup CD$. In addition, if we consider the action of $G$ on $G/K$ for any proper subgroup $K$ of $H$, then it is easy to show that $G = E^2$ for some conjugacy class $E$ of derangements.

\item[{\rm (f)}] The property $G = \Delta(G)^2$ also extends to some almost simple groups. For instance, in Theorem \ref{t:sym} we generalise \cite[Theorem B]{LST} by showing that $G = \Delta(G)^2$ for every faithful transitive action of $G = S_n$ with $n \geqs 4$. However, it is worth noting that there exist almost simple primitive groups with infinite derangement width, even under the assumption that the point stabilisers are soluble. For example, if we take $G = {\rm Aut}({\rm L}_2(3^r)) = {\rm PGL}_2(3^r){:}\la \varphi \ra$ and $\O = \varphi^G$, where $r \geqs 5$ is a prime, then $H = {\rm PGL}_2(3) \times \la \varphi \ra$ is soluble and every derangement in $G$ is contained in ${\rm PGL}_2(3^r)$. In particular, we have $\Delta(G)^k \subseteq {\rm PGL}_2(3^r) < G$ for every positive integer $k$.
\end{itemize}
\end{remk}

It is easy to show that there are infinitely many simple transitive groups $G$ such that $G \ne \{ 1 \} \cup C^2$ for every conjugacy class $C$ of derangements. For example, if $G = {\rm L}_2(q)$, $q$ is even and $H$ is a Borel subgroup, then 
\[
\Delta(G) = \{ x \in G \,:\, \mbox{$x \ne 1$, $|x|$ divides $q+1$} \} = \bigcup_{i=1}^{k} x_i^G
\]
and one can check that $(x_i^G)^2 = \{ x \in G \,:\, |x| \ne 2 \}$ for all $i$ (see \cite[Theorem 2(ii)]{Garion}). However, in this case we can show that $G = \{1\} \cup (x_1^G)(x_2^G)$. And as far as we are aware, the special case $(G,H) = ({\rm L}_2(7), S_4)$ is the only one with $G \ne \{1\} \cup CD$ for all classes $C,D$ of derangements. This leads us to propose the following stronger form of Conjecture \ref{c:lst00}.

\begin{conj}\label{c:lst0}
Let $G$ be a finite simple transitive group with point stabiliser $H$. Then there exist conjugacy classes $C$ and $D$ of derangements such that \eqref{e:CDD} holds.
\end{conj}

In addition to the groups recorded in part (i) of Theorem \ref{t:main4}, our proof of Theorem \ref{t:main4}(ii) shows that Conjecture \ref{c:lst0} also holds for every non-classical simple primitive group with soluble point stabilisers (see Remark \ref{r:main4}).

By combining a result of Malle, Saxl and Weigel \cite[Theorem 2.1]{MSW} with the main theorems of Lev \cite{Lev} and Guralnick et al. \cite{GPPS}, we can prove a special case of Conjecture \ref{c:lst00} for linear groups of arbitrary rank. In the statement, $P_k$ denotes the stabiliser in $G$ of a $k$-dimensional subspace of the natural module.

\begin{theoremm}\label{t:main_new}
Let $G = {\rm L}_n(q)$ be a finite simple transitive group with point stabiliser $H$.

\vspace{1mm}

\begin{itemize}\addtolength{\itemsep}{0.2\baselineskip}
\item[{\rm (i)}] If $n=3$ and $q \geqs 3$, then $G = \{1\} \cup C^2$ for some conjugacy class $C$ of derangements.

\item[{\rm (ii)}] If $n \geqs 4$ and $H \not\leqs P_1,P_{n-1}$, then $G = \Delta(G)^2$.
\end{itemize}
\end{theoremm}

\begin{remk}
Some comments on the statement of Theorem \ref{t:main_new}:

\vspace{1mm}

\begin{itemize}\addtolength{\itemsep}{0.2\baselineskip}
\item[{\rm (a)}] As recorded in Theorem \ref{t:main4}(i), if $G = {\rm L}_3(2) \cong {\rm L}_2(7)$ then there exist classes $C,D$ of derangements such that 
\[
G = \left\{ \begin{array}{ll}
C^2 \cup CD & \mbox{if $H = P_1$ or $P_2$} \\
\{1\} \cup CD & \mbox{otherwise.}
\end{array}\right.
\]

\item[{\rm (b)}] In part (ii), we work with a slightly stronger form of \cite[Theorem 2.4(i)]{LST} to show that either $G = \{1\} \cup CD$ for classes $C,D$ of derangements (in agreement with Conjecture \ref{c:lst0}), or $q = 2$, $H = {\rm Sp}_n(2)$ and $n \leqs 28$.
 
\item[{\rm (c)}] We refer the reader to Remark \ref{r:psl} for comments on the special case excluded in (ii), where $n \geqs 4$ and $H$ is contained in a $P_1$ or $P_{n-1}$ parabolic subgroup.
\end{itemize}
\end{remk}

Let $G \leqs {\rm Sym}(\O)$ be a finite primitive permutation group with socle $N$ and as before, let $\Delta(N)$ be the set of derangements contained in $N$. If $G$ is almost simple, then $N$ is simple and transitive, so Conjecture \ref{c:lst00} asserts that $N = \Delta(N)^2$. If we assume the veracity of this conjecture, then we can use the O'Nan-Scott theorem to establish the following generalisation.

\begin{theoremm}\label{t:width_gen}
Let $G \leqs {\rm Sym}(\O)$ be a finite primitive permutation group with socle $N$ and $|\O| \geqs 3$. If Conjecture \ref{c:lst00} holds, then $N = \Delta(N)^2$. 
\end{theoremm}

Our final result concerns the generation properties of finite simple transitive groups. It is well known that every finite simple group is $2$-generated (the proof requires CFSG) and there is a vast literature on the properties and distribution of generating pairs. For example, a theorem of Guralnick and Kantor \cite{GK} states that every finite simple group $G$ has a conjugacy class $C$ with the property that for all nontrivial $x \in G$, there exists an element $y \in C$ such that $G = \la x,y \ra$ (in this situation, we call $C$ a \emph{witness}). The proof of this theorem uses probabilistic methods, based on fixed point ratio estimates. Using  similar methods, we prove that every finite simple transitive permutation group is generated by two derangements. Moreover, we can always find a generating pair of conjugate derangements.

\begin{theoremm}\label{t:main6}
Let $G$ be a finite simple transitive group. Then there exist conjugate derangements $x,y \in G$ such that $G = \la x,y \ra$.
\end{theoremm}

Let $G$ be a finite simple group and suppose that $C = x^G$ and $D = y^G$ are  witnesses, as defined above. It is easy to see that the conclusion in Theorem \ref{t:main6} follows if no maximal overgroup of $x$ is conjugate to a maximal overgroup of $y$. Indeed, if this property holds, then either $x$ or $y$ is a derangement and the result follows since $G = \la x,x^a \ra = \la y, y^b\ra$ for some $a,b \in G$. So in the proof of Theorem \ref{t:main6} we are often interested in identifying witnesses $C$ and $D$ with the property that the maximal overgroups of their respective representatives are severely restricted. Along the way, with this aim in mind, we prove that every conjugacy class of Singer cycles in a simple classical group is a witness (see Proposition \ref{p:singer}), which may be of independent interest.

\vs\vs

\noindent \textbf{Notation.} Our notation is standard. For a finite group $G$ and positive integer $n$, we write $C_n$, or just $n$, for a cyclic group of order $n$ and $G^n$ for the direct product of $n$ copies of $G$. An unspecified extension of $G$ by a group $H$ will be denoted by $G.H$; if the extension splits then we may write $G{:}H$. We use $[n]$ for an unspecified soluble group of order $n$. Throughout the paper, we adopt the standard notation for simple groups of Lie type from \cite{KL} (so for example, we write ${\rm L}_n(q) = {\rm L}_n^{+}(q)$ and ${\rm U}_n(q) = {\rm L}_n^{-}(q)$ for the groups ${\rm PSL}_n(q)$ and ${\rm PSU}_n(q)$, respectively). In addition, for positive integers $a$ and $b$, we use the familiar Kronecker delta symbol $\delta_{a,b}$ (so $\delta_{a,b} = 1$ if $a=b$, otherwise $\delta_{a,b} = 0$) and we write $(a,b)$ for the highest common factor of $a$ and $b$.

\vs

\noindent \textbf{Organisation.} Let us briefly outline the structure of the paper. We begin in Section \ref{s:prop} by studying the proportion of derangements in transitive actions of alternating and sporadic groups, culminating in a proof of Theorem \ref{t:main1}. We also establish part (i) of Theorem \ref{t:main3}, which gives a natural extension of Theorem \ref{t:FG0} to primitive groups. In Section \ref{s:width} we turn to the derangement width of simple transitive groups, proving part (i) of Theorem \ref{t:main4}. We also establish our main result on linear groups (Theorem \ref{t:main_new}) and we present a short proof of Theorem \ref{t:width_gen}, which assumes the veracity of Conjecture \ref{c:lst00}. Next, in Section \ref{ss:sol} we focus on primitive simple groups with soluble point stabilisers and our main goal is to prove Theorem \ref{t:main2}. Here we also complete the proofs of Theorems \ref{t:main3} and \ref{t:main4}. Finally, in Section \ref{s:gen} we study the $2$-generation of simple groups and we prove Theorem \ref{t:main6}.

\vs

\noindent \textbf{Acknowledgements.} Both authors thank two anonymous referees for their careful reading of an earlier version of the paper and for many helpful comments, corrections and suggestions. TCB thanks Bob Guralnick, Martin Liebeck, Frank L\"{u}beck, Gunter Malle, Eamonn O'Brien and Pham Tiep for helpful discussions. MF thanks Marco Barbieri and Kamilla Rekvényi for their help and support, as well as the School of Mathematics at the University of Bristol for hosting a 4-month research visit in 2024.

\section{Proportions}\label{s:prop}

In this section we prove Theorem \ref{t:main1} and part (i) of Theorem \ref{t:main3}. Recall that if $G \leqs {\rm Sym}(\O)$ is a finite transitive permutation group with point stabiliser $H$, then $\Delta(G)$ denotes the set of derangements in $G$ and we write
\[
\delta(G) = \delta(G,\O) = \delta(G,H) = \frac{|\Delta(G)|}{|G|}
\]
for the proportion of derangements in $G$. For easy reference we state the Fulman-Guralnick theorem, which is proved in the sequence of papers \cite{FG1, FG2, FG3, FG4}.

\begin{thm}\label{t:FG}
There exists an absolute constant $\e > 0$ such that $\delta(G) \geqs \e$ for every finite simple transitive group $G$.
\end{thm}

We begin by proving Theorem \ref{t:main1}, handling the alternating and sporadic groups in Sections \ref{ss:prop_alt} and \ref{ss:prop_spor}, respectively. We will establish Theorem \ref{t:main3}(i) in Section \ref{ss:prim}. 

\subsection{Alternating groups}\label{ss:prop_alt}

Let $G = A_n$ be an alternating group with $n \geqs 5$ and recall that $\a(G)$ is defined to be the minimal value of $\delta(G,H)$ over all proper subgroups $H$ of $G$. Here our goal is to prove that 
\[
\lim_{n \to \infty} \a(A_n) = \frac{1}{e}.
\]

First observe that we only need to consider $\delta(G,H)$ when $H$ is a maximal subgroup of $G$. In addition, by appealing to a well known theorem of \L{}uczak and Pyber \cite{LP}, we see that the proportion of elements in $G$ that are contained in a proper transitive subgroup tends to $0$ as $n$ tends to infinity. As a consequence, we may assume $H = (S_k \times S_{n-k}) \cap G$ for some $1 \leqs k < n/2$, which allows us to identify $\O = G/H$ with the set of $k$-element subsets of $[n] = \{1, \ldots, n\}$. Let $f(n,k) = 1 - \delta(G)$ be the proportion of elements in $G$ that fix a $k$-element subset of $[n]$.

An easy application of the inclusion-exclusion principle shows that the proportion of derangements in $S_n$ with respect to the natural action on $[n]$ is given by the expression
\[
\sum_{j=0}^n\frac{(-1)^j}{j!} = \frac{[n!/e]}{n!},
\]
where $[x]$ denotes the nearest integer to $x$. As a consequence, we deduce that   
\[
f(n,1) = 1 -  \frac{[n!/e]}{n!} + \frac{(-1)^n(n-1)}{n!}
\]
for all $n \geqs 5$ (see \cite[Corollary 2.6]{Boston}). Clearly, $f(n,1) \to 1 - 1/e$ as $n \to \infty$, so Theorem \ref{t:main1}(i) is a consequence of the following result.

\begin{prop}\label{p:alt_intrans}
We have $f(n,k) \leqs f(n,1)$ for all $1 \leqs k < n/2$ and all $n \geqs 5$.
\end{prop}

In order to prove Proposition \ref{p:alt_intrans}, we need to introduce some additional notation. For integers $n \geqs 5$ and $1 \leqs k \leqs n$ we define 
\[
a(n,k) = \frac{|A(n,k)|}{|A_n|},\; b(n,k) = \frac{|B(n,k)|}{|A_n|},\; c(n,k) = \max\{ a(n,k), b(n,k) \},
\]
where $A(n,k)$ (respectively, $B(n,k)$) is the set of even (respectively, odd) elements in $S_n$ fixing a $k$-element subset of $[n]$. Note that $c(k,k) = 1$ and $c(n,k) = c(n,n-k)$. By arguing as in the proof of \cite[Lemma 2(i)]{Dixon}, it is straightforward to show that
\begin{equation}\label{e:cnk}
c(n,k) \leqs \frac{1}{n}\left(k+\sum_{j=k+1}^{n-k}c(n-j,k)\right).
\end{equation}

\begin{lem}\label{l:an_1}
We have $c(n,1) \leqs 2/3$ for all $n \geqs 5$.
\end{lem}

\begin{proof}
First observe that 
\[
a(n,1) = 1-\delta(A_n),\;\; b(n,1) = 2(1-\delta(S_n)) - a(n,1)
\]
with respect to the natural actions of $A_n$ and $S_n$ on $[n]$. By \cite[Corollary 2.6]{Boston} we have
\[
\delta(A_n) = \delta(S_n) - \frac{(-1)^n(n-1)}{n!},\;\; \delta(S_n) = \frac{[n!/e]}{n!}
\]
and it is easy to check that $c(n,1) \leqs 2/3$ for all $n \geqs 5$.
\end{proof}

\begin{lem}\label{l:an_2}
We have $c(n,2) \leqs 0.63$ for all $n \geqs 7$.
\end{lem}

\begin{proof}
First we compute $c(3,2) = 1$, $c(4,2) = 1/2$, $c(5,2) = 3/5$, $c(6,2) = 2/3$ and
\[
c(7,2) = \frac{38}{63}, \;\; c(8,2) = \frac{7}{12}, \;\; c(9,2) = \frac{3691}{6480}.
\]
In addition, we check that $c(n,2) \leqs 0.63$ for $n=10,11,12$. 
Now assume $n \geqs 13$ and suppose we have $c(m,2) \leqs 0.63$ for all $7 \leqs m \leqs n-1$. Then \eqref{e:cnk} gives
\[
c(n,2) \leqs \frac{1}{n}\left(2+\sum_{m=2}^9c(m,2) + \sum_{j=3}^{n-10}c(n-j,2)\right) \leqs \frac{1}{n}\left(\frac{341233}{45360}+\frac{63}{100}(n-12)\right)< 0.63
\]
and the result follows.
\end{proof}

\begin{lem}\label{l:an_3}
We have $c(n,k) \leqs 0.63$ for all $3 \leqs k < n/2$.
\end{lem}

\begin{proof}
First assume $k=3$, so $n \geqs 7$ and we compute 
\[
c(4,3)=\frac{3}{4}, \; c(5,3)=\frac{3}{5},\; c(6,3)=\frac{3}{8},\; c(7,3) = \frac{18}{35},\; c(8,3) = 
\frac{25}{48}.
\]
Suppose $n \geqs 9$ and $c(m,3) \leqs 0.63$ for all $5 \leqs m \leqs n-1$. Then by applying the upper bound in \eqref{e:cnk}, we deduce that 
\[
c(n,3) \leqs \frac{1}{n}\left(3+1+\frac{3}{4} + \sum_{j=4}^{n-5}c(n-j,3)\right) \leqs \frac{1}{n}\left(\frac{19}{4}+\frac{63}{100}(n-8)\right)< 0.63.
\]
The case $k=4$ is very similar.

Finally, let us assume $k \geqs 5$. By Lemmas \ref{l:an_1} and \ref{l:an_2} we have $c(k+1,k) = c(k+1,1) \leqs 2/3$ and $c(k+2,k) = c(k+2,2) \leqs 0.63$. Let us assume $c(m,k) \leqs 0.63$ for all $k+2 \leqs m \leqs n-1$. Then \eqref{e:cnk} gives
\[
c(n,k) \leqs \frac{1}{n}\left(k+1+\frac{2}{3}+\frac{63}{100}(n-2k-2)\right) = \frac{63}{100}-\frac{39k-61}{150n}<0.63
\]
and the result follows.
\end{proof}

We are now in a position to prove Proposition \ref{p:alt_intrans}. As explained above, this completes the proof of Theorem \ref{t:main1}(i).  

\begin{proof}[Proof of Proposition \ref{p:alt_intrans}]
The cases $n \in \{5,6\}$ can be checked directly, so let us assume $n \geqs 7$. Then Lemmas \ref{l:an_2} and \ref{l:an_3} imply that $f(n,k) \leqs 0.63$ for all $2 \leqs k < n/2$, while we have 
\[
f(n,1) = 1-\frac{[n!/e]}{n!}+\frac{(-1)^n(n-1)}{n!} > 0.63.
\]
The result follows.
\end{proof}

It seems difficult to prove a non-asymptotic version of Theorem \ref{t:main1}(i), but computations with the low degree alternating groups lead us to propose the following conjecture. In particular, this would imply that $13/45$ is the optimal constant for alternating groups in Theorem \ref{t:FG}. 

\begin{con}\label{c:1}
Let $G = A_n$ be a finite simple transitive group with point stabiliser $H$.

\vspace{1mm}

\begin{itemize}\addtolength{\itemsep}{0.2\baselineskip}
\item[{\rm (i)}] We have $\delta(G) \geqs 13/45$, with equality if and only if $G=A_8$ and $H = {\rm AGL}_{3}(2)$.
\item[{\rm (ii)}] If $n \geqs 9$, then 
\[
\delta(G) \geqs \frac{[n!/e]}{n!} - \frac{(-1)^n(n-1)}{n!},
\]
with equality if and only if $H = A_{n-1}$.
\end{itemize}
\end{con}

\subsection{Sporadic groups}\label{ss:prop_spor}

Next we use a computational approach to study the proportion of derangements for transitive actions of sporadic groups, working with \textsf{GAP} (version 4.13.0) and {\sc Magma} (version V2.28-8). 

Let $G \leqs {\rm Sym}(\O)$ be a finite simple transitive sporadic group with point stabiliser $H$ and let $\chi = 1^G_H$ be the corresponding permutation character, so we have 
\[
\Delta(G) = \{ x \in G \,:\, \chi(x) = 0\} = \{ x \in G \,:\, \mbox{$x^G \cap H$ is empty}\}.
\]
We can compute $\delta(G)$ if we know the sizes of the conjugacy classes in $G$, together with the \emph{fusion map} from $H$-classes to $G$-classes. The latter map describes the embedding of each $H$-class in $G$ and so we can use it to read off the conjugacy classes of derangements in $G$. As explained in the proof of Proposition \ref{p:spor1} below, in almost all cases the relevant information we need concerning conjugacy classes can be accessed via the \textsf{GAP} Character Table Library \cite{GAPCTL} and this allows us to compute $\delta(G)$ (and subsequently $\a(G)$) in a few seconds.

The following result completes the proof of Theorem \ref{t:main1} (note that in the final column of Table \ref{tab:spor}, we give $\a(G)$ to $3$ significant figures).

\begin{prop}\label{p:spor1}
Let $G$ be a finite simple transitive sporadic group with point stabiliser $H$.  

\vspace{1mm}

\begin{itemize}\addtolength{\itemsep}{0.2\baselineskip}
\item[{\rm (i)}] We have $\delta(G) \geqs \a(G)$, where $\a(G)$ is recorded in Table \ref{tab:spor}. In addition, $\delta(G) = \a(G)$ if and only if $H$ is conjugate to the maximal subgroup of $G$ listed in the second column of the table.
\item[{\rm (ii)}] We have $\delta(G) \geqs 2197/7425$, with equality if and only if $G = {\rm McL}$ and $H = 2.A_8$.
\end{itemize}
\end{prop}

{\scriptsize
\begin{table}
\[
\begin{array}{llll} \hline
G & H & \a(G) & \\ \hline
{\rm M}_{11} & {\rm M}_{10} & 23/66 & 0.348 \\
{\rm M}_{12} & {\rm M}_{11} &107/288 & 0.371 \\
{\rm M}_{22} & A_7 & 119/352 & 0.338 \\
{\rm M}_{23} & {\rm M}_{22} & 877/2415 & 0.363 \\
{\rm M}_{24} & {\rm M}_{23} & 1699/4608 & 0.368 \\
{\rm J}_{1} & D_6 \times D_{10} & 573/1463 & 0.391 \\
{\rm J}_{2} & 3.A_6.2 & 979/2100 & 0.466 \\
{\rm J}_{3} & (3 \times A_6).2 & 971/2907 & 0.334 \\
{\rm McL} & 2.A_8 & 2197/7425 & 0.295 \\
{\rm HS} & {\rm U}_{3}(5).2 & 13301/42240 & 0.314 \\
{\rm He} & 2^2.{\rm L}_{3}(4).S_3 & 23423/46648 & 0.502 \\
{\rm Ru} & 2^{1+4+6}.S_5 & 26967/65975 & 0.408 \\
{\rm Co}_{3} & {\rm McL}.2 & 47621/149040 & 0.319 \\
{\rm Co}_{2} & 2^{1+8}{:}{\rm Sp}_{6}(2) & 350303/759000 & 0.461 \\
{\rm Fi}_{22} & \O_7(3) & 934573/2365440 & 0.395 \\
{\rm Fi}_{23} &  {\rm P\O}_{8}^{+}(3).S_3 & 4624523/11561088 & 0.400 \\
{\rm Suz} & G_2(4) & 6579421/13471920 & 0.488 \\
{\rm O'N} & 4.{\rm L}_{3}(4).2 & 5402647/14286195 & 0.378 \\
{\rm HN} & 2^{1+8}.(A_5 \times A_5).2 & 13680272/24688125 & 0.554 \\ 
{\rm Th} & 2^{1+8}.A_9 & 13838827/39073671 & 0.354 \\
{\rm Ly} & 2.A_{11} & 23715556/63400425 & 0.374 \\ 
{\rm Co}_{1} & 2^{1+8}.\O_8^{+}(2) & 7948916279/15664849200 & 0.507 \\
{\rm J}_{4} & 2^{3+12}.(S_5 \times {\rm L}_{3}(2)) & 20243299027/43786049417 & 0.462 \\
{\rm Fi}_{24}' & {\rm Fi}_{23} & 765137684779/1654006894848 & 0.462 \\
\mathbb{B} & 2.{}^2E_6(2).2 & 94738750847635861/167684218416000000 & 0.564 \\
\mathbb{M} & 2^{1+24}.{\rm Co}_1 & 26707770823339783801504/49722462258718251877875 & 0.537 \\
\hline
\end{array}
\]
\caption{The values of $\a(G)$ for sporadic simple groups}
\label{tab:spor}
\end{table}
}

\begin{proof}
Part (ii) follows immediately from the information in Table \ref{tab:spor}, so we just need to consider part (i). As before, we may assume $G$ is primitive, which means that $H$ is a maximal subgroup of $G$.

To begin with, let us assume $G$ is not the Monster group $\mathbb{M}$. Then in each case, the character tables of $G$ and $H$ are available in the \textsf{GAP} Character Table Library \cite{GAPCTL}. In addition, the corresponding fusion map from $H$-classes to $G$-classes is also available, with the single exception of the case where $G = \mathbb{B}$ is the Baby Monster and $H = (2^2 \times F_4(2)){:}2$. Putting the latter case to one side for now, we can use the fusion map to determine the $G$-classes comprising $\Delta(G)$, which in turn allows us to compute $\delta(G)$ precisely. We can then read off the minimum over all maximal subgroups, which gives the value of $\a(G)$ recorded in Table \ref{tab:spor}. 

Now assume $G = \mathbb{B}$ and $H = (2^2 \times F_4(2)){:}2$. Here we use the function \textsf{PossibleClassFusions} to produce a list of $64$ candidate fusion maps and one checks that each candidate map produces the same permutation character $\chi = 1^G_H$. So once again we can calculate $\delta(G)$ precisely and the result follows.

Finally, let us assume $G = \mathbb{M}$ is the Monster. By the main theorem of \cite{DLP}, there are $46$ conjugacy classes of maximal subgroups of $G$. For $31$ of these classes, we can access the character tables of $G$ and a representative $H$ via the \textsf{GAP} function \textsf{NamesOfFusionSources}, as well as the corresponding fusion maps. So in each of these cases we can compute $\delta(G)$ precisely, just as we did above. In particular, if $\mathcal{M}_1$ denotes this specific collection of maximal subgroups, then 
\[
\min\{\delta(G,H) \,:\, H \in \mathcal{M}_1\} = \frac{26707770823339783801504}{49722462258718251877875} = \gamma,
\]
with equality if and only if $H = 2^{1+24}.{\rm Co}_1$. So in order to complete the proof, we need to show that $\delta(G,H) > \gamma$ for all of the remaining maximal subgroups $H$. To do this, let $\omega(H)$ be the \emph{spectrum} of $H$, which is simply the set of element orders in $H$, and consider the crude lower bound 
\begin{equation}\label{e:easy}
\delta(G) \geqs \frac{|\{ x \in G \,:\, |x| \not\in \omega(H)\}|}{|G|}.
\end{equation}

First assume $H$ is one of the following maximal subgroups:
\[
{\rm L}_2(13){:}2,\, {\rm L}_2(19){:}2,\, {\rm L}_2(29){:}2, \,  {\rm U}_3(4){:}4,\,  ({\rm L}_2(11) \times {\rm L}_2(11)).4, \, 11^2{:}(5 \times 2A_5), 
\]
\[
7^2{:}{\rm SL}_2(7), \, 3^8.{\rm P\O}_{8}^{-}(3).2, \, 3^{3+2+6+6}.({\rm L}_3(3) \times {\rm SD}_{16}),\, 3^{2+5+10}.({\rm M}_{11} \times 2S_4),
\]
\[
(3^2{:}2 \times {\rm P\Omega}_8^{+}(3)).S_4, \, 2^{3+6+12+18}.({\rm L}_3(2) \times 3S_6), \, 2^{2+11+22}.({\rm M}_{24} \times S_3).
\]
In each of these cases, a permutation representation of $H$ is given in the Web Atlas \cite{WebAt} and with the aid of {\sc Magma} \cite{magma} it is easy to determine $\omega(H)$. We can then verify the desired bound via \eqref{e:easy}. For example, if $H = 2^{3+6+12+18}.({\rm L}_3(2) \times 3S_6)$ then the Web Atlas provides a representation of $H$ on $1032192$ points and we can use {\sc Magma} to compute
\[
\omega(H) = \left\{ \begin{array}{c}
1, 2, 3, 4, 5, 6, 7, 8, 10, 12, 14, 15, 16, 20,21, 24, \\
  28, 30, 32, 35, 40, 42, 48, 56, 60, 70, 84, 105
 \end{array} \right\}.
\]
And then by inspecting the character table of $G$, the bound in \eqref{e:easy} yields 
\[
\delta(G) \geqs \frac{74617454008173302577265307}{105784031083359216398221125} > \gamma
\]
and the result follows.

Finally, let us assume $H = 2^{5+10+20}.(S_3 \times {\rm L}_5(2))$ or $2^{10+16}.\Omega_{10}^{+}(2)$. In these two cases, the Web Atlas does not provide a permutation representation of $H$ and so we need to modify the argument. In the first case, one can check that the trivial bound  
\begin{equation}\label{e:crude1}
\delta(G) \geqs \frac{|\{x \in G \,:\, \mbox{$|H|$ is indivisible by $|x|$}\}|}{|G|}
\end{equation}
is sufficient, so we may assume $H = 2^{10+16}.\Omega_{10}^{+}(2)$. Here the same bound yields $\delta(G) > 1/2$, but we need to work a little bit harder to force $\delta(G) > \gamma$.

To do this, we first observe that
\[
\{n \in \omega(G) \,:\, n \nmid |H|\} = \left\{ \begin{array}{c}
11, 13, 19, 22, 23, 26, 29, 33, 38, 39, 41, 44, 46, 47,52, \\
55, 57, 59, 66, 69, 71, 78, 87, 88, 92, 94, 95, 104, 110 
\end{array}
\right\} = A
\]
In addition, with the aid of {\sc Magma}, we compute
\[
\omega(\Omega_{10}^{+}(2)) = \{ 1, 2, 3, 4, 5, 6, 7, 8, 9, 10, 12, 14, 15, 17, 18, 20, 21, 24, 30, 31, 42, 45,51, 60 \}.
\]
Since $H/N = \Omega_{10}^{+}(2)$ for some normal $2$-subgroup $N$, it follows that every element in $G$ of order $93$, $105$ and $119$ is a derangement. So if we now set $B = A \cup \{93,105,119\}$, then 
\[
\delta(G) \geqs \frac{|\{ x\in G \,:\, |x| \in B\}|}{|G|} > \gamma
\]
and the proof is complete.
\end{proof}

\begin{rem}
It is easy to extend the above analysis to all almost simple sporadic groups. Let $G \leqs {\rm Sym}(\O)$ be such a group with point stabiliser $H$ and assume $G$ is non-simple. Then $\delta(G) \geqs 2516/7425$, with equality if and only if $G = {\rm McL}.2$ and $H = 2.S_8$.
\end{rem}

\subsection{Primitive groups}\label{ss:prim}

In this section we extend our analysis of $\delta(G)$ to primitive permutation groups and we prove Theorem \ref{t:main3}(i). The proof of Theorem \ref{t:main3}(ii) requires Theorem \ref{t:main2} and the details will be presented in Section \ref{ss:prim_sol}.

Let $G \leqs {\rm Sym}(\O)$ be a finite primitive permutation group with socle $N$. Recall that $N$ is the subgroup of $G$ generated by the minimal normal subgroups of $G$, and recall that the primitivity of $G$ implies that $N$ is a direct product $N = T_1 \times \cdots \times T_k$, where each $T_i$ is isomorphic to a fixed simple group $T$. The O'Nan-Scott theorem (see \cite{LPS}) describes the primitive groups in terms of the structure and action of the socle, which leads to the following division into five families (recall that a transitive subgroup of $G$ is \emph{regular} if every nontrivial element is a derangement):

\vspace{1mm}

\begin{itemize}\addtolength{\itemsep}{0.2\baselineskip}
\item[{\rm (a)}] Affine: $N = (C_p)^k$ is abelian and regular, $p$ prime, $k \geqs 1$, $G \leqs {\rm AGL}_k(p)$
\item[{\rm (b)}] Twisted wreath: $N$ is non-abelian and regular
\item[{\rm (c)}] Almost simple: $N = T$ is non-abelian simple, $G \leqs {\rm Aut}(T)$
\item[{\rm (d)}] Diagonal type: $N = T^k$ is non-abelian, $k \geqs 2$, $G \leqs N.({\rm Out}(T) \times S_k)$
\item[{\rm (e)}] Product type: $N = T^k$ is non-abelian, $k = ab$ with $a \geqs 1$, $b \geqs 2$, $G \leqs L \wr S_b$ with $L$ primitive of type (c) or (d).
\end{itemize}

\begin{proof}[Proof of Theorem \ref{t:main3}(i)]	
Let $G \leqs {\rm Sym}(\O)$ be a finite primitive permutation group with socle $N = T^k$, where $T$ is simple. The possibilities for $G$ and $N$ are described by the O'Nan-Scott theorem and we will refer to the five types briefly described above.
Our goal is to prove that $\delta(N) \geqs \e$, where $\e>0$ is the constant in Theorem \ref{t:FG}.

Of course, if $G$ is an affine group or a twisted wreath product, then $N$ is regular and thus   
\[ 
\delta(N)=\frac{\Delta(N)}{|N|}=1 - \frac{1}{|N|} \geqs \frac{1}{2}.
\]
And if $G$ is almost simple then $N$ is a simple transitive permutation group on $\O$ and thus Theorem \ref{t:FG} implies that $\delta(N) \geqs \e$. 

Next assume $G$ is a diagonal type group, so $k \geqs 2$. Let 
\[
D = \{(t, \ldots , t) \in N \,:\, t \in T \}
\]
be the diagonal subgroup of $N$ and note that we may identify $\O$ with the set $N/D$ of right cosets of $D$ in $N$. Given $x = (x_1 ,\ldots, x_k) \in N$ and $\omega = D(t_1,\ldots,t_k) \in \O$, we have $\omega^x = \omega$ if and only if $(t_1x_1t_1^{-1}, \ldots, t_kx_kt_k^{-1}) \in D$. Therefore, $x$ is a derangement if and only if at least two of the components of $x$ are not $T$-conjugate and thus
\[
\delta(N) = \frac{|N| - \left(\sum_{i=1}^m |t_i^T|^k\right)}{|N|} = 1 - \sum_{i=1}^m|C_T(t_i)|^{-k} \geqs 1 - \sum_{i=1}^m|C_T(t_i)|^{-2},
\]
where $\{t_1, \ldots, t_m\}$ is a complete set of representatives of the conjugacy classes in $T$. Since $|C_T(t_i)| \geqs 3$ for all $i$ (recall that a finite group with a self-centralising involution is soluble; see \cite[Proposition 4.2]{BBW}, for example), we conclude that 
\[
\delta(N) \geqs 1 - \frac{1}{3}\sum_{i=1}^m |C_T(t_i)|^{-1} = \frac{2}{3}.
\]

Finally, let us assume $G$ is a product type group. Here $G \leqs L \wr S_b$, where $L \leqs {\rm Sym}(\Gamma)$ is a primitive group with socle $S = T^a$, $b \geqs 2$ and $L$ is either almost simple or diagonal type. Then $N = S^b = T^{ab}$ and $G$ acts on $\O = \Gamma^b$ with the product action. In particular, if $x = (x_1, \ldots, x_b) \in N$, where each $x_i$ is contained in $S$, then 
\[
(\gamma_1, \ldots, \gamma_b)^{x} = (\gamma_1^{x_1}, \ldots, \gamma_b^{x_b})
\]
for all $\omega = (\gamma_1, \ldots, \gamma_b) \in \O$. It follows that $x$ is a derangement on $\O$ if and only if at least one $x_i$ is a derangement on $\Gamma$, whence
\begin{equation}\label{e:sharp}
\delta(N) > \frac{|\Delta(S,\Gamma)||S|^{b-1}}{|S|^b} = \delta(S,\Gamma).
\end{equation}
If $L$ is almost simple, then Theorem \ref{t:FG} implies that $\delta(N) \geqs \e$. And if $L$ is a diagonal type group, then our previous argument gives $\delta(N) \geqs 2/3$.
\end{proof}

\section{Derangement width}\label{s:width}

Let $G \leqs {\rm Sym}(\O)$ be a finite transitive permutation group, let $k$ be a positive integer and recall that 
\[
\Delta(G)^k = \{ x_1 \cdots x_k \,:\, \mbox{$x_i \in \Delta(G)$ for all $i$} \}.
\]
In \cite{LST}, Larsen, Shalev and Tiep prove that $G = \Delta(G)^2$ for all sufficiently large simple transitive groups, and they propose the following conjecture (this is stated as Conjecture \ref{c:lst00} in Section \ref{s:intro}).

\begin{con}\label{c:lst}
We have $G = \Delta(G)^2$ for every finite simple transitive group $G$.
\end{con}

This conjecture is proved for alternating groups in \cite[Theorem B]{LST}. In this section we establish a strong form of Conjecture \ref{c:lst} for all sporadic groups and all rank one groups of Lie type, and we revisit the problem for alternating groups. In particular, we prove part (i) of Theorem \ref{t:main4}. We also establish an extension of \cite[Theorem B]{LST} for symmetric groups (see Theorem \ref{t:sym}) and we prove Theorem \ref{t:main_new}, which settles Conjecture \ref{c:lst} for all the linear groups ${\rm L}_n(q)$ under a mild additional assumption on the point stabilisers (see Theorem \ref{t:psl}). At the end of the section, we present a proof of Theorem \ref{t:width_gen}.

\subsection{Sporadic groups}\label{ss:width_spor}

Here we establish a strong form of Conjecture \ref{c:lst} for sporadic groups. We will need the following result, which will also be useful later on. 

\begin{lem}\label{l:frob}
Let $G$ be a finite group with conjugacy classes $C_i = g_i^G$ for $i=1,2$. For an element $x \in G$, let $N(x)$ be the number of solutions to the equation $x = y_1y_2$ with $y_i \in C_i$. Then
\[ 
N(x)=\frac{|C_1||C_2|}{|G|}\sum_{\chi \in {\rm Irr}(G)}\frac{\chi(g_1)\chi(g_2)\overline{\chi(x)}}{\chi(1)},
\]
where ${\rm Irr}(G)$ is the set of complex irreducible characters of $G$.
\end{lem}

\begin{proof}
This is a special case of a well known formula of Frobenius, which follows from the  familiar orthogonality relations satisfied by the character table of $G$. See \cite[Theorem 30.4]{JL}, for example.
\end{proof}

\begin{rem}
Note that $G = C_1C_2$ if and only if $N(x_i) > 0$ for all $i$, where $\{x_1, \ldots, x_k\}$ is a complete set of representatives of the conjugacy classes in $G$.
\end{rem}

The following result establishes a strong form of Theorem \ref{t:main4} for sporadic groups.

\begin{prop}\label{p:spor_width}
Let $G$ be a transitive sporadic simple group with point stabiliser $H$.

\vspace{1mm}

\begin{itemize}\addtolength{\itemsep}{0.2\baselineskip}
\item[{\rm (i)}] We have $G = CD$, where $C$ and $D$ are conjugacy classes of derangements.
\item[{\rm (ii)}] Moreover, either $G = C^2$ for some class $C$ of derangements, or $(G,H)$ is one of the cases in Table \ref{tab:spor1} and 
\[
G = CD = \{1\} \cup C^2 = \{1\} \cup D^2
\]
for the classes $C$ and $D$ of derangements indicated in the table.
\end{itemize}
\end{prop} 

{\scriptsize
\begin{table}
\[
\begin{array}{llll} \hline
G & H & C & D \\ \hline
{\rm M}_{11} & S_5 & \texttt{11A} & \texttt{11B} \\ 
{\rm M}_{22} & 2^4{:}A_6, \, 2^4{:}S_5 & \texttt{11A} & \texttt{11B} \\
{\rm M}_{23} & {\rm M}_{22}, \, {\rm L}_3(4).2, \, 2^4{:}A_7, \, {\rm M}_{11}, \, 2^4{:}(3 \times A_5).2 & \texttt{23A} & \texttt{23B} \\
& 2^4{:}A_6, \, 2^4{:}S_5 \mbox{ (two classes)}, \, 2^4.(15{:}4) & & \\
{\rm M}_{24} &  {\rm M}_{12}.2 & \texttt{23A} & \texttt{23B} \\ \hline
\end{array}
\]
\caption{The pairs $(G,H)$ in Proposition \ref{p:spor_width}(ii)}\label{tab:spor1}
\end{table}
}

\begin{proof}
Let $\mathcal{C}(G)$ be the set of conjugacy classes $C$ in $G$ with $G = C^2$. Using the  \textsf{GAP} Character Table Library \cite{GAPCTL} and Lemma \ref{l:frob}, it is easy to  determine all of the conjugacy classes in $\mathcal{C}(G)$. For example, if $G = \mathbb{M}$ is the Monster, then $146$ of the $194$ conjugacy classes in $G$ are contained in $\mathcal{C}(G)$.

First assume $G \ne \mathbb{B},\mathbb{M}$. In each of these cases, we can use the \textsf{GAP} function \textsf{Maxes} to access the character table of every maximal subgroup $H$ of $G$, together with the fusion map from $H$-classes to $G$-classes. This allows us to determine all the conjugacy classes of derangements for the action of $G$ on $G/H$ and it is a routine exercise to check whether or not one of these classes is contained in $\mathcal{C}(G)$. We find that there is always at least one such class, unless $(G,H)$ is one of the following:
\begin{equation}\label{e:spec}
\begin{array}{rl}
{\rm M}_{11}{:} & H = S_5 \\
{\rm M}_{22}{:} & \mbox{$H = 2^4{:}A_6$ or $2^4{:}S_5$} \\
{\rm M}_{23}{:} & \mbox{$H = {\rm M}_{22}$, ${\rm L}_3(4).2$, $2^4{:}A_7$, ${\rm M}_{11}$ or $2^4{:}(3 \times A_5).2$} \\
{\rm M}_{24}{:} & \mbox{$H ={\rm M}_{12}.2$} 
\end{array}
\end{equation}

In each of these cases, we first check that we can always find two classes of derangements $C$ and $D$ with $G = CD$, which is easy to verify using \textsf{GAP} and Lemma \ref{l:frob}. In addition, we check that $G = \{1\} \cup C^2 = \{1\} \cup D^2$.

So to complete the proof for $G \ne \mathbb{B}, \mathbb{M}$, we just need to inspect the special cases $(G,H)$ in \eqref{e:spec} in order to determine whether or not $H$ has a proper subgroup $K$ that meets every class in $\mathcal{C}(G)$. To do this, we use {\sc Magma} to construct $G = {\rm M}_n$ and $H$ in the natural permutation representation on $n$ points and we inspect the maximal subgroups of $H$. In this way, we deduce that if $G = {\rm M}_{11}$, ${\rm M}_{22}$ or ${\rm M}_{24}$, then every proper subgroup of $H$ fails to intersect at least one of the classes in $\mathcal{C}(G)$. Now assume $G = {\rm M}_{23}$. Here $G$ has four maximal subgroups 
\[
H \in \{ {\rm M}_{22}, \, {\rm L}_3(4).2, \, 2^4{:}A_7, 2^4{:}(3 \times A_5).2\}
\]
with a maximal subgroup $K<H$ meeting every class in $\mathcal{C}(G)$. These second maximal subgroups of $G$ are recorded in the second row for $G = {\rm M}_{23}$ in Table \ref{tab:spor1}, up to conjugacy in $G$. In turn, one checks that if $J<K$ is maximal then $J \cap C$ is empty for at least one class $C \in \mathcal{C}(G)$, so no further examples arise.

Next assume $G = \mathbb{B}$ is the Baby Monster. As above, we can use the \textsf{GAP} function \textsf{Maxes} to access the character table of each maximal subgroup $H$ of $G$. And as explained in the proof of Proposition \ref{p:spor1}, we can also compute the corresponding permutation character $1^G_H$, which allows us to check that there is always at least one class of derangements in $\mathcal{C}(G)$.
		
Finally, suppose $G = \mathbb{M}$ is the Monster. Recall that $G$ has $46$ conjugacy classes of maximal subgroups (see \cite{DLP}). Now $G$ has a unique class $C$ of elements of order $41$ and our earlier computation shows that $C \in \mathcal{C}(G)$. By inspecting \cite[Table 1]{DLP}, we deduce that the only maximal subgroups intersecting $C$ are $3^8.{\rm P\O}_{8}^{-}(3).2$ and $41{:}40$. In addition, $G$ has a unique class $D$ of elements of order $19$ and we note that $D \in \mathcal{C}(G)$, so the result follows since neither $3^8.{\rm P\O}_{8}^{-}(3).2$ nor $41{:}40$ contains elements in $D$.
\end{proof}

\subsection{Alternating groups}\label{ss:width_alt}

The main result of this section is the following proposition, which establishes a stronger version of \cite[Theorem B]{LST}. In particular, this gives Theorem \ref{t:main4} for alternating groups.
	
\begin{prop}\label{p:alt_width}
Let $G = A_n$ be a simple transitive group with point stabiliser $H$.

\vspace{1mm}

\begin{itemize}\addtolength{\itemsep}{0.2\baselineskip}
\item[{\rm (i)}] There exists a conjugacy class $C$ of derangements such that $G = \{1\} \cup C^2$. 

\item[{\rm (ii)}] Moreover, either $G = C^2$, or one of the following holds:

\vspace{1mm}

\begin{itemize}\addtolength{\itemsep}{0.2\baselineskip}
\item[{\rm (a)}] $(G,H) = (A_5,A_4)$, $(A_5,S_3)$ or $(A_8, 2^4{:}(S_3\times S_3))$. 
\item[{\rm (b)}] $n \geqs 27$, $n \equiv 3 \imod{4}$ and $H = A_{n-1}$.
\end{itemize}

\item[{\rm (iii)}] If $n \geqs 6$ then $G = CD$ for classes $C,D$ of derangements.
\end{itemize}
\end{prop} 

\begin{rem}
\mbox{ }
\begin{itemize}\addtolength{\itemsep}{0.2\baselineskip}
\item[{\rm (a)}] The special cases in part (ii)(a) of Proposition \ref{p:alt_width} are genuine exceptions in the sense that $G \ne C^2$ for every class $C$ of derangements. For the two cases with $G = A_5$ we have
\[
G = \{1\} \cup CD = \{1\} \cup C^2 = \{1\} \cup D^2,
\]
where $C$ and $D$ are the two classes of elements of order $5$. And for $G = A_8$ we get 
\begin{equation}\label{e:CD}
G = CD = \{1\} \cup C^2 = \{1\} \cup D^2
\end{equation}
if we take $C$ and $D$ to be the two classes of $7$-cycles.

\item[{\rm (b)}] We expect that the case recorded in (ii)(b) is not a genuine exception. Here \cite[Theorem 1]{LT} gives $G = \{1\} \cup D^2$ where $D$ is a class of $n$-cycles, and we conjecture that $G = C^2$ for the class $C$ of elements with cycle-type $(n-4,2^2)$. We have used {\sc Magma} \cite{magma} to check the latter claim computationally for $n \in \{7,11,15,19,23\}$, which explains why we include the condition $n \geqs 27$.

\item[{\rm (c)}] For $n \geqs 11$, the proof of \cite[Theorem B]{LST} combines a technical result \cite[Proposition 7.1]{LST} with the main theorem of \cite{Bertram} to show that either $G = C^2$ for some $S_n$-class $C$ of derangements in $G$, or $n$ is even and $\O = \{1, \ldots, n\}$. In the latter case, the authors apply an inductive argument to show that $G = \Delta(G)^2$.  
\end{itemize}
\end{rem}
	
Our proof of Proposition \ref{p:alt_width} relies heavily on Proposition \ref{p:alt_cover} below, which combines results from \cite{Bertram, Brenner1, LT}. In order to state the proposition, let $\ell \leqs n$ be a positive integer and let $C_\ell$ be the $S_n$-class of the $\ell$-cycle $x=(1, \ldots, \ell)$. And if $\ell$ is odd, set $D_{\ell} = x^{A_n}$ and note that $D_{\ell} = C_{\ell}$ unless $\ell =n$ or $n-1$.

\begin{prop}\label{p:alt_cover}
Let $G = A_n$ with $n \geqs 5$.

\vspace{1mm}

\begin{itemize}\addtolength{\itemsep}{0.2\baselineskip}
\item[{\rm (i)}] If $\frac{3}{4}n \leqs \ell \leqs n$, then $G = C_{\ell}^2$.
\item[{\rm (ii)}] If $n \geqs 9$ and $n \equiv 1 \imod{4}$, then $G = D_n^2$.
\item[{\rm (iii)}] Let $D = x^{S_n}$, where $x = (1, \ldots, \ell_1)(\ell_1+1, \ldots, \ell_1+\ell_2)$ with $\ell_1,\ell_2 \geqs 2$ and $\ell_1+\ell_2 > \frac{3}{4}n+3$. Then $G = D^2$. 
\end{itemize}
\end{prop}

\begin{proof}
Part (i) is the main theorem of \cite{Bertram}, while part (ii) is a recent result of Larsen and Tiep \cite[Theorem 5(i)]{LT}. Finally, part (iii) is due to Brenner \cite{Brenner1}. 
\end{proof}
	
We are now ready to prove Proposition \ref{p:alt_width}.

\begin{proof}[Proof of Proposition \ref{p:alt_width}]
The groups with $5 \leqs n \leqs 18$ can be checked using {\sc Magma} \cite{magma}. To do this, we first construct the character table of $G$ and we use the Frobenius formula (see Lemma \ref{l:frob}) to determine the set $\mathcal{C}(G)$ of conjugacy classes $C$ with $G = C^2$. Next we construct a representative $H$ of each conjugacy class of maximal subgroups of $G$ and it is straightforward to check that every class in $\mathcal{C}(G)$ meets $H$ if and only if $(G,H)$ is one of the cases recorded in part (ii)(a) of Proposition \ref{p:alt_width}. And in each of these cases, it is routine to check that if $K<H$ is any maximal subgroup, then there exists a class $C \in \mathcal{C}(G)$ such that $C \cap K$ is empty. For the remainder, we will assume $n \geqs 19$.

Suppose $H$ acts intransitively on $[n] = \{1, \ldots, n\}$, so $H \leqs (S_{k} \times S_{n-k}) \cap G$ for some positive integer $k \leqs n/2$. If $n$ is even then by applying parts (i) and (iii) of Proposition \ref{p:alt_cover} we deduce that $G = D_{n-3}^2$ if $k \geqs 4$ and $G = D^2$ if $k \leqs 3$, where $D = x^G$ and $x = (1, \ldots, 4)(5, \ldots, n)$. If $n \equiv 1 \imod{4}$, then Proposition \ref{p:alt_cover}(ii) gives $G = D_n^2$, so we may assume $n \equiv 3 \imod{4}$. If $k \geqs 3$ then every $(n-2)$-cycle in $G$ is a derangement and we have $G = D_{n-2}^2$ by Proposition \ref{p:alt_cover}(i). And if $k=2$ then Proposition \ref{p:alt_cover}(iii) implies that $G = D^2$, where $D = x^G$ and $x = (1,\ldots,4)(5, \ldots, n-1)$. Finally, suppose $k=1$. As above, let $C = x^G$ with $x = (1, \ldots, n)$ and set $D = (x^{-1})^G \ne C$. Then \cite[Theorem 1]{LT} implies that \eqref{e:CD} holds.

Next assume $H$ is imprimitive, so $H \leqs (S_a \wr S_b) \cap G$, where $n=ab$ and $a,b \geqs 2$. If $n$ is odd then $a \geqs 3$ and every $(n-2)$-cycle is a derangement because it does not preserve a partition of $[n]$ into $b$ blocks of size $a$. The result follows since $G = D_{n-2}^2$ by Proposition \ref{p:alt_cover}(i). Now assume $n$ is even and set $\ell=5$ if $a=3$, and $\ell=3$ otherwise. Then every $(n-\ell)$-cycle is a derangement and the desired conclusion holds  since $G = D_{n-\ell}^2$ (note that if $a=3$ then $n \geqs 21$ and thus $n-5 \geqs 3n/4$).

Finally, let us assume $H$ acts primitively on $[n]$. Fix an odd integer $\ell$ such that $3n/4 \leqs \ell \leqs n-3$. Then Proposition \ref{p:alt_cover}(i) gives $G = D_{\ell}^2$ and the main theorem of \cite{Jones} implies that every $\ell$-cycle in $G$ is a derangement (indeed, no proper primitive subgroup of $G$ contains an $\ell$-cycle). The result follows. 
\end{proof}

\subsection{Symmetric groups}\label{ss:sym}

Here we apply work of Brenner \cite{Brenner1} to prove a strong form of  \cite[Theorem B]{LST} for symmetric groups.

\begin{thm}\label{t:sym}
Let $G = S_n$ be a transitive group on a set $\O$ with point stabiliser $H$, where $n \geqs 4$. 

\vspace{1mm}

\begin{itemize}\addtolength{\itemsep}{0.2\baselineskip}
\item[{\rm (i)}] We have $G = \Delta(G)^2$. 
\item[{\rm (ii)}] Moreover, if $\O \ne \{1, \ldots, n\}$ then there exist conjugacy classes $C$ and $D$ of derangements such that $G = C^2 \cup CD$.
\end{itemize}
\end{thm}

Set $L = A_n$ and observe that if $C$ and $D$ are conjugacy classes in $G = S_n$, then $C^2 \subseteq L$ and either $CD \subseteq L$ or $CD \subseteq G \setminus L$, according to whether or not the elements comprising $C$ and $D$ have the same parity. In particular, $G \ne \{1\} \cup C^2$ and $G \ne \{1\} \cup CD$, so the conclusion $G = C^2 \cup CD$ in part (ii) of Theorem \ref{t:sym} is essentially best possible. Clearly, in order to prove Theorem \ref{t:sym} we may assume that $H$ is maximal in $L$ or $G$. (Note that $G$ may contain a non-maximal subgroup $H$ that is not contained in a core-free maximal subgroup of $G$, such as $H = {\rm AGL}_3(2)$ in $G = S_8$. But in this situation, $H$ is always maximal in $L$.)

\begin{rem}
We expect that the conclusion $G = C^2 \cup CD$ for classes $C,D$ of derangements is still valid when $\O = \{1, \ldots, n\}$. Indeed, we conjecture that we can take $C = x^G$ and $D = y^G$, where $x = (1, \ldots,n)$ and $y = (1,2)(3, \ldots, n)$. We have used {\sc Magma} to check this for $4 \leqs n \leqs 20$.
\end{rem}

It will be convenient to handle the low degree groups computationally.

\begin{lem}\label{l:sym_small}
The conclusion to Theorem \ref{t:sym} holds for $n \leqs 15$.
\end{lem}

\begin{proof}
This is a straightforward computation. First we use {\sc Magma} \cite{magma} to construct the character table of $G$ and we then apply Lemma \ref{l:frob} to determine the set $\mathcal{C}'(G)$ of pairs of classes $(C,D)$ such that $G = C^2 \cup CD$. We then construct a representative $H$ of each conjugacy class of core-free maximal subgroups in $A_n$ and $G$, and in each case it is easy to check that there is at least one pair $(C,D) \in \mathcal{C}'(G)$ such that $C \cap H$ and $D \cap H$ are both empty (including the special case $\O = \{1, \ldots, n\}$). 
\end{proof}

For the remainder, we may assume $n \geqs 16$. The main ingredient in the proof of Theorem \ref{t:sym} is the following. The statement is a combination of special cases of two results due to Brenner \cite{Brenner1}.

\begin{prop}\label{p:brenner}
Let $G = S_n$ with $n \geqs 16$ and set $C = x^G$ and $D=y^G$, where

\vspace{1mm}

\begin{itemize}\addtolength{\itemsep}{0.2\baselineskip}
\item[{\rm (i)}] $x = (1, \ldots, \ell)$ and $y = (1,\ldots, \ell+1)$, with $n-4 \leqs \ell < n$; or

\item[{\rm (ii)}] $x = (1,2,3)(4,5, \ldots, n-1)$ and $y = (1,2,3)(4,5,\ldots, n)$.
\end{itemize}
Then $G = C^2 \cup CD$.
\end{prop}

\begin{proof}
Suppose $x \in G$ has $n-t$ fixed points and $c$ nontrivial orbits on $\{1, \ldots, n\}$. Then $t+c \leqs 3n/2$ and the condition $n \geqs 16$ allows us to appeal to Theorems 2.02 and 4.02 in \cite{Brenner1}, which immediately give $L = C^2$ and $G \setminus L = CD$ in both (i) and (ii).
\end{proof}

We will also need the following elementary lemma.

\begin{lem}\label{l:sym}
Let $G = S_n$ with $n \geqs 4$. Then for all $x \in G$, there exists an $n$-cycle $y \in G$ such that $xy$ has no fixed points on $\{1, \ldots, n\}$.
\end{lem}

\begin{proof}
We may assume $x$ is nontrivial. Since the set of derangements on $\{1, \ldots, n\}$ is a normal subset of $G$, it suffices to show that some conjugate of $x$ has the desired property. By considering the disjoint cycle decomposition of $x$, it is easy to see that there exists a conjugate $z$ of $x$ such that $n^{z} \ne 1$ and $i^{z} \ne i+1$ for all $i \in \{1, \ldots, n-1\}$. Then $z(n, \ldots, 1)$ has no fixed points on $\{1, \ldots, n\}$ and the result follows. 
\end{proof}

We are now in a position to prove Theorem \ref{t:sym}.

\begin{proof}[Proof of Theorem \ref{t:sym}]
In view of Lemma \ref{l:sym_small}, we may assume $n \geqs 16$. Also recall that we may assume $H$ is maximal in $L = A_n$ or $G = S_n$. Set $[n] = \{1, \ldots, n\}$.

First assume $H$ is intransitive on $[n]$, which means that we may assume $H = S_k \times  S_{n-k}$ with $1 \leqs k < n/2$ and we identify $\O$ with the set of $k$-element subsets of $[n]$. If $k=1$ then Lemma \ref{l:sym} implies that $G = \Delta(G)C$, where $C$ is the class of $n$-cycles, and the result follows since every $n$-cycle is a derangement. And for $k \geqs 2$ we can define $C$ and $D$ as in part (i) of Proposition \ref{p:brenner} with $\ell=n-1$. Then the elements in $C$ and $D$ are derangements, and the proposition gives $G = C^2 \cup CD$, as required.

Next suppose $H$ is transitive and imprimitive on $[n]$. Here we may assume $H = S_a \wr S_b$ for integers $a$ and $b$ with $n = ab$ and $a \geqs 2$. If $a \geqs 3$ then every $(n-2)$-cycle and every $(n-1)$-cycle is a derangement, and Proposition \ref{p:brenner} yields $G = C^2 \cup CD$, where $C$ and $D$ are defined as in part (i) of the proposition, with $\ell = n-2$. And if $a=2$ then the same conclusion holds if we define $C$ and $D$ as in part (ii) of Proposition \ref{p:brenner}.

Finally, let us assume $H$ acts primitively on $[n]$. By the main theorem of \cite{Jones}, every $(n-4)$-cycle and every $(n-3)$-cycle is a derangement. So the result follows by defining $C$ and $D$ as in part (i) of Proposition \ref{p:brenner} with $\ell=n-4$.
\end{proof}

\subsection{Groups of Lie type}\label{ss:width_lie}

In this section we consider the derangement width of simple groups of Lie type. We begin by stating our main result for rank one groups (see \eqref{e:rank1}), which completes the proof of Theorem \ref{t:main4}(i). Recall  that ${}^2G_2(3)'$ is isomorphic to ${\rm L}_2(8)$, so we assume $q \geqs 27$ when considering $G = {}^2G_2(q)$ in Theorem \ref{t:rankone}, and also in Proposition \ref{p:rankone} below.

\begin{thm}\label{t:rankone}
Let $G$ be a finite transitive simple rank one group of Lie type with point stabiliser $H$. Then there exist conjugacy classes $C$ and $D$ of derangements such that
\[
G = \left\{ \begin{array}{ll}
C^2 & \mbox{if $G = {}^2B_2(q)$ or ${}^2G_2(q)$} \\
\{1\} \cup C^2 & \mbox{if $G = {\rm U}_3(q)$} \\
\{1\} \cup CD & \mbox{if $G = {\rm L}_2(q)$ and $(G,H) \ne ({\rm L}_2(7),S_4)$} \\
C^2 \cup CD & \mbox{if $(G,H) = ({\rm L}_2(7),S_4)$.} 
\end{array}\right.
\]
In particular, we have $G = \Delta(G)^2$.
\end{thm}

The main ingredient in the proof of Theorem \ref{t:rankone} is the following result. Note that ${\rm L}_2(5) \cong {\rm L}_2(4)$, so we assume $q \geqs 7$ if $G = {\rm L}_2(q)$ with $q$ odd.

\begin{prop}\label{p:rankone}
Let $G$ be a rank one finite simple group of Lie type and set $C = x^G$ and $D = y^G$, where the order of $x$ and $y$ is given in Table \ref{tab:rank1}. 

\vspace{1mm}

\begin{itemize}\addtolength{\itemsep}{0.2\baselineskip}
\item[{\rm (i)}] If $G \in \{ {}^2B_2(q), {}^2G_2(q), {\rm L}_2(q) \, \mbox{{\rm ($q \geqs 7$ odd)}}\}$, then $G = C^2 = D^2$.
\item[{\rm (ii)}] If $G = {\rm U}_3(q)$ then $G = \{1\} \cup C^2$.
\item[{\rm (iii)}] If $G = {\rm L}_2(q)$ with $q$ even, then $G = C^2 = \{1\} \cup AB$, where $A = z^G$, $B = (z^2)^G$ and $|z| = q+1$.
\item[{\rm (iv)}] If $G = {\rm L}_2(7)$ then $G = A^2 \cup AB = B^2 \cup AB$, where $A$ and $B$ are the two classes of elements of order $7$.
\end{itemize}
\end{prop}

{\scriptsize
\begin{table}
\[
\begin{array}{llll} \hline
G & \mbox{Conditions} & |x| & |y| \\ \hline
{}^2B_2(q) & q \geqs 8 & q+\sqrt{2q}+1 & q-\sqrt{2q}+1 \\
{}^2G_2(q) & q \geqs 27 & q+\sqrt{3q}+1 & q-\sqrt{3q}+1 \\
{\rm U}_3(q) & q \geqs 3 & (q^2-q+1)/(3,q+1) & (q^2-1)/(3,q+1) \\
{\rm L}_2(q) & \mbox{$q \geqs 5$ odd} & (q+1)/2 & (q-1)/2 \\
& \mbox{$q \geqs 4$ even} & q-1 &  \\ \hline
\end{array}
\]
\caption{The classes $C = x^G$ and $D = y^G$ in Proposition \ref{p:rankone}}\label{tab:rank1}
\end{table}
}

\begin{proof}
Parts (i) and (ii) follow immediately from \cite[Theorem 7.1]{GM}, noting that every semisimple conjugacy class in ${}^2B_2(q)$ and ${}^2G_2(q)$ is real (for example, this follows from Lemma \ref{l:real} in Section \ref{ss:exccc} below). Now assume $G = {\rm L}_2(q)$ with $q$ even. Here \cite[Theorem 7.1]{GM} gives $G = C^2$, so it just remains to show that $G = \{1\} \cup AB$, where $A$ and $B$ are the conjugacy classes of $z$ and $z^2$, respectively, and we have $|z| = |z^2| = q+1$.

Given a nontrivial element $g \in G$, let $N(g)$ be the number of solutions to the equation $g = ab$ with $a \in A$ and $b \in B$ and recall Lemma \ref{l:frob}, which states that
\[
N(g) = \frac{|A||B|}{|G|}\sum_{\chi \in {\rm Irr}(G)}\frac{\chi(z)\chi(z^2)\overline{\chi(g)}}{\chi(1)}
\]
where ${\rm Irr}(G)$ is the set of complex irreducible characters of $G$. By inspecting the character table of $G$ (see \cite[Section 38]{Dorn}, for example), it is a straightforward exercise to show that
\[
N(g) = \left\{ \begin{array}{ll}
q & \mbox{if $|g| = 2$} \\
q-1 & \mbox{if $|g|$ divides $q-1$} \\
\mbox{$q+1$ or $1$} & \mbox{if $|g|$ divides $q+1$} 
\end{array}\right.
\]
and we conclude that $|N(g)| \geqs 1$ for all $1 \ne g \in G$ (this can also be checked computationally, using the \textsf{GAP} package \textsf{Chevie} \cite{Chevie}). Therefore, $G = \{1\} \cup AB$ as claimed.
\end{proof} 

\begin{rem}\label{r:garion}
For $G = {\rm L}_2(q)$, Garion \cite{Garion} gives a complete classification of the conjugacy classes $C$ with $G = C^2$. In particular, \cite[Theorem 2(ii)]{Garion} shows that if $q$ is even and $C = x^G$ with $|x|$ dividing $q+1$, then $C^2 = \{ y \in G \,:\, |y| \ne 2\}$.
\end{rem}

\begin{proof}[Proof of Theorem \ref{t:rankone}]
As usual, we may assume $H$ is a maximal subgroup. For now let us assume $G \ne {\rm L}_2(q)$ if $q$ is even (we will handle this special case at the end of the proof) and define the classes $C = x^G$ and $D = y^G$ as in Table \ref{tab:rank1}. In view of Proposition \ref{p:rankone}, it suffices to show that either $x$ or $y$ is a derangement.

Suppose $G = {}^2B_2(q)$. Here $N_G(\la x \ra)$ is the unique maximal overgroup of $x$, so $x$ is a derangement unless $H = (q+\sqrt{2q}+1).4$. And in the latter case, $y$ is a derangement since $|H|$ is indivisible by $|y|$ and the result follows. An entirely similar argument handles the groups $G = {}^2G_2(q)$ with $q \geqs 27$.

Next assume $G = {\rm U}_3(q)$ with $q \geqs 3$, and note that $1 \not\in C^2 \cup D^2$ since neither $C$ nor $D$ are real classes. Suppose $q \not\in \{3,5\}$. Since $|x| = (q^2-q+1)/(3,q+1)$, it follows that $x$ is a Singer cycle and the main theorem of \cite{Ber} implies that $x$ is a derangement unless $H$ is a field extension subgroup of type ${\rm GU}_1(q^3)$. 
But in the latter case, $y$ is a derangement and once again the result follows. Finally, if $q=3$ then $x$ is a derangement unless $H = {\rm L}_2(7)$, in which case $y$ is a derangement. Similarly, if $q=5$ then $x$ is a derangement unless $H$ is isomorphic to $A_7$ (there are three conjugacy classes of such subgroups), and in the latter situation it is clear that $y$ is a derangement.

Now suppose $G = {\rm L}_2(q)$ with $q \geqs 7$ odd. If $q \geqs 11$, then \cite{Ber} implies that $x$ is a derangement unless $H = D_{q+1}$ is a field extension subgroup of type ${\rm GL}_1(q^2)$, in which case $y$ is a derangement. If $q=9$ then $x$ is a derangement unless $H = A_5$ (there are two classes of such subgroups). But in the latter case, $H$ does not contain any elements of order $|y|=4$, so $y$ is a derangement. Now assume $q=7$, so $|x| = 4$, $|y|=3$ and $C,D$ are the unique conjugacy classes in $G$ containing elements of order $4,3$, respectively. If $H = 7{:}3$ then $x$ is a derangement. However, if $H = S_4$ (there are two such classes), then $x$ and $y$ both have fixed points. Indeed, $\Delta(G) = A \cup B$ is the union of the two conjugacy classes of elements of order $7$. It is easy to check that 
\[
A^2 = B^2 = \{ z \in G \,:\, |z| \ne 1,4\},\;\; AB = \{ z \in G \,:\, |z| \ne 2\}
\]
and thus $G$ is not equal to $\{1\} \cup A^2$, $\{1\} \cup B^2$ nor $\{1\} \cup AB$, but we do get $G = A^2 \cup AB = B^2 \cup AB$ as required. Moreover, if we consider the action of $G$ on $G/K$, where $K$ is a proper subgroup of $H$, then we find that either $x$ or $y$ is a derangement, so we have $G = E^2$ for some conjugacy class $E$ of derangements.

To complete the proof, we may assume $G = {\rm L}_2(q)$ with $q \geqs 4$ even. Set $C = x^G$ and $A = z^G$, $B = (z^2)^G$, where $|x| = q-1$ and $|z|=q+1$, so  Proposition \ref{p:rankone}(iii) implies that $G = C^2 = \{1\} \cup AB$. By \cite{Ber}, $z$ and $z^2$ are derangements unless $H = D_{2(q+1)}$ is a field extension subgroup of type ${\rm GL}_1(q^2)$, and in the latter case we see that $x$ is a derangement. The result follows.
\end{proof}

To conclude this section, we briefly consider Conjecture \ref{c:lst} for groups of Lie type of rank two or more, focussing on the linear groups ${\rm L}_n(q)$ with $n \geqs 3$. Note that ${\rm L}_3(2) \cong {\rm L}_2(7)$, so we may assume $q \geqs 3$ in part (i) of Theorem \ref{t:psl} below. In part (ii), we write $P_k$ for the stabiliser of a $k$-dimensional subspace of the natural module for $G$.

The following result establishes Theorem \ref{t:main_new}, as stated in Section \ref{s:intro}.

\begin{thm}\label{t:psl}
Let $G = {\rm L}_n(q)$ be a transitive simple group on a set $\O$ with point stabiliser $H$ and assume $n \geqs 3$.

\vspace{1mm}

\begin{itemize}\addtolength{\itemsep}{0.2\baselineskip}
\item[{\rm (i)}] If $n=3$ and $q \geqs 3$, then $G = \{1\} \cup C^2$ for some conjugacy class $C$ of derangements.

\item[{\rm (ii)}] If $n \geqs 4$ and $H \not\leqs P_1,P_{n-1}$, then $G = \Delta(G)^2$. Moreover, if $q \geqs 3$ or $H \not\leqs {\rm Sp}_n(2)$, then 
$G = \{1\} \cup CD$ for classes $C,D$ of derangements.
\end{itemize}
\end{thm}

\begin{rem}\label{r:psl}
Suppose $n \geqs 4$ and $H$ is contained in a $P_1$ or $P_{n-1}$ parabolic subgroup. If $n \geqs 33$, then the two classes $C,D$ in \cite[Theorem 2.4(i)]{LST} are derangements and we have $G = \{1\} \cup CD$. And the same conclusion holds if $n \geqs 7$ and $q> 7^{481}$ by \cite[Theorem 2.4(ii)]{LST}. So the open cases here are when $4 \leqs n \leqs 6$ (for all $q$ at most some unspecified constant) and when $7 \leqs n \leqs 32$ with $q \leqs 7^{481}$.
\end{rem}

We will need the following lemma in the case where $G = {\rm L}_n(2) = {\rm GL}_n(2)$ and $n \geqs 30$. For an integer $2 \leqs k < n/2$, let $z_k \in G$ be a regular semisimple element of the form $z_k = {\rm diag}(A,B)$, where $A \in {\rm GL}_k(2)$ and $B \in {\rm GL}_{n-k}(2)$ are Singer cycles, so we have
\begin{equation}\label{e:zk}
|z_k| = {\rm lcm}(2^k-1, 2^{n-k}-1).
\end{equation}

\begin{lem}\label{l:lst_new}
Let $G = {\rm L}_n(2)$ with $n \geqs 30$ and let $x,y \in G$ be regular semisimple elements of the form $x = z_2$ and $y = z_3$. Then $G = \{1\} \cup CD$, where $C = x^G$ and $D = y^G$.
\end{lem}

\begin{proof}
This follows from \cite[Theorem 2.4(i)]{LST} if $n \geqs 33$. By slightly tweaking some of the estimates in the proof of this result, we will show that the same argument also works for  $n = 30,31$ and $32$.

In view of Lemma \ref{l:frob}, we need  
\[
\left| \sum_{\chi \in {\rm Irr}(G)} \frac{\chi(x)\chi(y)\overline{\chi(g)}}{\chi(1)}\right| > 0 
\]
for each non-identity element $g \in G$. As explained in the proof of \cite[Theorem 2.4(i)]{LST}, it suffices to show that
\[
\sum_{i=2}^8 \frac{|\chi_i(g)|}{\chi_i(1)} < 1,
\]
where the $\chi_i$ are the non-principal unipotent characters of $G$ listed in \cite[(2.3)]{LST}. By repeating the argument in the proof of \cite[Theorem 2.4(i)]{LST}, setting $q=2$ and applying \cite[Lemma 2.3]{LST} and \cite[Theorem 1.6(i)]{GLT}, we get
\[
\sum_{i=2}^4 \frac{|\chi_i(g)|}{\chi_i(1)} < \frac{2^{n-1}+4}{2^n-2} + 0.1254 + \frac{1.76}{2^{(4n-15)/n}} < 0.781
\]
for all $n \geqs 30$. As in \cite[(2.8)]{LST}, the combined contribution from the four remaining unipotent characters is at most 
\[
4\cdot 2^{(21-n)/2} < 0.177,
\]
so 
\[
\sum_{i=2}^8 \frac{|\chi_i(g)|}{\chi_i(1)} < 0.781 + 0.177 < 1
\]
and the result follows.
\end{proof}

\vs

First we handle the case $n=3$.

\begin{lem}\label{l:n3}
The conclusion to Theorem \ref{t:psl} holds for $n=3$.
\end{lem}

\begin{proof}
Here $G = {\rm L}_3(q)$ with $q \geqs 3$ and we set $d = (3,q-1)$. By \cite[Theorem 7.3]{GM} we have $G = \{1\} \cup C^2 = \{1\} \cup D^2$, where $C = x^G$ and $D = y^G$ with $|x| = (q^2+q+1)/d$ and $|y| = (q^2-1)/d$. Note that 
$x$ is a Singer cycle. 

For $q \ne 4$, the main theorem of \cite{Ber} implies that $x$ is a derangement unless $H$ is a field extension subgroup of type ${\rm GL}_1(q^3)$. But in the latter case, $y$ is clearly a derangement and the result follows. Finally, if $q=4$ then $|x| = 7$, $|y| = 5$ and one checks that no maximal subgroup of $G$ has order divisible by $35$, so in each case either $x$ or $y$ is a derangement.
\end{proof}

Now assume $n \geqs 4$. The following result is part of \cite[Theorem 2.1]{MSW}.

\begin{prop}\label{p:msw}
Let $G = {\rm L}_n(q)$ with $n \geqs 4$ and set $d = (n,q-1)$ and $e=(2,q-1)$. Then $G = \{1\} \cup CD$, where $C = x^G$, $D = y^G$ and
\[
|x| = \left\{ \begin{array}{ll}
(q^{n/2}+1)/e & \mbox{if $n$ is even} \\
(q^n-1)/d(q-1) & \mbox{if $n$ is odd,}
\end{array}\right.
\;\;\;
|y| = \left\{ \begin{array}{ll}
(q^{n-1}-1)/d & \mbox{if $n$ is even} \\
q^{(n-1)/2}+1 & \mbox{if $n$ is odd.}
\end{array}\right.
\]
\end{prop}

Thompson's conjecture for the linear groups ${\rm L}_n(q)$ was first proved by Lev in \cite{Lev}, where he established several additional results on conjugacy classes and their products. We will use Proposition \ref{p:lev} below, which is a special case of \cite[Theorem 1]{Lev}. In order to state this result, we need to introduce some notation.

Suppose $n \geqs 3$ and $n_1, \ldots, n_k$ are positive integers such that $n = 2+\sum_{i}n_i$. For a positive integer $m$ and scalar $\l \in \mathbb{F}_q^{\times}$, let
\[
J_m(\l) = \left( \begin{array}{ccccc}
\l & 1 & & &  \\
& \l & 1 & &  \\
& & \ddots & \ddots &  \\
& & &  \l & 1 \\
& & & & \l
\end{array}\right) \in {\rm GL}_m(q)
\]
be a standard $m \times m$ Jordan block with eigenvalue $\l$ (for $\l = 1$ we will often write $J_m$, rather than $J_m(1)$). Then for an irreducible matrix $A \in {\rm GL}_2(q)$, let  
\[
x = (A, J_{n_1}(\l_1), \ldots, J_{n_k}(\l_k)) \in {\rm PGL}_n(q)
\]
denote the image of the block-diagonal matrix ${\rm diag}(A, J_{n_1}(\l_1), \ldots, J_{n_k}(\l_k)) \in {\rm GL}_n(q)$. Note that $x \in {\rm L}_n(q)$ if $\prod_i \l_i^{n_i} = \det(A)^{-1}$.

\begin{prop}\label{p:lev}
Let $G = {\rm L}_n(q)$ with $n \geqs 3$ and set $C = x^G$, where $x \in G$ is defined
\[
x = \left\{ \begin{array}{ll}
(J_n) & \mbox{if $q=2$} \\
(A, J_{n_1}(\l_1), \ldots, J_{n_k}(\l_k)) & \mbox{if $q \geqs 3$,}
\end{array}\right.
\]
and the $\l_i$ are distinct elements of $\mathbb{F}_{q}^{\times}$. Then $G = C^2$.
\end{prop}

\begin{proof}
If $q \geqs 3$ then \cite[Theorem 1]{Lev} implies that $G = CC^{-1}$ and the result follows since $C = C^{-1}$ by \cite[Theorem 1]{Newman} (recall that $A \in {\rm GL}_2(q)$ is irreducible). Similarly, if $q=2$ then $C$ is the unique conjugacy class of cyclic matrices in $G$ all of whose eigenvalues are contained in $\mathbb{F}_2$ and once again the result follows from \cite[Theorem 1]{Lev} (also see \cite[p.1243]{Lev}).
\end{proof}

The following elementary observation will also be useful.

\begin{lem}\label{l:easy2}
Let $G$ be a finite permutation group with $\delta(G) > 1/2$. Then $G = \Delta(G)^2$.
\end{lem}

\begin{proof}
If $x \not\in \Delta(G)^2$, then $x\Delta(G)^{-1} \cap \Delta(G)$ is empty and thus $|\Delta(G)| \leqs |G|/2$, which is a contradiction.
\end{proof}

We are now ready to complete the proof of Theorem \ref{t:psl} (recall that this is stated as Theorem \ref{t:main_new} in Section \ref{s:intro}).

\begin{proof}[Proof of Theorem \ref{t:psl}]
In view of Lemma \ref{l:n3}, we may assume $n \geqs 4$. In addition, we may assume $H \ne P_1,P_{n-1}$ is a maximal subgroup of $G$. The groups ${\rm L}_6(2)$ and ${\rm L}_7(2)$ can be handled using {\sc Magma} and in both cases we check that $G = C^2$ for some conjugacy class $C$ of derangements. So for the remainder, we may assume $(n,q) \ne (6,2), (7,2)$. 

Write $q=p^f$, where $p$ is a prime. Recall that if $e \geqs 2$ is an integer, then 
a prime divisor $r$ of $q^e-1$ is a \emph{primitive prime divisor} if $q^i-1$ is indivisible by $r$ for all $i<e$. By a theorem of Zsigmondy \cite{Zsig}, such a divisor exists unless $(e,q) = (6,2)$, or if $e=2$ and $q$ is a Mersenne prime. Define the conjugacy classes $C = x^G$ and $D = y^G$ as in Proposition \ref{p:msw} and set 
\[
m = \min\{|x|,|y|\}.
\]
Since $(n,q) \ne (6,2), (7,2)$ it follows that $|x|$ and $|y|$ are divisible by primitive prime divisors of $p^{fn}-1$ and $p^{f(n-1)}-1$, respectively, so the maximal overgroups of $x$ and $y$ are described in \cite{GPPS} and we can work through the possibilities arising in \cite[Examples 2.1-2.9]{GPPS}. In doing so, it will be convenient to adopt the standard notation $\C_1 \cup \cdots \cup \C_8 \cup \mathcal{S}$ for the nine collections of maximal subgroups of $G$ arising in Aschbacher's subgroup structure theorem \cite{Asch}, appealing to \cite{KL} for a detailed description of the subgroups arising in each collection. We will often refer to the \emph{type} of $H$, which gives a rough description of its structure (our usage is consistent with \cite[p.58]{KL}). Let $V$ be the natural module for $G$.

First assume $H \in \mathcal{C}_1$, so $H = P_k$ is a maximal parabolic subgroup with $2 \leqs k \leqs n-2$ and we may identify $\O$ with the set of $k$-dimensional subspaces of $V$. Since $x$ acts irreducibly on $V$ we have $x \in \Delta(G)$. Similarly, $y$ fixes a decomposition $V = U \oplus W$, where $U$ is $1$-dimensional and $y$ acts irreducibly on $W$. This means that $U$ and $W$ are the only proper nonzero subspaces of $V$ fixed by $y$, so $y$ is also a derangement on $\O$ and we conclude by applying Proposition \ref{p:msw}.

Next suppose $H \in \mathcal{C}_2$. By \cite{GPPS}, either $x$ and $y$ are both derangements, or $H$ is of type ${\rm GL}_1(q) \wr S_n$ with $q \geqs 5$ and $(n,q) \ne (4,5)$ (see \cite[Example 2.3]{GPPS}, \cite[Table 3.5.A]{KL} and \cite[Table 8.8]{BHR}). Since $x$ and $y$ have at most $2$ composition factors on $V$, containment in (a conjugate of) $H$ implies that $m \leqs n^2(q-1)/4$ (see \cite[Remark 5.1(i)]{BGK}) but it is easy to check that $m > n^2(q-1)/4$, so $x$ and $y$ are derangements as required. 

Now assume $H \in \C_5$ is a subfield subgroup of type ${\rm GL}_n(q_0)$, where $q=q_0^k$ with $k$ a prime divisor of $f$. We have $H \leqs {\rm PGL}_n(q_0)$ and it is easy to see that $|{\rm PGL}_n(q_0)|$ is not divisible by a primitive prime divisor of $p^{f(n-\a)}-1$ for $\a=0$ or $1$, whence $x$ and $y$ are derangements and the result follows.

Next assume $H \in \C_6$. Then by inspecting \cite[Example 2.5]{GPPS} and \cite[Proposition 4.6.6]{KL} we deduce that $y$ is a derangement, and that $x$ has a fixed point only if $n = 2^k$ and $q =p \equiv 1 \imod{4}$ for some $k \geqs 2$. Since $n/2$ is even, we note that $|x| = (q^{n/2}+1)/(2,q-1)$ is odd. But $H \leqs 2^{2k}.{\rm Sp}_{2k}(2)$ and $|{\rm Sp}_{2k}(2)|$ is indivisible by $|x|$, so $x$ must also be a derangement and this completes the argument in this case.

By the main theorem of \cite{GPPS}, both $x$ and $y$ are derangements if $H$ is one of the tensor product subgroups in $\C_4 \cup \C_7$, so in order to complete the proof we may assume 
\[
H \in \C_3 \cup \C_8 \cup \mathcal{S}.
\]

Suppose $H \in \C_3$, in which case $H$ is a field extension subgroup of type ${\rm GL}_{n/k}(q^k)$ for some prime divisor $k$ of $n$. First assume $q \geqs 3$ and set $E = z^G$, where $z = (A,J_{n-2})$ and $A \in {\rm SL}_2(q)$ has order $q+1$. Then $G = E^2$ by Proposition \ref{p:lev} and we note that $z^{q+1}$ has Jordan form $(J_{n-2},J_1^2)$. But there are no such elements in $H$, so $z^{q+1}$, and hence $z$ itself, is a derangement and the result follows. Similarly, if $q=2$ then $G = E^2$ for the class $E$ of regular unipotent elements (see Proposition \ref{p:lev}) and once again the result follows since $E \cap H$ is empty.

Next assume $H \in \C_8$. There are several cases to work through:

\vspace{1mm}

\begin{itemize}\addtolength{\itemsep}{0.2\baselineskip}
\item[{\rm (a)}] $H$ is of type ${\rm U}_n(q_0)$ with $q = q_0^2$;
\item[{\rm (b)}] $H$ is of type ${\rm O}_n^{\e}(q)$ with $q$ odd;
\item[{\rm (c)}] $H$ is of type ${\rm Sp}_n(q)$ with $n$ even and $q \geqs 2$.
\end{itemize}

\vspace{1mm}

First consider case (a), where we have $H \leqs {\rm PGU}_n(q_0)$. We claim that $x$ and $y$ are derangements, where $x$ and $y$ are defined as in Proposition \ref{p:msw}. To see this, first observe that every semisimple element in $H$ has order at most $\a = q^{(n-1)/2} +(-1)^n$ by \cite[Lemma 2.15]{GMPS}. If $n$ is odd, then $m > \a$ and so the claim holds. Now assume $n$ is even, in which case $|y|>\a$ and thus $y$ is a derangement. And since $|x|$ is divisible by a primitive prime divisor of $p^{fn}-1$, but $|H|$ is not, we conclude that $x$ is also a derangement and the claim follows.

Finally, let us consider cases (b) and (c). Suppose $q=2$ and $H = {\rm Sp}_n(2)$, so $n \geqs 8$ is even. If $n \geqs 30$ then Lemma \ref{l:lst_new} gives $G = \{1\} \cup CD$, where $C$ and $D$ are both conjugacy classes of elements of order at least $2^{n-3}-1$ (see \eqref{e:zk}). And since $|z| \leqs 2^{n/2+1}$ for all $z \in H$ (see \cite[Theorem 2.16]{GMPS}), we deduce that $C$ and $D$ comprise derangements and the result follows. For $8 \leqs n \leqs 28$ we can use {\sc Magma} to determine all the conjugacy classes in $G$ and $H$, which allows us to show that 
\[
\delta(G) \geqs \frac{|\{ x \in G \,:\, \mbox{$|x| \not\in \omega(H)$}\}|}{|G|} > \frac{1}{2},
\]
where $\omega(H) = \{ |x| \,:\, x \in H\}$ is the spectrum of $H$. Therefore, $G = \Delta(G)^2$ by Lemma \ref{l:easy2}.

Now assume $q \geqs 3$ and write $\mathbb{F}_q^{\times} = \la \l \ra$. Consider the conjugacy class $E = z^G$, where $z \in G$ is defined as follows:
\[
z = \left\{\begin{array}{ll}
(A,J_1(\l^{-1}),J_{n-3}(1)) & \mbox{if $H$ is of type ${\rm O}_n(q)$ or ${\rm Sp}_n(q)$} \\
(B,J_{n-2}(1)) & \mbox{if $H$ is of type ${\rm O}_n^{\pm}(q)$.}
\end{array}\right.
\]
Here $A,B \in {\rm GL}_2(q)$ are irreducible, with $\det(A) = \l$ and $\det(B)=1$. By Proposition \ref{p:lev} we have $G = E^2$ and we claim that $z$ is a derangement.

To see this, first assume $H$ is of type ${\rm O}_n(q)$, so $n$ is odd. Now $z^{q^2-1}$ has Jordan form $(J_{n-3},J_1^3)$, but there are no such unipotent elements in $H$ since all even size unipotent Jordan blocks must occur with an even multiplicity (see \cite[Theorem 3.1(ii)]{LS_book}, for example). Therefore $z$ is a derangement and the result follows. Similarly, if $H$ is of type ${\rm Sp}_n(q)$ with $q \geqs 3$, then $z$ is a derangement since odd size Jordan blocks have even multiplicity in the Jordan form of any unipotent element of $H$ (see \cite[Theorem 3.1(ii)]{LS_book} and \cite[Lemma 6.2]{LS_book}). Finally, if $n$ is even and $H$ is of type ${\rm O}_{n}^{\e}(q)$, then $z$ is a derangement because $z^{q^2-1}$ has Jordan form $(J_{n-2},J_1^2)$, which is not compatible with the form of any unipotent element in $H$ since the even-size Jordan block $J_{n-2}$ has multiplicity $1$.

To complete the proof, we may assume $H \in \mathcal{S}$, so $H$ is almost simple with socle $H_0$. We claim that the elements $x$ and $y$ defined in Proposition \ref{p:msw} are derangements. Recall that $m = \min\{|x|,|y|\}$. 

For the groups with $4 \leqs n \leqs 12$, we can read off the possibilities for $H$ by inspecting the relevant tables in \cite[Chapter 8]{BHR} and it is straightforward to show that $H$ does not contain any elements of order $|x|$ or $|y|$. For example, if $n=6$ then $m = (q^3+1)/(2,q-1)$ and by inspecting \cite[Table 8.25]{BHR} we see that either
\[
H_0 \in \{ {\rm M}_{12}, A_7, {\rm L}_2(11), {\rm L}_3(4), {\rm U}_4(3) \},
\]
or $H_0 = {\rm L}_3(q)$ with $q$ odd. In the latter case, \cite[Theorem 2.16]{GMPS} gives $|z| \leqs q^2+q+1 < m$ for all $z \in {\rm Aut}(H_0)$, so $x$ and $y$ are derangements. Similarly, in the remaining cases one can check that $|z| \leqs 28$ for all $z \in {\rm Aut}(H_0)$, which reduces the problem to $q \in \{2,3\}$. Closer inspection of \cite[Table 8.25]{BHR} shows that $(q,H) = (3,{\rm M}_{12})$ is the only possibility and the result follows since $|z| \leqs 11< 14 =m$ for all $z \in H$. 

So for the remainder we may assume $n \geqs 13$. We now apply the main theorem of \cite{GPPS}, using the fact that $|x|$ and $|y|$ are divisible by primitive prime divisors of $q^n-1$ and $q^{n-1}-1$, respectively. We consider \cite[Examples 2.6-2.9]{GPPS} in turn. 

First observe that Ex. 2.6(b,c) and Ex. 2.8 do not arise since 
$n \geqs 13$. In Ex. 2.6(a), $H_0 = A_d$ is an alternating group and $V$ is the fully deleted permutation module over $\mathbb{F}_p$. However, this representation embeds $A_d$ in a symplectic or orthogonal group, so it does not arise. The relevant sporadic groups occurring in Ex. 2.7, as well as the cases in \cite[Table 7]{GPPS}, can all be eliminated by considering element orders, so it just remains to handle the cases recorded in \cite[Table 8]{GPPS}. Here $H_0$ is a simple classical group over $\mathbb{F}_t$, where $(t,q) = 1$, and these possibilities can also be eliminated by considering element orders. For example, if $H_0 = {\rm PSp}_d(t)$ with $t$ odd, then $n \geqs (t^{d/2}-1)/2$ and \cite[Theorem 2.16]{GMPS} gives $|z| \leqs (t^{d/2}+1)/(t-1)$ for all $z \in H$. But we have $m \geqs 2^{(n-1)/2}+1$ for all $q$, so $|z| < m$ and we conclude that $x$ and $y$ are derangements. Similarly, if $H_0 = {\rm L}_2(t)$ and $n = (t-1)/2$, then $t \geqs 27$ since $n \geqs 13$, and we have $|z| \leqs t+1$ for all $z \in H$. The result now follows since $m \geqs 2^{(n-1)/2}+1 > t+1$.
\end{proof}

\subsection{Primitive groups}\label{ss:prim_width}

Let $G \leqs {\rm Sym}(\O)$ be a finite primitive permutation group with socle $N$ and recall that the possibilities for $G$ and $N$ are described by the O'Nan-Scott theorem. Here our goal is to prove Theorem \ref{t:width_gen}, which states that $N = \Delta(N)^2$ if $|\O| \geqs 3$, under the assumption that Conjecture \ref{c:lst} holds. 

\begin{proof}[Proof of Theorem \ref{t:width_gen}]
We consider the various possibilities for $G$ in turn. If $G$ is almost simple, then $N$ is simple and transitive, so $N = \Delta(N)^2$ since we are assuming Conjecture \ref{c:lst} holds. If $G$ is an affine group, or a twisted wreath product, then $N$ is regular and we deduce that $\delta(N) = 1-|N|^{-1}$. So if $|N|>2$ then $\delta(N) > 1/2$ and thus $N = \Delta(N)^2$ by Lemma \ref{l:easy2}. And if $|N| = 2$ then $G = N = S_2$, $\O = \{1,2\}$, $\Delta(N) = \{(1,2)\}$ and $\Delta(N)^2 = \{1\} \ne N$, which explains why we assume $|\O| \geqs 3$ in Theorem \ref{t:width_gen}.

If $G$ is a diagonal type group, then our proof of Theorem \ref{t:main3}(i) at the end of Section \ref{s:prop} shows that $\delta(N) \geqs 2/3$ and thus $N = \Delta(N)^2$ by Lemma \ref{l:easy2}. Finally, let us assume $G \leqs L \wr S_b$ is a product type group, where $b \geqs 2$ and $L \leqs {\rm Sym}(\Gamma)$ is a primitive almost simple or diagonal type group with socle $S = T^a$. Then $N = S^b = T^{ab}$ and $G$ acts on $\O = \Gamma^b$ with the product action. Now 
\[
\Delta(N) = \{ (x_1, \ldots, x_b) \in S^b \,:\, \mbox{at least one $x_i$ is a derangement on $\Gamma$} \}
\]
and thus $\Delta(N)$ contains $\Delta(S) \times S^{b-1}$. The result now follows since we have already shown that $\Delta(S)^2 = S$.
\end{proof}

\section{Soluble stabilisers}\label{ss:sol}

Let $G$ be a (non-abelian) finite simple group and let $\mathcal{S}$ be the set of soluble maximal subgroups of $G$ (note that $\mathcal{S}$ may be empty). In addition, let $\mathcal{S}^{+}$ be the set of soluble subgroups of $G$ of the form $H = M \cap G$, where $M$ is a maximal subgroup of an almost simple group with socle $G$. Note that if 
$H = M \cap G \in \mathcal{S}^{+}$ then the solubility of ${\rm Out}(G)$ implies that $GM/G \cong M/H$ is soluble, and thus $M$ is also soluble. The subgroups in $\mathcal{S}$ and $\mathcal{S}^{+}$ were determined (up to conjugacy) by Li and Zhang in \cite{LZ} (in particular, we refer the reader to Tables 14-20 in \cite{LZ}).

The main result of this section is the following.

\begin{thm}\label{t:main_sol}
Let $G$ be a finite simple transitive group with soluble point stabiliser $H$.

\vspace{1mm}

\begin{itemize}\addtolength{\itemsep}{0.2\baselineskip}
\item[{\rm (i)}] If $H \in \mathcal{S}^{+}$, then $\delta(G) \geqs 89/325$, with equality if and only if $G = {}^2F_4(2)'$ and $H = 2^2.[2^8].S_3$.
\item[{\rm (ii)}] If $H \in \mathcal{S}$, then $G = \Delta(G)^2$.
\end{itemize}
\end{thm}

Note that part (i) establishes a slightly stronger version of Theorem \ref{t:main2} (see Remark \ref{r:main2}(b)), while part (ii) is Theorem \ref{t:main4}(ii). In Section \ref{ss:prim_sol}, we will use part (i) to prove Theorem \ref{t:main3}(ii) on primitive groups with soluble point stabilisers. In particular, in this section we will complete the proofs of Theorems \ref{t:main2}, \ref{t:main3} and \ref{t:main4}.

\begin{rem}
Our proof of part (ii) of Theorem \ref{t:main_sol} shows that Conjecture \ref{c:lst0} holds for every non-classical simple primitive group $G$ with soluble point stabilisers. That is to say, in every case we will exhibit classes $C,D$ of derangements such that $G = \{1\} \cup CD$. 
\end{rem}

\subsection{Sporadic groups}

We begin the proof of Theorem \ref{t:main_sol} by assuming $G$ is a sporadic group. Note that in the following result, $H$ is an arbitrary soluble subgroup.

\begin{prop}\label{p:spor_sol}
Let $G$ be a transitive finite simple sporadic group with soluble point stabiliser $H$. Then $G = \Delta(G)^2$ and $\delta(G) \geqs 21/55$, with equality if and only if $G = {\rm M}_{11}$ and $H =  {\rm U}_3(2).2$ or $2.S_4$.
\end{prop}

\begin{proof}
The first claim $G = \Delta(G)^2$ follows from Proposition \ref{p:spor_width}. More precisely, either $G = C^2$ for some conjugacy class $C$ of derangements, or $(G,H) = ({\rm M}_{23}, 2^4.(15{:}4))$ and 
\[
G = \{1\} \cup C^2 = \{1\} \cup D^2 = CD,
\]
where $C$ and $D$ are the two conjugacy classes in $G$ of elements of order $23$.

Now let us turn to the lower bound on $\delta(G)$. It will be convenient to define
\begin{equation}\label{e:alphas}
\a_{\rm s}(G) = \min\{ \delta(G,H) \,:\, \mbox{$H < G$ is soluble}\}.
\end{equation}
Since $\a_s(G) \geqs \a(G)$, we only need to consider the groups with $\a(G) \leqs 21/55$. By inspecting Table \ref{tab:spor}, it follows that $G \in \mathcal{A} \cup \mathcal{B}$, where
\begin{align*}
\mathcal{A} & = \{ {\rm M}_{11}, {\rm M}_{12}, {\rm M}_{22}, {\rm M}_{23}, {\rm M}_{24}, {\rm J}_3, {\rm McL}, {\rm HS}, {\rm Co}_3, {\rm O'N} \} \\
\mathcal{B} & = \{ {\rm Ly}, {\rm Th}\}
\end{align*}

For the groups in $\mathcal{A}$ we can use {\sc Magma} to construct $G$ as a primitive permutation group of minimal degree, together with a set of representatives of the conjugacy classes of elements and maximal subgroups of $G$. For each maximal subgroup $H$, it is then straightforward to determine the set of derangements on $G/H$ and we can read off $\delta(G,H)$. In particular, if $H$ is soluble, then we find that $\delta(G,H) \geqs 21/55$, with equality if and only if $G = {\rm M}_{11}$ and $H = {\rm U}_3(2).2$ or $2.S_4$. And if $H$ is insoluble and $\delta(G,H) \leqs 21/55$, then we construct a set of representatives of the conjugacy classes of maximal subgroups $K$ of $H$ and in every case one can check that $\delta(G,K) > 21/55$. 

Finally, let us assume $G = {\rm Ly}$ or ${\rm Th}$, and let $H$ be a maximal subgroup of $G$. First we use \textsf{GAP} (as in the proof of Proposition \ref{p:spor1}) to show that $\delta(G,H) \leqs 21/55$ if and only if $(G,H) = ({\rm Ly}, 2.A_{11})$ or $({\rm Th}, 2^{1+8}.A_9)$.

Suppose $G = {\rm Ly}$. Here we use the {\sc Magma} function \textsf{MatrixGroup} to obtain $G$ as a subgroup of ${\rm GL}_{111}(5)$ and we construct $H = 2.A_{11}$ via the function \textsf{MaximalSubgroups}. We then construct a set of representatives of the conjugacy classes of maximal subgroups $K$ of $H$, and in each case we check that 
\[
\delta(G,K) \geqs \frac{|\{ x \in G \,:\, |x| \not\in \omega(K)\}|}{|G|} > \frac{21}{55},
\]
where $\omega(K) = \{ |x| \,:\, x \in K\}$ is the spectrum of $K$. The case $G = {\rm Th}$ is entirely similar, working with a $248$-dimensional matrix representation of $G$ over $\mathbb{F}_2$.
\end{proof}

\begin{rem}\label{r:depth}
We can compute $\a_s(G)$ precisely for all of the sporadic groups $G$ recorded in Table \ref{tab:alphas}. To explain how we do this, let $d$ be a positive integer and let $\mathcal{M}_d$ be a set of representatives of the conjugacy classes in $G$ of subgroups $K$ such that
\[
K = K_d < K_{d-1} < \cdots < K_{1} < K_0 = G,
\]
where $K_i$ is maximal in $K_{i-1}$ for all $i$. Setting $\b_0(G) = 1$, we define
\begin{align*}
\b_d(G) & = \min\{ \b_{d-1}(G), \, \delta(G,H) \,:\, \mbox{$H \in \mathcal{M}_d$ is soluble}\} \\
\gamma_d(G) & = \min\{1, \, \delta(G,H) \,:\, \mbox{$H \in \mathcal{M}_d$ is insoluble}\}
\end{align*}
and we note that $\a_{\rm s}(G) = \b_d(G)$, where $d \geqs 1$ is minimal such that $\b_d(G) \leqs \gamma_d(G)$. This observation allows us to compute $\a_{\rm s}(G)$ for each of the groups in Table \ref{tab:alphas} by proceeding as in the proof of Proposition \ref{p:spor_sol}, working with a permutation representation of $G$ and repeatedly applying the function \textsf{MaximalSubgroups} to descend deeper in to the subgroup lattice of $G$. 
\end{rem}

{\scriptsize
\begin{table}
\[
\begin{array}{lllll} \hline
G & \a_{\rm s}(G) & & & \\ \hline
A_5 & 1/3 & & {\rm M}_{11} & 21/55 \\
A_6 & 2/5 & & {\rm M}_{12} & 244/495 \\
A_7 & 17/35 & & {\rm M}_{22} & 401/693 \\
A_8 & 17/35 & & {\rm M}_{23} & 982/1771 \\
A_9 & 227/560 & & {\rm M}_{24} & 87749/161920 \\
A_{10} & 7141/12600 & & {\rm J}_1 & 573/1463 \\
A_{11} & 47063/69300 & & {\rm J}_2 & 3251/6048 \\
A_{12} & 66541/103950 & & {\rm J}_3 & 1398703/2511648 \\
A_{13} & 48632009/64864800 & & {\rm HS} & 108805/177408 \\
A_{14} & 1319450477/1816214400 & & {\rm McL} & 10673/18711 \\
A_{15} & 1728445871/2421619200 & & {\rm He } & 130507/217600 \\
A_{16} & 541166748751/697426329600 & & {\rm Suz} & 1030209/1601600 \\
& & & {\rm Ru} & 1137154/1781325 \\
& & & {\rm O'N} & 7129127/11111485 \\
& & & {\rm Co}_3 & 54763537/82627776 \\
& & & {\rm Co}_2 & 78264622699/112814456832 \\ 
& & & {\rm Fi}_{22} & 18609741/25625600 \\ \hline
\end{array}
\]
\caption{The values of $\a_s(G)$ for some alternating and sporadic groups}
\label{tab:alphas}
\end{table}
}

\subsection{Alternating groups}

\begin{prop}\label{p:alt_sol}
Let $G = A_n$ be a transitive finite simple alternating group with soluble point stabiliser $H$. 

\vspace{1mm}

\begin{itemize}\addtolength{\itemsep}{0.2\baselineskip}
\item[{\rm (i)}] We have $G = \Delta(G)^2$.

\item[{\rm (ii)}] If $5 \leqs n \leqs 16$, then $\delta(G) \geqs 1/3$, with equality if and only if $(G,H) = (A_5,D_{10})$.

\item[{\rm (iii)}] If $n \geqs 17$ and $H \in \mathcal{S}^{+}$, then $\delta(G) \geqs 4531887936311/5230697472000$.
\end{itemize}
\end{prop}

\begin{proof}
Part (i) is a special case of Proposition \ref{p:alt_width}. More precisely, either $G = C^2$ for some conjugacy class $C$ of derangements, or
\[
(G,H) \in \{ (A_5,A_4), \, (A_5,S_3), \, (A_8, 2^4{:}(S_3\times S_3)) \}
\]
and $G = \{1\} \cup CD$, where $C$ and $D$ are the two classes of elements order $5$ (for $G = A_5$) or $7$ (for $G = A_8$). 

Part (ii) can be checked using {\sc Magma} and the approach described in  Remark \ref{r:depth}. In particular, we can compute $\a_{\rm s}(G)$ precisely for all $5 \leqs n \leqs 16$ (see \eqref{e:alphas} and Table \ref{tab:alphas}).

For the remainder, we may assume $n \geqs 17$ and $H \in \mathcal{S}^{+}$. 
Then the solubility of $H$ implies that $H$ acts primitively on $\{1, \ldots, n\}$, which in turn means that $n = p^d$ is a prime power and $H = {\rm AGL}_d(p) \cap G$. Moreover, since $n \geqs 17$, it follows that $d=1$ and $H = {\rm AGL}_1(p) \cap G = p{:}(p-1)/2$. 

Suppose $x \in H$ has order $r$. Then either $r=p$ and $x$ is a $p$-cycle, or $r$ divides $(p-1)/2$ and $x$ has cycle-type $(r^{(p-1)/r},1)$ as an element of $G$. (Note that $H$ contains a Sylow $p$-subgroup of $G$, so every $p$-cycle in $G$ has fixed points on $\O$.) So if $N$ denotes the number of elements in $G$ with fixed points on $\O$, then
\[
N = 1 + (p-1)! + \sum_{r \in \Lambda} \frac{p!}{((p-1)/r)!r^{(p-1)/r}},
\]
where $\Lambda$ is the set of divisors $r \geqs 2$ of $(p-1)/2$. For $17 \leqs p<100$, it is now entirely straightforward to check that the lower bound in part (iii) is satisfied (with equality if $p=17$), so we may assume $p>100$.

For $r \in \Lambda$, set 
\[
f(r) = \frac{p!}{((p-1)/r)!r^{(p-1)/r}}
\]
We claim that $f(r)$ is maximal when $r = (p-1)/2$. To see this, simply observe that $((p-1)/r)! \geqs 2$ and $r^{(p-1)/r}$ is decreasing as a function of $r$, so we have $r^{(p-1)/r} \geqs ((p-1)/2)^2$ and the claim follows. Therefore, 
\[
f(r) \leqs 2\left(\frac{p!}{(p-1)^2}\right)
\]
and using the crude bound $|\Lambda| \leqs 2\sqrt{(p-1)/2}$ we deduce that
\[
N \leqs 1+(p-1)! + 2\sqrt{(p-1)/2} \cdot 2\left(\frac{p!}{(p-1)^2}\right).
\]
In turn, this yields 
\[
\delta(G) \geqs \frac{2\left(\frac{p!}{2} - 1 - (p-1)! - 2\sqrt{(p-1)/2} \cdot \frac{2(p!)}{(p-1)^2}\right)}{p!} = 1-\frac{2}{p!} - \frac{2}{p} - \frac{8\sqrt{(p-1)/2}}{(p-1)^2}
\]
and we conclude that $\delta(G) > 0.97$ for all $p > 100$. 
\end{proof}

\subsection{Exceptional groups}\label{ss:exccc}

Next we assume $G$ is a simple exceptional group of Lie type. Our main result is the following, which establishes strong forms of Theorems \ref{t:main2} and \ref{t:main4}(ii) for exceptional groups. Note that the possibilities for $G$ and $H$ are recorded in \cite[Table 20]{LZ}.

\begin{prop}\label{p:ex_sol}
Let $G$ be a transitive finite simple exceptional group of Lie type with soluble point stabiliser $H$. 

\vspace{1mm}

\begin{itemize}\addtolength{\itemsep}{0.2\baselineskip}
\item[{\rm (i)}] If $H \in \mathcal{S}^{+}$, then $\delta(G) \geqs 89/325$, with equality if and only if $G = {}^2F_4(2)'$ and $H = 2^2.[2^8].S_3$.
\item[{\rm (ii)}] If $H \in \mathcal{S}$, then $G = \{1\} \cup CD$ for classes $C,D$ of derangements.
\end{itemize}
\end{prop}

We begin by recording two preliminary lemmas. Recall that a conjugacy class $x^G$ is \emph{real} if $x^{-1} \in x^G$. 

\begin{lem}\label{l:real}
Let $G$ be a finite simple exceptional group of Lie type over $\mathbb{F}_q$ and assume $G \ne E_6(q), {}^2E_6(q), G_2(2)'$. Then every semisimple conjugacy class in $G$ is real.
\end{lem}

\begin{proof}
The groups ${}^2G_2(3)'$ and ${}^2F_4(2)'$ can be checked directly. In each of the remaining cases, the result follows from \cite[Proposition 3.1]{TZ} since the Weyl group of $G$ contains a central involution.
\end{proof}

In the next lemma, $G$ is a simple group of Lie type over $\mathbb{F}_q$, where $q = p^f$ and $p$ is a prime, and $h$ is the \emph{Coxeter number} of $G$. The latter is defined by 
\[
h = \frac{\dim \bar{G}}{{\rm rank}\, \bar{G}} - 1,
\]
where $\bar{G}$ is the ambient simple algebraic group defined over the algebraic closure of $\mathbb{F}_q$. So for example, the Coxeter number of $E_8(q)$ is $248/8-1 = 30$. Also recall that an element $x \in G$ is \emph{regular semisimple} if the connected component of $C_{\bar{G}}(x)$ is a maximal torus of $\bar{G}$. This is equivalent to the condition that $|C_G(x)|$ is indivisible by $p$. 

For the proof of the following lemma, we thank Martin Liebeck for drawing our attention to \cite[Section 4]{HK}, which plays a key role in the argument. 

\begin{lem}\label{l:cox}
Let $G \ne {\rm Sp}_4(2)', G_2(2)', {}^2G_2(3)'$ be a finite simple group of Lie type with Coxeter number $h$ and suppose $x \in G$ is semisimple. Then $x$ is regular only if $|x| \geqs h$.
\end{lem}

\begin{proof}
The claim is trivial if $h=2$ and it can be checked directly for $G = {}^2F_4(2)'$, so we may assume $h \geqs 3$ and $G \ne {}^2F_4(2)'$. Set $|x| = d$.

Let $\bar{G} = X_{{\rm sc}}$ and $\bar{L} = X_{{\rm ad}}$ be the corresponding simply connected and adjoint simple algebraic groups defined over the algebraic closure of $\mathbb{F}_q$ and write $G = O^{p'}(\bar{L}_{\s})$ for a suitable Steinberg endomorphism $\s$ of $\bar{L}$. We may assume that $x$ is the image of $y \in \bar{G}$ under the natural map $\bar{G} \to X_{{\rm ad}}$. Embed $y$ in a maximal torus $\bar{T}$ of $\bar{G}$ and let $\{\a_1, \ldots, \a_r\}$ be a set of simple roots for the corresponding root system of $\bar{G}$. Let $\a_0 = \sum_i m_i\a_i$ be the highest root and note that $\sum_im_i = h-1$. 

The structure of $C_{\bar{G}}(y)$ is described in \cite[Section 4]{HK}. In particular, from the discussion on \cite[p.315]{HK} it follows that there exist integers $b_i \geqs 0$ such that $\sum_i b_im_i \leqs d$ with the property that $C_{\bar{G}}(y)$ is a maximal torus only if each $b_i$ is positive and $d> \sum_i b_im_i$. Therefore, $x$ is regular only if 
\[
d \geqs 1+\sum_i m_i = h
\]
and the result follows.
\end{proof}

It is convenient to use computational methods to prove Proposition \ref{p:ex_sol} for some of the small exceptional groups. 

\begin{lem}\label{l:ex_small}
The conclusion to Proposition \ref{p:ex_sol} holds if $G$ is one of the following:
\[
{}^2G_2(3)', \, G_2(2)', \, G_2(3), \, {}^2F_4(2)', \, {}^3D_4(2), \, {}^3D_4(3).
\]
\end{lem}

\begin{proof}
We can use {\sc Magma} \cite{magma} to handle these cases, working with a primitive permutation representation of $G$ of minimal degree. First we construct the character table of $G$ and we use Lemma \ref{l:frob} to identify all the pairs of classes $(C,D)$ with $G = \{1\} \cup CD$. Then for each $H \in \mathcal{S}^{+}$ we determine the set of derangements for the action of $G$ on $G/H$ and we check that $\delta(G,H) \geqs 89/325$, with equality if and only if $G = {}^2F_4(2)'$ and $H = 2^2.[2^8].S_3$. For example, if $G = {}^2G_2(3)' \cong {\rm L}_2(8)$ then 
$\delta(G,H) \geqs 3/7$, with equality if and only if $H = D_{18}$. Finally, in every case we identify two classes $C,D$ of derangements with $G = \{1\} \cup CD$. 
\end{proof}

We now partition our analysis of the remaining exceptional groups according to the structure of the point stabiliser $H$.

\begin{lem}\label{l:ex_par}
The conclusion to Proposition \ref{p:ex_sol} holds if $H$ is a parabolic subgroup of $G$.
\end{lem}

\begin{proof}
Suppose $H \in \mathcal{S}^{+}$ is a parabolic subgroup of $G$, in which case the possibilities for $G$ and $H$ are recorded in parts (ii) and (iii) of \cite[Lemma 5.4]{Bur21}.

We begin by assuming $(G,H)$ is one of the cases listed in \cite[Lemma 5.4(ii)]{Bur21}. In view of Lemma \ref{l:ex_small}, we can assume $G = F_4(2)$ and $H = [2^{22}]{:}S_3^2$. Here $H$ is non-maximal in $G$, but $H.2$ is maximal in $L = {\rm Aut}(G) = G.2$, so $H \in \mathcal{S}^{+} \setminus \mathcal{S}$ and we just need to bound $\delta(G,H)$. To do this, we first use {\sc Magma} to construct $G$ as a permutation group of degree $139776$ and then we obtain $H$ by constructing the maximal subgroups of ${\rm Aut}(G)$. In the usual manner, working with the conjugacy classes in $G$ and $H$, we compute $\delta(G,H) = 5166407/7309575$.

To complete the proof of the lemma, it remains to consider the three infinite families recorded in \cite[Table 2]{Bur21}:

\vspace{1mm}

\begin{itemize}\addtolength{\itemsep}{0.2\baselineskip}
\item[(a)] $G = {}^2B_2(q)$ and $H = [q^2]{:}(q-1)$, where $q = 2^{2m+1} \geqs 8$. 
\item[(b)] $G = {}^2G_2(q)$ and $H = [q^3]{:}(q-1)$, where $q  = 3^{2m+1} \geqs 27$.
\item[(c)] $G = G_2(q)$ and $H = [q^6]{:}(q-1)^2$, where $q=3^m \geqs 9$.
\end{itemize}

\vspace{1mm}

In cases (a) and (b), Theorem \ref{t:rankone} implies that $G = C^2$ for some conjugacy class $C$ of derangements. Also note that $H \in \mathcal{S}^{+}\setminus \mathcal{S}$ in (c), so in all three cases it just remains to show that $\delta(G,H) > 89/325$.

First consider case (a). Let $\chi = 1^G_H$ be the corresponding permutation character and note that $\chi = 1 + {\rm St}$, where $1$ and ${\rm St}$ denote the trivial and Steinberg characters of $G$, respectively (for example, see \cite[p.416]{LLS}). By inspecting the character table of $G$ (see \cite[Theorem 13]{Suz}), we deduce that $x \in G$ is a derangement if and only if $x$ is a regular semisimple element with $|C_G(x)| = q \pm \sqrt{2q}+1$. As a consequence, 
\[
|\Delta(G,H)| = \frac{1}{4}(q+\sqrt{2q}) \cdot \frac{|G|}{q+\sqrt{2q}+1} + \frac{1}{4}(q-\sqrt{2q}) \cdot \frac{|G|}{q-\sqrt{2q}+1} = \frac{1}{2}q^3(q-1)^2
\]
and thus
\[
\delta(G,H) = \frac{q(q-1)}{2(q^2+1)} \geqs \frac{28}{65}
\]
for all $q \geqs 8$.

Case (b) is very similar. Once again we have $\chi = 1 + {\rm St}$ and we compute
\[
|\Delta(G,H)| = \frac{1}{6}(q+\sqrt{3q}) \cdot \frac{|G|}{q+\sqrt{3q}+1} + \frac{1}{6}(q-\sqrt{3q}) \cdot \frac{|G|}{q-\sqrt{3q}+1} + \frac{1}{6}(q-3) \cdot \frac{|G|}{q+1},
\]
which yields 
\[
\delta(G,H) = \frac{q^3-2q^2-1}{2(q^3+1)} \geqs \frac{2278}{4921}.
\]

Finally, we turn to case (c). If $x \in G$ is a regular semisimple element with $|C_G(x)| = q^2 \pm q +1$ or $(q+1)^2$, then $(|H|,|C_G(x)|) \leqs 4$ and thus $x$ is a derangement by Lemma \ref{l:cox} (note that $h=6$). By inspecting \cite{Lub}, we deduce that
\[
\delta(G,H) \geqs \frac{1}{12}\left(\frac{2(q-1)(q+2)}{q^2+q+1}+\frac{2(q+1)(q-2)}{q^2-q+1}+\frac{q(q-4)}{(q+1)^2}\right) > \frac{89}{325}
\]
for all $q \geqs 9$.
\end{proof}

Next we consider the groups where $H$ is the normaliser of a maximal torus. These cases require a detailed analysis and we divide the proof into two separate lemmas.

In the proof of the following result, we thank Gunter Malle for his assistance with a \textsf{Chevie} \cite{Chevie} computation arising when $G = {}^2F_4(q)$ in Case 4 of the proof.

\begin{lem}\label{l:ex_tor}
The conclusion to Proposition \ref{p:ex_sol} holds if $H = N_G(T)$ is the normaliser of a maximal torus and $G \ne F_4(q), E_6^{\e}(q), E_8(q)$.
\end{lem}

\begin{proof}
The possibilities for $H$ can be read off from \cite[Table 5.2]{LSS} and we divide the proof into a number of separate cases. Note that for the rank one groups ${}^2B_2(q)$ and ${}^2G_2(q)$ we have already shown that $G = \{1\} \cup CD$ for classes $C,D$ of derangements in Theorem \ref{t:rankone}.

\vs

\noindent \emph{Case 1. $G = {}^2B_2(q)$}

\vs

Here $q \geqs 8$ and $H$ is either $D_{2(q-1)}$ or $(q \pm \sqrt{2q}+1){:}4$. If $H = D_{2(q-1)}$ then every semisimple element $x \in G$ with $|C_G(x)| = q \pm \sqrt{2q}+1$ is a derangement (since $(|H|,|C_G(x)|) = 1$) and thus
\[
\delta(G,H) \geqs \frac{1}{4}\left(\frac{q+\sqrt{2q}}{q+\sqrt{2q}+1} + \frac{q-\sqrt{2q}}{q-\sqrt{2q}+1}\right) \geqs \frac{28}{65}.
\]
Similarly, if $H = (q +\e\sqrt{2q}+1){:}4$ then every semisimple element $x \in G$ with $|C_G(x)| = q-1$ or $q-\e\sqrt{2q}+1$ is a derangement (since $(|H|,|C_G(x)|) = 1$), whence
\[
\delta(G,H) \geqs \frac{1}{2}\left(\frac{q-2}{q-1} + \frac{q-\e\sqrt{2q}}{2(q-\e\sqrt{2q}+1)}\right) \geqs \frac{22}{35}.
\]

\vs

\noindent \emph{Case 2. $G = {}^2G_2(q)$, $q \geqs 27$}

\vs

This is very similar to the previous case (note that the group ${}^2G_2(3)'$ was  handled in Lemma \ref{l:ex_small}). 
If $H = (q+1).6$ then each element $x \in G$ with $|C_G(x)| = q \pm \sqrt{3q}+1$ is a derangement (once again, we have $(|H|,|C_G(x)|) = 1$) and thus
\[
\delta(G,H) \geqs \frac{1}{6}\left(\frac{q+\sqrt{3q}}{q+\sqrt{3q}+1} + \frac{q-\sqrt{3q}}{q-\sqrt{3q}+1}\right) \geqs \frac{225}{703}.
\]
Similarly, if $H = (q+\e\sqrt{3q}+1).6$ and $|C_G(x)| = q-1$ or $q-\e\sqrt{3q}+1$, then we have $|x|>2$ and $(|H|,|C_G(x)|) \leqs 2$, so $x$ is a derangement and we deduce that  
\[
\delta(G,H) \geqs \frac{1}{6}\left(\frac{3(q-3)}{q-1} + \frac{q-\e\sqrt{3q}}{q-\e\sqrt{3q}+1}\right) \geqs \frac{153}{247}.
\]

\vs

\noindent \emph{Case 3. $G = G_2(q)$}

\vs

Next assume $G = G_2(q)$. By inspecting \cite[Table 5.2]{LSS} we see that  $p=3$, $q \geqs 9$ and $H \in \mathcal{S}^{+} \setminus \mathcal{S}$, so we just need to verify the lower bound $\delta(G) > 89/325$.

For $H = (q-\e)^2.D_{12}$ it is clear that every regular semisimple element $x \in G$ with $|C_G(x)| = q^2 \pm q +1$ is a derangement and by inspecting \cite{Lub} we deduce that
\[
\delta(G,H) \geqs \frac{1}{6}\left(\frac{q^2+q}{q^2+q+1}+\frac{q^2-q}{q^2-q+1}\right) \geqs \frac{2187}{6643}
\]
for all $q \geqs 9$. Similarly, if $H = (q^2+\e q+1).6$ then each $x \in G$ with $|C_G(x)| = q^2-1$ or $q^2-\e q+1$ is a derangement and it is easy to check that this gives $\delta(G,H)>89/325$.

\vs

\noindent \emph{Case 4. $G = {}^2F_4(q)$, $q \geqs 8$}

\vs

Next suppose $G = {}^2F_4(q)$ with $q \geqs 8$ (see Lemma \ref{l:ex_small} for $G = {}^2F_4(2)'$). If $C$ is a conjugacy class of elements of order $q^2+ \sqrt{2q^3}+q+ \sqrt{2q}+1$, then \cite[Theorem 7.3]{GM} implies that 
$G = C^2$ (noting that $C$ is real by Lemma \ref{l:real}). And by arguing as in the proof of \cite[Theorem 7.3]{GM}, using the \textsf{GAP} package \textsf{Chevie} \cite{Chevie}, one can show that $G = D^2$ when $D$ is a class of elements of order $q^2- \sqrt{2q^3}+q- \sqrt{2q}+1$.

First assume $H = (q+1)^2.{\rm GL}_2(3)$. If $x \in G$ is a regular semisimple element with 
\[
|C_G(x)| \in \{ q^2-q+1, \, (q-1)(q \pm \sqrt{2q}+1), \, q^2+ \sqrt{2q^3}+q+\sqrt{2q}+1\}
\]
then $x$ is a derangement (in each case, $(|H|,|C_G(x)|) = 1$ or $3$) and by inspecting \cite[Table IV]{Shinoda} we deduce that
\begin{align*}
\delta(G,H) \geqs & \; \frac{(q-2)(q+1)}{6(q^2-q+1)} + \frac{(q-2)(q+\sqrt{2q})}{8(q-1)(q+\sqrt{2q}+1)} + \frac{(q-2)(q-\sqrt{2q})}{8(q-1)(q-\sqrt{2q}+1)} \\
& + \frac{q^2+ \sqrt{2q^3}+q+\sqrt{2q}}{12(q^2+ \sqrt{2q^3}+q+\sqrt{2q}+1)} > \frac{89}{325}
\end{align*}
for all $q \geqs 8$. In addition, $G = C^2$ for the class $C$ of derangements defined above.

Similarly, if $H = (q+\e \sqrt{2q}+1)^2.(4 \circ {\rm GL}_2(3))$ then every regular semisimple element $x \in G$ with $|C_G(x)| = q^2-q+1$, $q^2-1$ or $q^2+ \sqrt{2q^3}+q+ \sqrt{2q}+1$ is a derangement and the desired bound $\delta(G,H) > 89/325$ quickly follows. And finally, if $H = (q^2+ \e \sqrt{2q^3}+q+ \e \sqrt{2q}+1).12$, then each $x \in G$ with $|C_G(x)| = q^2-q+1$, $q^2-1$ or $q^2- \e \sqrt{2q^3}+q-\e \sqrt{2q}+1$ is a derangement and we deduce that $\delta(G,H) > 89/325$. Note that in each of these cases, either $C$ or $D$ is a class of derangements and we have $G = C^2 = D^2$.

\vs

By inspecting \cite[Table 5.2]{LSS}, it just remains to consider the groups $G = {}^3D_4(q)$ in order to complete the proof of the lemma.

\vs

\noindent \emph{Case 5. $G = {}^3D_4(q)$}

\vs

In view of Lemma \ref{l:ex_small}, we may assume $q \geqs 4$. We refer the reader to \cite{DM} for information on the semisimple conjugacy classes in $G$ (for example, the number of semisimple classes with a given centraliser structure can be read off from \cite[Table 4.4]{DM}). By \cite[Theorem 7.3]{GM} and Lemma \ref{l:real} we have $G = C^2$, where $C$ is a class of elements of order $q^4-q^2+1$. 

Suppose $H = (q^4-q^2+1).4$. If $x \in G$ is regular semisimple with $|C_G(x)| = (q^3-\e)(q-\e)$, $(q^3+\e)(q-\e)$ or $(q^2+\e q+1)^2$, then $(|H|,|C_G(x)|) \leqs 4$ and thus $x$ is a derangement by Lemma \ref{l:cox}. By adding up the contribution to $\delta(G)$ from these elements, we get
\begin{align*}
\delta(G,H) \geqs & \frac{1}{24}\left( \frac{q^4+2q^3-q^2-2q}{(q^2+q+1)^2} + \frac{q^4-2q^3-q^2+2q}{(q^2-q+1)^2} + \frac{2(q^4-4q^3+2q^2-2q+12)}{(q^3-1)(q-1)} \right. \\
& \left. \hspace{8mm} + \frac{2(q^4-2q^3+2q^2-4q)}{(q^3+1)(q+1)} + \frac{6(q^4-2q)}{(q^3-1)(q+1)} + \frac{6(q^4-2q^3)}{(q^3+1)(q-1)}\right) > \frac{89}{325}
\end{align*}
for all $q \geqs 4$. In addition, if we let $A = x^G$ and $B = y^G$, where $|C_G(x)| = (q^2+q+1)^2$ and $|C_G(y)| = (q^2-q+1)^2$, then we can use \textsf{Chevie} \cite{Chevie} to show that $G = \{1\} \cup AB$. In particular, $G = \Delta(G)^2$.

Finally, let us assume $H = (q^2+\e q+1)^2.{\rm SL}_2(3)$. First note that each element $x \in G$ with $|C_G(x)| = q^4-q^2+1$ is a derangement (since $(|H|,|C_G(x)|) = 1$) and we have $G = C^2$ as noted above. In addition, elements with $|C_G(x)| = (q^2-\e q+1)^2$ are also derangements since $(|H|,|C_G(x)|) = 1$ or $3$. It follows that 
\[
\delta(G,H) \geqs \frac{1}{24}\left(\frac{6(q^4-q^2)}{q^4-q^2+1} + \frac{q^4+2\e q^3-q^2-2\e q}{(q^2+\e q+1)^2}\right) \geqs \frac{11345}{40729} > \frac{89}{325}
\]
for all $q \geqs 4$.
\end{proof}

\begin{lem}\label{l:ex_tor2}
The conclusion to Proposition \ref{p:ex_sol} holds if $H = N_G(T)$ is the normaliser of a maximal torus.
\end{lem}

\begin{proof}
In view of Lemma \ref{l:ex_tor}, we may assume $G = F_4(q)$, $E_6^{\e}(q)$ or $E_8(q)$. By inspecting \cite[Table 5.2]{LSS}, one of the following holds:

\vspace{1mm}

\begin{itemize}\addtolength{\itemsep}{0.2\baselineskip}
\item[{\rm (a)}] $G = F_4(q)$, $q$ even and $H \in \mathcal{S}^{+} \setminus \mathcal{S}$;
\item[{\rm (b)}] $G = E_6^{\e}(q)$ and $T = \frac{1}{e}(q^2+\e q+1)^3$, where $e = (3,q-\e)$;
\item[{\rm (c)}] $G = E_8(q)$ and $T = (q^4-q^2+1)^2$ or $q^8+\e q^7-\e q^5-q^4-\e q^3+\e q+1$ with $\e = \pm$.
\end{itemize}

Recall Lemma \ref{l:cox}, which states that $|x| \geqs h$ for every regular semisimple element $x \in G$, where $h$ is the Coxeter number of $G$ (so $h = 12,12,30$ in cases (a),(b),(c) above). We refer the reader to \cite{Lub} for a convenient source of detailed information on the semisimple conjugacy classes in $G$. We will write $\Phi_i$ for the $i$-th cyclotomic polynomial in $\mathbb{Z}[q]$. 

\vs

\noindent \emph{Case 1. $G = F_4(q)$, $q$ even}

\vs

Here $H \in \mathcal{S}^{+} \setminus \mathcal{S}$, so it suffices to verify the bound  $\delta(G) > 89/325$. If $q=2$ then we have $H = 7^2.(3 \times {\rm SL}_2(3)) = N_G(P)$, where $P$ is a Sylow $7$-subgroup of $G$, and with the aid of {\sc Magma} it is easy to check the result. For the remainder, we may assume $q \geqs 4$. 

Suppose $H = (q-1)^4.W$, where $W$ is the Weyl group of $G$ (this is a soluble group of order $1152$). According to \cite[Table 5.2]{LSS}, we may assume $q \geqs 8$. Let $x \in G$ be a regular semisimple element with 
\[
|C_G(x)| \in \{
\Phi_8, \, \Phi_{12}, \, \Phi_2^2\Phi_6, \, \Phi_2^2\Phi_4, \, \Phi_3^2, \, \Phi_6^2\}.
\]
Since $(|H|,|C_G(x)|) \leqs 9$, Lemma \ref{l:cox} implies that $x$ is a derangement and by inspecting \cite{Lub} we deduce that 
\begin{align*}
\delta(G,H) \geqs & \; \frac{q^4}{8\Phi_8} + \frac{q^2(q^2-1)}{12\Phi_{12}} + \frac{(q+1)(q^3-3q^2+4q-4)}{18\Phi_2^2\Phi_6} + \frac{q^3(q-2)}{32\Phi_2^2\Phi_4} \\
 & + \frac{(q-1)(q^3+3q^2-2q-8)}{72\Phi_3^2} + \frac{q(q-1)(q^2-q-6)}{72\Phi_6^2} > \frac{89}{325}
\end{align*}
for all $q \geqs 8$.

Next assume $H = (q+1)^4.W$ with $q \geqs 4$. Here each regular semisimple element $x \in G$ with $|C_G(x)| = \Phi_8$, $\Phi_{12}$, $\Phi_1^2\Phi_3$, $\Phi_1^2\Phi_4$, $\Phi_3^2$ or $\Phi_6^2$ is a derangement (as above, $(|H|,|C_G(x)|) <h$ in each case) and by counting these elements we get 
$\delta(G,H) > 89/325$ for all $q \geqs 8$. For $q=4$, we also claim that each $x \in G$ with $|C_G(x)| = \Phi_1\Phi_2\Phi_6 = 195$ is a derangement. To see this, first note that $C_G(x) = T$ is a cyclic maximal torus (see \cite{DF}) and we have $(|H|,|C_G(x)|) = 15 = q^2-1$. We can embed $T$ in a subgroup $M = {\rm Sp}_2(4) \times {\rm Sp}_6(4)$ of $G$ and we observe that $C_M(S) = 
3 \times {\rm GU}_3(4)$, where $S$ is the unique subgroup of $T$ of order $15$. So if $x$ has fixed points, then $T < C_M(x)$, which is absurd. This justifies the claim. By including the contribution from these elements, it is easy to check that $\delta(G,H) > 89/325$.

Next assume $H = (q^2+q+1)^2.(3 \times {\rm SL}_2(3))$ with $q \geqs 4$. Here we find that every regular semisimple element $x \in G$ with $|C_G(x)| = \Phi_8$, $\Phi_{12}$, $\Phi_2^2\Phi_6$, $\Phi_2^2\Phi_4$, $\Phi_6^2$ or $\Phi_1\Phi_2\Phi_6$ is a derangement (since $(|H|,|C_G(x)|)<h$) and we get $\delta(G,H) > 89/325$ by adding up the contribution from these elements.

Similarly, if $H = (q^2-q+1)^2.(3 \times {\rm SL}_2(3))$ and $q \geqs 4$, then each element with $|C_G(x)| = \Phi_8$, $\Phi_{12}$, $\Phi_1^2\Phi_3$, $\Phi_1^2\Phi_4$ or $\Phi_3^2$ is a derangement and this yields 
$\delta(G,H) > 89/325$ for $q \geqs 16$. If $q \in \{4,8\}$ and $|C_G(x)| = \Phi_1\Phi_2\Phi_3$, then we get $(|H|,|C_G(x)|) = 9$, so these regular semisimple elements are also derangements and this yields $\delta(G,H) > 89/325$ as required.

The remaining two cases with $G = F_4(q)$ can be handled in a very similar fashion. If $H = (q^2+1)^2.(4 \circ {\rm GL}_2(3))$ then each regular semisimple element $x \in G$ with $|C_G(x)| = \Phi_8$, $\Phi_{12}$, $\Phi_1^2\Phi_3$, $\Phi_2^2\Phi_6$, $\Phi_3^2$, $\Phi_6^2$ or $\Phi_1\Phi_2\Phi_3$ is a derangement (indeed, in each case $(|H|,|C_G(x)|) \leqs 3$) and the desired bound follows. Similarly, if $H = (q^4-q^2+1).12$ then the elements $x \in G$ with $|C_G(x)| = \Phi_8$, $\Phi_1^2\Phi_3$, $\Phi_2^2\Phi_6$, $\Phi_3^2$, $\Phi_6^2$ or $\Phi_1\Phi_2\Phi_3$ are derangements and once again we conclude that $\delta(G,H) > 89/325$.

\vs

\noindent \emph{Case 2. $G = E_6(q)$}

\vs

Next assume $G = E_6(q)$ and $H = (q^2+q+1)^3/e.3^{1+2}.{\rm SL}_2(3)$, where $e = (3,q-1)$. First note that every element in $G$ of order $\Phi_8$ and $\Phi_9/e$ is a derangement (indeed, $|H|$ is indivisible by these two numbers), whence \cite[Theorem 7.7]{GM} implies that there exist classes $C,D$ of derangements with $G = \{1\} \cup CD$. Therefore, it remains to show that $\delta(G)>89/325$.

If $q=2$ then $H$ is the normaliser of a Sylow $7$-subgroup and it is easy to verify the desired bound using {\sc Magma}. For $q \geqs 3$ it will be convenient to work in the quasisimple group $L = Z.G$, where $Z = Z(L)$ has order $e$. Then $H = J/Z$ with $J = (q^2+q+1)^3.3^{1+2}.{\rm SL}_2(3)$ and we note that $\delta(G) = \delta(L,J)$.

Let $x \in L$ be a regular semisimple element. If $|C_L(x)| = \Phi_9$ then 
$(|J|,|C_L(x)|) = 1$ and by inspecting \cite{Lub} we see that there are at least 
$\Phi_1\Phi_3(q^3+2)/9$ distinct $L$-classes of such elements. Next assume $|C_L(x)| = \Phi_1\Phi_2\Phi_8$, in which case $C_L(x) = T$ is a cyclic maximal torus (see \cite{DF}) and $(|J|, |C_L(x)|)$ divides $q^2-1$. So if $x$ has fixed points on $L/J$, then $x \in S<T$, where $S$ is the unique subgroup of $T$ of order $q^2-1$. But $T$ is contained in a subgroup $M<G$ of type $(q^2-1) \times \O_{8}^{-}(q)$ and we deduce that $T < C_M(x)$, which is absurd since  $C_L(x) = T$. Therefore, $x$ is a derangement on $L/J$ and by inspecting \cite{Lub} we read off that $L$ has at least $\Phi_1^2\Phi_2^2\Phi_4/8$ distinct classes of such elements.

A similar argument shows that each $y \in L$ with $|C_L(y)| = \Phi_1\Phi_2\Phi_4\Phi_6$ is a derangement. To see this, first note that $C_L(y) = T$ is a cyclic maximal torus and $(|J|,|C_L(y)|)$ divides $(q-1)(q^3+1)$. Let $S$ be the unique subgroup of $T$ of order $(q-1)(q^3+1)$ and note that $y \in S$ if $y$ has fixed points. Now $T$ is contained in a subgroup $M$ of type $(q-1) \times \O_{10}^{+}(q)$ and we observe that $T<C_M(S)$. So if $y$ has fixed points, then $T<C_M(y)$ and we reach a contradiction as before. Therefore, $y$ is a derangement and \cite{Lub} indicates that $L$ has at least $q^3\Phi_1^2\Phi_2/12$ classes of regular semisimple elements of this form. 

Putting this together, we deduce that
\[
\delta(G) \geqs \frac{\Phi_1\Phi_3(q^3+2)}{9\Phi_9}+\frac{\Phi_1^2\Phi_2^2\Phi_4}{8\Phi_1\Phi_2\Phi_8}+\frac{q^3\Phi_1^2\Phi_2}{12\Phi_1\Phi_2\Phi_4\Phi_6} > \frac{89}{325}
\]
for all $q \geqs 3$.

\vs

\noindent \emph{Case 3. $G = {}^2E_6(q)$}

\vs

Next assume $G = {}^2E_6(q)$ and $H = (q^2-q+1)^3/e.3^{1+2}.{\rm SL}_2(3)$, where $e = (3,q+1)$ and $q \geqs 3$ (see \cite[Table 5.2]{LSS}). This is essentially  identical to the previous case and so we only give brief details.

First we note that $|H|$ is indivisible by $\Phi_8$ and $\Phi_{18}/e$, so by appealing to \cite[Theorem 7.7]{GM} we deduce that $G = \{1\} \cup CD$, where $C$ and $D$ are conjugacy classes of derangements. 

As in Case 2, it is convenient to work in the quasisimple group $L = Z.G$ and we set $H = J/Z$, where $Z = Z(L)$ has order $e$. By arguing as above, we can show that every regular semisimple element $x \in L$ with $|C_L(x)| = \Phi_{18}$, $\Phi_1\Phi_2\Phi_8$ or $\Phi_1\Phi_2\Phi_3\Phi_4$ is a derangement, and then by inspecting \cite{Lub} we conclude that 
\[
\delta(G) \geqs \frac{\Phi_2\Phi_6( q^3-2 )}{9\Phi_{18}} + \frac{\Phi_1^2\Phi_2^2\Phi_4}{8\Phi_1\Phi_2\Phi_8} + \frac{q^3\Phi_1\Phi_2^2}{12\Phi_1\Phi_2\Phi_3\Phi_4} > \frac{89}{325}
\]
for all $q \geqs 3$. 

\vs

\noindent \emph{Case 4. $G = E_8(q)$}

\vs

Finally, let us assume $G = E_8(q)$, in which case $H$ is one of the following:
\[
(q^4-q^2+1)^2.(12 \circ {\rm GL}_2(3)),\;\; (q^8\pm q^7 \mp q^5-q^4 \mp q^3 \pm q +1){:}30.
\]
First observe that every element in $G$ of order $\Phi_{20}$ or $\Phi_{24}$ is a derangement (since $|H|$ is indivisible by these numbers), so \cite[Theorem 7.7]{GM} implies that $G = \{1\} \cup CD$ for classes $C,D$ of derangements. Therefore, we just need to verify the bound $\delta(G) > 89/325$ and we will consider the three possibilities for $H$ in turn.

\vs

\noindent \emph{Case 4(a). $H = (q^4-q^2+1)^2.(12 \circ {\rm GL}_2(3))$}

\vs

Let $x \in G$ be a regular semisimple element and recall that $|x| \geqs 30$ by Lemma \ref{l:cox}. First observe that $(|H|,|C_G(x)|) = 1$ if $|C_G(x)| = \Phi_{i}$ with $i \in \{20,24,15,30\}$, so $x$ is a derangement and by inspecting \cite{Lub}, we see that
\[
\frac{(q^2+1)( q^6-2q^4+3q^2-4)}{20\Phi_{20}} + \frac{q^4(q^4-1)}{24\Phi_{24}} + \frac{q(q^4-1)}{30}\left(\frac{q^3-q^2+1}{\Phi_{15}} + \frac{q^3+q^2-1}{\Phi_{30}}\right)
\]
is the contribution to $\delta(G,H)$ from these elements.

Next we claim that $x$ is a derangement if  
\[
|C_G(x)| \in \{ \Phi_1^2\Phi_7, \, \Phi_2^2\Phi_{14}, \, \Phi_1\Phi_2\Phi_7, \, \Phi_1\Phi_2\Phi_{14} \}.
\]
To see this, first assume $|C_G(x)| = \Phi_1^2\Phi_7$ and set $T = C_G(x)$ and $d = (|H|,|T|)$. Then $d$ divides $(q-1)^2$ and we note that $T = (q-1) \times (q^7-1)$ is contained in a subgroup $M$ of type ${\rm SL}_2(q) \times E_7(q)$. If $x$ has fixed points, then $|x|$ divides $d$. But $T<C_M(S)$ for every subgroup $S<T$ of order dividing $d$, so $T < C_M(x)$ and we reach a contradiction. The other three cases are very similar and we omit the details (in every case, we can embed $T = C_G(x)$ in a subgroup of type ${\rm SL}_2(q) \times E_7(q)$). By inspecting \cite{Lub}, we see that the contribution to $\delta(G,H)$ from these elements is at least
\[
\frac{q(q^6-1)}{28}\left(\frac{q-2}{\Phi_1^2\Phi_7} + \frac{q}{\Phi_2^2\Phi_{14}} +\frac{q}{\Phi_1\Phi_2\Phi_7} +\frac{q-2}{\Phi_1\Phi_2\Phi_{14}}\right).
\]

Putting all of this together, we deduce that $\delta(G,H) > 89/325$ if $q \geqs 7$, so we may assume $q \leqs 5$.  

We handle the remaining groups with $q \leqs 5$ by identifying some additional derangements. If $|C_G(x)| = \Phi_1\Phi_2\Phi_9$, then $(|H|,|C_G(x)|) < 30$ and thus $x$ is a derangement by Lemma \ref{l:cox}. And the same conclusion holds if $|C_G(x)| = \Phi_1\Phi_2\Phi_{18}$ and $q \leqs 4$. Therefore, we can add
\[
\frac{q(q^3+2)(q^3-1)(q-1)}{36\Phi_1\Phi_2\Phi_9} + (1-\delta_{5,q})\frac{(q^3+2)(q^3+1)(q+1)(q-2)}{36\Phi_1\Phi_2\Phi_{18}}
\]
to our previous lower bound and this is good enough to force $\delta(G,H) > 89/325$ if $q \in \{3,4,5\}$. Finally, if $q =2$ and $|C_G(x)| = q^{8}-1$, then $(|H|,|C_G(x)|) = 3$ and so we can include an additional contribution of $q^4(q^4-2)/16(q^8-1)$. One can now check that the desired bound holds for $q=2$.

\vs

\noindent \emph{Case 4(b). $H= \Phi_{30}{:}30 = (q^8+ q^7 - q^5-q^4 - q^3 + q +1){:}30$}

\vs

We claim that each regular semisimple element $x \in G$ with 
\begin{equation}\label{e:list000}
|C_G(x)| \in \{ \Phi_{20},\, \Phi_{24}, \, \Phi_1\Phi_2\Phi_4\Phi_8,\, \Phi_1\Phi_2\Phi_4\Phi_{12},\, \Phi_1^2\Phi_7, \, \Phi_2^2\Phi_{14}, \, \Phi_1\Phi_2\Phi_7, \, \Phi_1\Phi_2\Phi_{14}\}
\end{equation}
is a derangement. To see this, set $T = C_G(x)$ and observe that $d = (|H|,|T|)$ divides $30$ in every case. Therefore, in view of Lemma \ref{l:cox}, we immediately reduce to the case where $d = |x| = 30$. In particular, $x$ is a derangement if $|C_G(x)| = \Phi_{20}$ or $\Phi_{24}$ since $d \leqs 5$. 

In each of the remaining cases, we can embed $T$ in a maximal rank subgroup $M$ of $G$ with the property that $T<C_M(S)$ for every order $d$ subgroup $S$ of $T$. So if $x$ has fixed points, then $T<C_M(x) \leqs C_G(x) = T$, which is absurd. 

For example, if $|T| = \Phi_1\Phi_2\Phi_4\Phi_8$ then $d$ divides $q^4-1$ and we can embed the cyclic group $T$ in a subgroup $M$ of type ${\rm SL}_9(q)$ such that $T<C_M(S)$, where $S$ is the unique order $d$ subgroup of $T$. Similarly, if $|T| = \Phi_1\Phi_2\Phi_4\Phi_{12}$ then $T$ is cyclic, $d$ divides $q^4-1$ and we proceed by embedding $T$ in $M = {\rm SU}_5(q^2)$. And in each of the four remaining cases, we can argue by embedding $T$ in a subgroup $M$ of type ${\rm SL}_2(q) \times E_7(q)$.

Using \cite{Lub} to calculate the total number of regular semisimple elements $x \in G$ with $|C_G(x)|$ as in \eqref{e:list000}, we deduce that $\delta(G,H) > 89/325$ for all $q \geqs 3$. Finally, for $q=2$ it is easy to check that $(|H|,|C_G(x)|)<30$ for all
\[
|C_G(x)| \in \{ \Phi_{15},\, \Phi_1\Phi_2\Phi_9, \, \Phi_1\Phi_2\Phi_{18},\, \Phi_1\Phi_2\Phi_4\Phi_5, \, \Phi_1\Phi_2\Phi_4\Phi_{10} \}
\]
and we get $\delta(G,H) > 89/325$ by including the additional contribution from these elements.

\vs

\noindent \emph{Case 4(c). $H= \Phi_{15}{:}30 = (q^8-q^7+q^5-q^4+q^3-q+1){:}30$}

\vs

This is entirely similar to the previous case, the only difference being that the elements with $|C_G(x)| = \Phi_{30}$ are derangements, rather than those with $|C_G(x)| = \Phi_{15}$. We omit the details.
\end{proof}

We are now ready to complete the proof of Proposition \ref{p:ex_sol}. Recall that this completes the proofs of Theorems \ref{t:main2} and \ref{t:main4}(ii) for exceptional groups.

\begin{proof}[Proof of Proposition \ref{p:ex_sol}]
In view of Lemmas \ref{l:ex_par} and \ref{l:ex_tor}, we may assume $H$ is neither a parabolic subgroup, nor the normaliser of a maximal torus. The relevant cases are labelled (a)-(f) in the proof of \cite[Proposition 7.1]{Bur21}, and we have already handled (a)-(c) in Lemma \ref{l:ex_small}.

First assume $(G,H) = (F_4(2), 3.{\rm U}_3(2)^2.3.2)$ or $({}^2E_6(2), 3.{\rm U}_3(2)^3.3^2.S_3)$. In both cases, the character table of $G$ is available in \cite{GAPCTL} and one can check that the crude bound in \eqref{e:crude1} implies that $\delta(G,H) > 1/2$. Using Lemma \ref{l:frob}, it is also easy to check that $G = C^2$, where $C$ is a conjugacy class of elements of order $17$.

Finally, let us assume $G = E_8(2)$ and $H = 3^2.{\rm U}_{3}(2)^4.3^2.{\rm GL}_2(3)$ is of type ${\rm SU}_3(2)^4$. As before, let $\Phi_i$ be the $i$-th cyclotomic polynomial evaluated at $q=2$. If $x \in G$ is any regular semisimple element with 
\[
|C_G(x)| \in \{ \Phi_{20}, \, \Phi_{24}, \, \Phi_{15}, \, \Phi_{30}, \, \Phi_2^2\Phi_{14},\, \Phi_1\Phi_2\Phi_7, \, \Phi_1\Phi_2\Phi_9, \, \Phi_1\Phi_2\Phi_{18}, \, \Phi_1\Phi_2\Phi_4\Phi_8 \}
\] 
then $(|H|,|C_G(x)|) < 30$ and thus $x$ is a derangement by Lemma \ref{l:cox}. By inspecting \cite{Lub}, we can calculate the total number $N$ of such elements and we check that $\delta(G,H) \geqs N/|G| > 89/325$. In addition, \cite[Theorem 7.7]{GM} implies that $G = \{1\} \cup CD$, where $C$ and $D$ are classes of elements of order $\Phi_{20} = 205$ and $\Phi_{24} = 241$, respectively. 
\end{proof}

\subsection{Classical groups}\label{ss:class_sol}

In order to complete the proof of Theorem \ref{t:main_sol}, we may assume $G$ is a classical group and so the possibilities for $G$ and $H$ are listed in Tables 16-19 of \cite{LZ}. Our main result is the following and we note that  this completes the proofs of Theorems \ref{t:main2} and \ref{t:main4}.

\begin{prop}\label{p:class_sol}
Let $G$ be a transitive finite simple classical group with soluble point stabiliser $H$. 

\vspace{1mm}

\begin{itemize}\addtolength{\itemsep}{0.2\baselineskip}
\item[{\rm (i)}] If $H \in \mathcal{S}^{+}$, then $\delta(G) > 89/325$. 
\item[{\rm (ii)}] If $H \in \mathcal{S}$, then $G = \Delta(G)^2$.
\end{itemize}
\end{prop}

Throughout this section, $G$ will denote a finite simple classical group over $\mathbb{F}_q$, where $q=p^f$ and $p$ is a prime. We will write $V$ for the natural module and we set $n = \dim V$. In view of isomorphisms between some of the low dimensional groups, we may assume $G$ is one of the following:
\begin{equation}\label{e:class_list}
{\rm L}_n(q)  \, (n \geqs 2), \; {\rm U}_n(q) \, (n \geqs 3), \; 
{\rm PSp}_n(q) \, (n \geqs 4), \; {\rm P\O}_{n}^{\e}(q) \, (n \geqs 7)
\end{equation}
In addition, we may exclude the following groups
\begin{equation}\label{e:omit}
{\rm L}_2(q) \, (q \leqs 9), \, {\rm L}_4(2), \, {\rm U}_3(3), \, {\rm Sp}_4(2)', \, {\rm PSp}_4(3)
\end{equation}
due to the existence of the following isomorphisms (see \cite[Proposition 2.9.1]{KL}, for example):
\[
\def\arraystretch{1.5}
\begin{array}{c}
{\rm L}_2(4) \cong {\rm L}_2(5) \cong A_5,\; {\rm L}_2(7) \cong {\rm L}_3(2),\; {\rm L}_2(8) \cong {}^2G_2(3)',\;  {\rm L}_2(9) \cong {\rm Sp}_4(2)' \cong A_6, \\
{\rm L}_4(2) \cong A_8, \; {\rm U}_3(3) \cong G_2(2)', \; {\rm PSp}_4(3) \cong {\rm U}_4(2)
\end{array}
\]

\subsubsection{Two-dimensional linear groups}

\begin{lem}\label{l:psl2_sol}
The conclusion to Proposition \ref{p:class_sol} holds when $G = {\rm L}_2(q)$.
\end{lem}

\begin{proof}
As explained above, we may assume $q \geqs 11$. By Theorem \ref{t:rankone} we have $G = \Delta(G)^2$, so we just need to show that $\delta(G,H) > 89/325$. 
The possibilities for $H$ are recorded in \cite[Table 8.1]{BHR} and we refer the reader to \cite[Section 38]{Dorn} for the character table of $G$. Set $d = (2,q-1)$. 

First assume $H = P_1$ is a Borel subgroup, which allows us to identify $\O=G/H$ with the set of $1$-dimensional subspaces of $V$. Then $x \in G$ is a derangement if and only if it acts irreducibly on $V$, so either $x$ is semisimple and $|C_G(x)| = (q+1)/d$, or $q \equiv 3 \imod{4}$ and $x$ is an involution. As a consequence, we deduce that 
\begin{equation}\label{e:psld}
\delta(G,H) = \frac{q-1+\delta_{2,p}}{2(q+1)} \geqs \frac{5}{12}
\end{equation}
for all $q \geqs 11$ and the result follows.

Next assume $H$ is of type ${\rm GL}_1(q) \wr S_2$, so $H = D_{d(q-1)}$. Here $x \in G$ is a derangement if and only if $x$ is semisimple with $|C_G(x)| = (q+1)/d$, or if $|x| = p$ is odd. So for $p=2$ we see that \eqref{e:psld} holds, while we get 
\[
\delta(G,H) \geqs \frac{q^2-1}{|G|} + \frac{q-3}{2(q+1)} = \frac{q^2+q+4}{2q(q+1)} > \frac{1}{2}
\]
if $p >2$. Similarly, if $H$ is of type ${\rm GL}_1(q^2)$ then $H = D_{d(q+1)}$ and $x \in G$ is a derangement if and only if $x$ is semisimple with $|C_G(x)| = (q-1)/d$, or if $|x| = p$ is odd. So for $p=2$ we get
\[
\delta(G,H) = \frac{q-2}{2(q-1)} \geqs \frac{7}{15}
\]
for all $q \geqs 16$, and for $p>2$ and $q \geqs 11$ we compute
\[
\delta(G,H) \geqs \frac{q^2-1}{|G|} + \frac{q-5}{2(q-1)} = \frac{q^2-q-4}{2q(q-1)} \geqs \frac{53}{110}.
\]

Next assume $q = 3^k$ and $H$ is a subfield subgroup of type ${\rm GL}_2(3)$, where $k$ is an odd prime. Here $H = {\rm L}_2(3) \cong A_4$ and $x \in G$ is a derangement if and only if $|x| > 3$, so we get
\[
\delta(G,H) = \frac{1}{|G|}\left(|G| - \frac{|G|}{q+1} - q^2\right) = \frac{q(q-3)}{q^2-1} \geqs \frac{81}{91}.
\]

Finally, suppose $q = p \geqs 11$ and $H$ is of type $2^{1+2}_{-}.{\rm O}_{2}^{-}(2)$, so $H \in \{A_4,S_4\}$, with $H = S_4$ if and only if $q \equiv \pm 1 \imod{8}$  (see \cite[Table 8.1]{BHR} for the precise conditions on $q$ required for maximality). Here the set of nontrivial elements with fixed points comprise the unique classes of elements of order $2$ and $3$, plus one class of elements of order $4$ if $H = S_4$. This implies that
\[
\delta(G,H) \geqs \frac{1}{|G|}\left(|G| - 1 - \frac{5}{2}q(q+1)\right) = \frac{q^3-5q^2-6q-2}{q(q^2-1)} > \frac{1}{2}
\]
for all $q \geqs 13$ (with equality if $q \equiv 1 \imod{24}$, for example). If $q = 11$, then $H = A_4$ and we compute $\delta(G) = 32/55$.
\end{proof}

For the remainder, we may assume $G \ne {\rm L}_2(q)$. We now divide the possibilities for $H$ according to whether or not $H$ is a parabolic subgroup. The cases where $H$ is parabolic will be handled in Section \ref{sss:parab} and the possibilities that arise are recorded in \cite[Lemma 5.4]{Bur21}; they comprise a handful of ``sporadic" cases, involving certain low-dimensional groups defined over small fields (see Table \ref{tab:parab1}), together with three infinite families $(G,H)$ where $G = {\rm L}_3^{\e}(q)$ or ${\rm Sp}_4(q)$ and $H$ is a Borel subgroup. The remaining non-parabolic subgroups are listed in \cite[Lemma 6.2]{Bur21} and they will be treated in Section \ref{sss:nparab}.

\subsubsection{Parabolic subgroups}\label{sss:parab}

In this section we handle the case where $G \ne {\rm L}_2(q)$ and $H$ is a parabolic subgroup. We fix our notation for parabolic subgroups: 

\vspace{1mm}

\begin{itemize}\addtolength{\itemsep}{0.2\baselineskip}
\item[{\rm (a)}] We write $P_k$ for the stabiliser in $G$ of a $k$-dimensional totally singular subspace of $V$ (if $G = {\rm L}_n(q)$ then all subspaces of $V$ are totally singular). 

\item[{\rm (b)}] If $G = {\rm L}_n(q)$ and $1 \leqs k < n/2$, then $P_{k,n-k}$ denotes the stabiliser of a flag $0 < U < W< V$, where $\dim U = k$ and $\dim W = n-k$.

\item[{\rm (c)}] For $G = {\rm P\O}_8^{+}(q)$, we define $P_{1,3,4}$ to be the image of the parabolic subgroup
\[
[q^{11}]{:}\left(\frac{q-1}{d}\right)^2.\frac{1}{d}{\rm GL}_2(q).d^2 < \O_8^{+}(q),
\]
where $d= (2,q-1)$ and $\frac{1}{d}{\rm GL}_2(q)$ is the unique subgroup of ${\rm GL}_2(q)$ with index $d$.

\item[{\rm (d)}] If $G = {\rm Sp}_4(q)$ with $q \geqs 4$ even, then $P_{1,2} = [q^4]{:}(q-1)^2$ is a Borel subgroup of $G$ (that is to say, it is the normaliser of a Sylow $2$-subgroup of $G$).
\end{itemize}

None of the subgroups $H$ in (b), (c) or (d) are maximal in $G$, but they are of the form $H = M \cap G$ for some maximal subgroup $M$ of an almost simple group with socle $G$.

The possibilities for $G$ and $H$ with $H \in \mathcal{S}^{+}$ are recorded in \cite[Lemma 5.4]{Bur21} and we begin by handling the cases arising in part (i) of this lemma. These are the cases listed in Table \ref{tab:parab1}, recalling that we may (and do) exclude the classical groups in \eqref{e:omit}.

\begin{lem}\label{l:parab1}
The conclusion to Proposition \ref{p:class_sol} holds if $(G,H)$ is one of the cases in Table \ref{tab:parab1}. 
\end{lem}

{\scriptsize
\begin{table}
\[
\begin{array}{lll} \hline
G & \mbox{Type of $H$} & \delta(G) \\ \hline
{\rm L}_3(2) & P_1, \, P_2,\, P_{1,2}  & 0.285,\, 0.285,\, 0.619 \\
{\rm L}_3(3) & P_1, \, P_2,\, P_{1,2}  & 0.307,\, 0.307,\, 0.682 \\
{\rm L}_4(3) & P_2, \, P_{1,3}  & 0.507,\, 0.623 \\
{\rm L}_5(2) & P_{2,3}  & 0.679 \\
{\rm L}_5(3) & P_{2,3} & 0.731  \\
{\rm L}_6(2) & P_{2,4} & 0.797  \\
{\rm L}_6(3) & P_{2,4} & 0.848 \\
{\rm U}_4(2) & P_1 &  0.422 \\
{\rm U}_4(3) & P_1 & 0.485  \\
{\rm U}_5(2) & P_1 & 0.455  \\
{\rm Sp}_6(2) & P_2 & 0.453  \\
{\rm PSp}_6(3) & P_2 & 0.496  \\
\O_7(3) & P_2 & 0.524 \\
\O_{8}^{+}(2) & P_2, \, P_{1,3,4} & 0.613, \, 0.804 \\
{\rm P\O}_8^{+}(3) & P_2, \, P_{1,3,4} & 0.656, \, 0.839 \\ \hline
\end{array}
\]
\caption{The groups $(G,H)$ in Lemma \ref{l:parab1}}
\label{tab:parab1}
\end{table}
}

\begin{proof}
This is a straightforward {\sc Magma} computation. In each case, we construct $G$ and $H$ in terms of a permutation representation of $G$ of minimal degree, and by computing conjugacy classes we determine the set $\Delta(G)$ of derangements with respect to the action of $G$ on $G/H$. This allows us to compute $\delta(G)$, which is recorded in the final column of Table \ref{tab:parab1} to $3$ significant figures. If $\delta(G) > 1/2$ then $G = \Delta(G)^2$ by Lemma \ref{l:easy2}. Otherwise, we use {\sc Magma} to construct the character table of $G$ and we apply  Lemma \ref{l:frob} to check that either $G = {\rm L}_3(2)$ and $H \in \{P_1,P_2\}$, or there exist classes $C,D$ of derangements such that $G = \{1\} \cup CD$. Note that in the former case we have $G \cong {\rm L}_2(7)$, $H \cong S_4$ and thus $G = C^2 \cup CD$ for the two classes $C,D$ of elements of order $7$ (see Proposition \ref{p:rankone}(iv)).
\end{proof}

To complete the proof for parabolic subgroups, it just remains to consider the three infinite families recorded in \cite[Lemma 5.4(ii)]{Bur21}:

\vspace{1mm}

\begin{itemize}\addtolength{\itemsep}{0.2\baselineskip}
\item[{\rm (a)}] $G = {\rm L}_3(q)$, $H = P_{1,2}$ and $q \geqs 4$;
\item[{\rm (b)}] $G = {\rm U}_3(q)$, $H = P_1$ and $q \geqs 4$;
\item[{\rm (c)}] $G = {\rm Sp}_4(q)$, $H = P_{1,2}$ and $q \geqs 4$ even.
\end{itemize}

\vspace{1mm}

Notice that in each case, $H = N_G(P)$ is a Borel subgroup, where $P$ is a Sylow $p$-subgroup of $G$. Also note that $H$ is non-maximal in $G$ in cases (a) and (c).

\begin{lem}\label{l:psl3}
The conclusion to Proposition \ref{p:class_sol} holds if $G = {\rm L}_3(q)$ and $H = P_{1,2}$.
\end{lem}

\begin{proof}
Set $d = (3,q-1)$, $e = (q^2+q+1)/d$ and note that $H \in \mathcal{S}^{+} \setminus \mathcal{S}$, so it suffices to show that $\delta(G) > 89/325$. 

Let $\chi$ be the corresponding permutation character and note that the character table of $G$ is available in the literature (see \cite[Table 2]{SF}). As observed in the proof of \cite[Lemma 5.6]{Bur21}, we have 
\[
\chi = 1 + {\rm St} + 2\psi
\]
where ${\rm St} = \chi_{q^3}$ is the Steinberg character and $\psi = \chi_{q(q+1)}$ in the notation of \cite{SF}. From the character table, we deduce that $x \in G$ is a derangement if and only if $x$ is regular semisimple with $|C_G(x)| = (q^2-1)/d$ or $e$. Since there are $(e-1)/2-(q-1)/d-(3-d)/2$ classes of the first type, and $(e-1)/3$ classes of the second, it follows that 
\[
\delta(G) = \left(\frac{1}{2}(e-1)-\frac{q-1}{d}-\frac{3-d}{2}\right)\frac{d}{q^2-1}+\frac{1}{3}(e-1)\frac{1}{e} \geqs \frac{24}{35}
\]
for all $q \geqs 4$.
\end{proof}

\begin{lem}\label{l:psu3}
The conclusion to Proposition \ref{p:class_sol} holds if $G = {\rm U}_3(q)$ and $H = P_{1}$.
\end{lem}

\begin{proof}
Set $d = (3,q+1)$, $e = (q^2-q+1)/d$ and let $\chi$ be the permutation character. Then $\chi = 1+{\rm St}$, where ${\rm St}$ is the Steinberg character, and by inspecting the character table of $G$ (see \cite[Table 2]{SF}) we deduce that $x \in G$ is a derangement if and only if $|C_G(x)| = (q+1)^2/d$ or $e$, or if $|x| = d = 3$ and $|C_G(x)| = (q+1)^2$. Therefore,
\[
\delta(G) = \frac{1}{6}(e-1)\frac{d}{(q+1)^2} + \frac{1}{3}(e-1)\frac{1}{e} + \delta_{3,d}\frac{1}{(q+1)^2} \geqs \frac{126}{325}
\]
for all $q \geqs 4$. Finally, Theorem \ref{t:rankone} implies that $G = \{1\} \cup C^2$ for a conjugacy class $C$ of derangements, whence $G = \Delta(G)^2$.
\end{proof}

\begin{lem}\label{l:sp4}
The conclusion to Proposition \ref{p:class_sol} holds if $G = {\rm Sp}_4(q)$ and $H = P_{1,2}$.
\end{lem}

\begin{proof}
Here $q \geqs 4$ and $H$ is non-maximal in $G$, so we just need to establish the bound $\delta(G) > 89/325$. If $x \in G$ is a regular semisimple element with $|C_G(x)| = (q+1)^2$ or $q^2+1$, then $(|H|,|C_G(x)|) = 1$ and thus $x$ is a derangement. The number of distinct conjugacy classes of such elements is recorded in \cite[Table IV-1]{Eno2} and we deduce that
\[
\delta(G) \geqs \frac{q(q-2)}{8(q+1)^2} + \frac{q^2}{4(q^2+1)} \geqs \frac{117}{425}> \frac{89}{325}
\]
for all $q \geqs 4$.
\end{proof}

\subsubsection{Non-parabolic subgroups}\label{sss:nparab}

In order to complete the proof of Proposition \ref{p:class_sol}, we may assume $H$ is a non-parabolic subgroup. The possibilities for $G$ and $H$ are described in  \cite[Lemma 6.2]{Bur21} and we begin by handling the ``sporadic" cases recorded in Table \ref{tab:nparab}, which involve certain low dimensional groups with $n \leqs 16$ and $q \leqs 4$. Note that in the third column of the table we indicate whether or not $H$ is a maximal subgroup of $G$ (this follows from the information in \cite{BHR,KL}).

\begin{lem}\label{l:nparab1}
The conclusion to Proposition \ref{p:class_sol} holds if $(G,H)$ is one of the cases in Table \ref{tab:nparab}.
\end{lem}

{\scriptsize
\begin{table}
\[
\begin{array}{llcl} \hline
G & \mbox{Type of $H$} & \mbox{Maximal} & \delta(G) \\ \hline
{\rm L}_3(2) & {\rm GL}_2(2) \oplus {\rm GL}_1(2) & \texttt{n} & 0.535 \\
{\rm L}_3(3) & {\rm GL}_2(3) \oplus {\rm GL}_1(3), \, {\rm O}_3(3) & \texttt{n},\, \texttt{y} & 0.418, \, 0.742 \\
{\rm L}_3(4) & {\rm GU}_3(2) & \texttt{y} & 0.685 \\
{\rm L}_4(3) & {\rm GL}_2(3) \wr S_2,\, {\rm O}_4^{+}(3) & \texttt{n}, \, \texttt{y} & 0.649, \,  0.804 \\
{\rm L}_6(3) & {\rm GL}_2(3) \wr S_3 & \texttt{y} & 0.907 \\
{\rm L}_8(3) & {\rm GL}_2(3) \wr S_4 & \texttt{y} & 0.976 \\
{\rm U}_4(2) & {\rm GU}_3(2) \perp {\rm GU}_1(2) & \texttt{y} & 0.418 \\
{\rm U}_4(3) & {\rm GU}_2(3) \wr S_2 & \texttt{y} & 0.646 \\
{\rm U}_5(2) & {\rm GU}_3(2) \perp {\rm GU}_2(2) & \texttt{y} & 0.538 \\
{\rm U}_6(2) & {\rm GU}_3(2) \wr S_2 & \texttt{y} & 0.583 \\
{\rm U}_6(3) & {\rm GU}_2(3) \wr S_3 & \texttt{y} & 0.883 \\
{\rm U}_8(3) & {\rm GU}_2(3) \wr S_4 & \texttt{y} & 0.971 \\
{\rm U}_9(2) & {\rm GU}_3(2) \wr S_3 & \texttt{y} & 0.897 \\
{\rm U}_{12}(2) & {\rm GU}_3(2) \wr S_4 & \texttt{y} & 0.947 \\
{\rm PSp}_6(3) & {\rm Sp}_2(3) \wr S_3 & \texttt{y} & 0.644 \\
{\rm PSp}_8(3) & {\rm Sp}_2(3) \wr S_4 & \texttt{y} & 0.824 \\
\O_7(3) & {\rm O}_4^{+}(3) \perp {\rm O}_3(3) & \texttt{y} & 0.644 \\
\O_8^{+}(2) & {\rm O}_2^{-}(2) \times {\rm GU}_3(2) & \texttt{n} & 0.744 \\
{\rm P\O}_{8}^{+}(3) & {\rm O}_4^{+}(3) \wr S_2 & \texttt{y} & 0.769 \\ 
{\rm P\O}_{12}^{+}(3) & {\rm O}_4^{+}(3) \wr S_3 & \texttt{y} & 0.955 \\
{\rm P\O}_{16}^{+}(3) & {\rm O}_4^{+}(3) \wr S_4 & \texttt{y} &  0.993 \\ \hline
\end{array}
\]
\caption{The groups $(G,H)$ in Lemma \ref{l:nparab1}}
\label{tab:nparab}
\end{table}
}

\begin{proof}
All of these cases can be handled using {\sc Magma} \cite{magma}. In the final column of Table \ref{tab:nparab} we calculate $\delta(G)$ to $3$ significant figures.

First assume $H$ is maximal in $G$ and write $G = L/Z$ and $H = J/Z$, where $L$ is one of the quasisimple matrix groups ${\rm SL}_n^{\e}(q)$, ${\rm Sp}_n(q)$, $\O_n^{\e}(q)$ and $Z = Z(L)$. In this situation, working with the standard matrix representation of $L$, we can use the function \textsf{ClassicalMaximals} to construct $J$ and we then compute $\delta(L,J) = \delta(G)$ by inspecting the  conjugacy classes in $L$ and $J$.

There are four remaining cases where $H$ is non-maximal. Here $H = M \cap G$, where $M$ is maximal in an almost simple group $A$ with socle $G$. To handle these cases, we use the function \textsf{AutomorphismGroupSimpleGroup} to construct a permutation representation of $A$ and we then construct $M$, and hence $H$, by using the function \textsf{MaximalSubgroups}. We then compute $\delta(G)$ by considering the conjugacy classes in $G$ and $H$.

If $\delta(G)>1/2$ then $G = \Delta(G)^2$ by Lemma \ref{l:easy2}. Otherwise, $(G,H) = ({\rm L}_3(3),{\rm GL}_2(3))$ or $({\rm U}_4(2),{\rm GU}_3(2))$ and by  working with the character table of $G$ and Lemma \ref{l:frob} it is easy to show that $G = \{ 1 \} \cup CD$ for certain conjugacy classes $C,D$ of derangements. 
\end{proof}

Finally, let us turn to the remaining infinite families $(G,H)$ with $H$ non-parabolic, which are recorded in Table \ref{tab:nparab2} (see \cite[Lemma 6.2(v)]{Bur21}). We  consider each case in turn.

{\scriptsize
\begin{table}
\[
\begin{array}{lll} \hline
G & \mbox{Type of $H$} & \mbox{Conditions} \\ \hline
{\rm L}_{n}^{\e}(q) & {\rm GL}_1^{\e}(q^n) & \mbox{$n \geqs 3$ prime} \\  
 & {\rm GL}_1^{\e}(q) \wr S_n & n \in \{3,4\} \\ 
 & 3^{1+2}.{\rm Sp}_2(3) & \mbox{$n=3$, $q = p \equiv \e \imod{3}$} \\
 & {\rm GU}_3(2) & \mbox{$(n,\e) = (3,-)$, $q = 2^k$, $k \geqs 3$ prime} \\ 
{\rm Sp}_4(q) & {\rm O}_2^{\e}(q) \wr S_2 & \mbox{$q \geqs 4$ even} \\ 
& {\rm O}_2^{-}(q^2) & \mbox{$q \geqs 4$ even} \\ 
{\rm P\O}_{8}^{+}(q) & {\rm O}_{2}^{-}(q^2) \times {\rm O}_{2}^{-}(q^2) & \\
& {\rm O}_{2}^{\e}(q) \wr S_4 & \\ \hline
\end{array}
\]
\caption{The infinite families with $H$ non-parabolic}
\label{tab:nparab2}
\end{table}
}

\begin{lem}\label{l:singer0}
The conclusion to Proposition \ref{p:class_sol} holds if $G = {\rm L}_{n}^{\e}(q)$ and $H$ is of type ${\rm GL}_1^{\e}(q^n)$, where $n \geqs 3$ is prime.
\end{lem}

\begin{proof}
Here $H = N_G(T) = T.n$, where $T$ is a cyclic maximal torus of order $(q^n-\e)/d(q-\e)$ with $d = (n,q-\e)$. If $G = {\rm L}_3(2)$ then it is easy to check that $\delta(G) = 3/8$ and $G = \{1\} \cup CD$, where $C$ and $D$ are the unique conjugacy classes of elements of order $2$ and $4$, respectively. In each of the remaining cases we claim that $\delta(G) > 1/2$, so $G = \Delta(G)^2$ by Lemma \ref{l:easy2}.

 First observe that every non-identity element $x \in H$ is regular as an element of $G$. More precisely, either $x \in T$ and $C_G(x) = T$, or $x$ has order $n$ and we calculate that $G$ contains at most $d|G|/c$ regular elements of order $n$, where $c = (q-1)^{n-1}$ if $\e=+$ and $d=n$, otherwise $c = (q^2-1)^{(n-1)/2}$. For example, if $n$ divides $q$, then every regular element in $G$ of order $n$ has Jordan form $(J_n)$ on the natural module for $G$ and we deduce that there are 
\[
\frac{|{\rm SL}_{n}^{\e}(q)|}{q^{n-1}} = \frac{d|G|}{q^{n-1}} < \frac{d|G|}{(q^2-1)^{(n-1)/2}}
\]
such elements in $G$.

Since there are exactly $(|T|-1)/n$ distinct $G$-classes of regular semisimple elements $x \in G$ with $|C_G(x)|=|T|$ (see \cite[Example (a), p.484]{FJ93}, for example), it follows that  
\[
\delta(G) \geqs \frac{1}{|G|}\left(|G|-1-\frac{d|G|}{c} -\frac{|T|-1}{n}\cdot \frac{|G|}{|T|}\right) = 1 - \frac{1}{|G|}-\frac{d}{c}-\frac{1}{n}\left(1-\frac{1}{|T|}\right)
\]
and thus $\delta(G)> 1/2$ unless $G = {\rm L}_3(3)$. In the latter case, we compute 
$\delta(G) = 251/432$.
\end{proof}

\begin{lem}\label{l:nparab2}
The conclusion to Proposition \ref{p:class_sol} holds if $G = {\rm L}_{3}^{\e}(q)$ and $H$ is of type ${\rm GL}_1^{\e}(q) \wr S_3$, $3^{1+2}.{\rm Sp}_2(3)$ or ${\rm GU}_3(2)$.
\end{lem}

\begin{proof}
Set $d = (3,q-\e)$ and $e = (q^2 + \e q + 1)/d$. 

First assume $H$ is of type ${\rm GL}_1^{\e}(q) \wr S_3$ and note that $q \geqs 5$ if $\e=+$ (see \cite[Table 8.3]{BHR}). If $x \in G$ is regular semisimple with $|C_G(x)| = e$, then $(|H|,|C_G(x)|) = 1$ and thus $x$ is a derangement. As noted in \cite{SF}, there are precisely $(e-1)/3$ conjugacy classes of such elements and thus 
\[
\delta(G) \geqs \frac{e-1}{3e} \geqs \frac{2}{7}.
\]
In addition, \cite[Theorems 7.1 and 7.3]{GM} imply that $G = \{1\} \cup C^2$, where $C$ is any class of elements of order $e$.

The case $H = 3^{1+2}.{\rm Sp}_2(3)$ is entirely similar. Here $q = p \equiv \e \imod{3}$, $H \leqs {\rm ASL}_2(3)$ and every regular semisimple element with $|C_G(x)| = e$ is a derangement. Similarly, if $H$ is of type ${\rm GU}_3(2)$, then $\e=-$, $q = 2^k$ for some odd prime $k$ and the same argument applies since we have $H \leqs {\rm PGU}_3(2) \cong {\rm ASL}_2(3)$.
\end{proof}

\begin{lem}\label{l:nparab3}
The conclusion to Proposition \ref{p:class_sol} holds if $G = {\rm L}_{4}(q)$ and $H$ is of type ${\rm GL}_1(q) \wr S_4$.
\end{lem}

\begin{proof}
Set $d = (4,q-1)$ and note that $q \geqs 5$ (see \cite[Table 8.8]{BHR}).

First assume $x \in G$ is regular semisimple with $|C_G(x)| = (q^4-1)/d(q-1)$. We claim that $x$ is a derangement. Seeking a contradiction, suppose $x \in H$ and note that  $(|H|,|C_G(x)|)$ divides $q+1$, which implies that $|x|$ divides $q+1$. 
Since $x$ is regular, it follows that $x$ is conjugate to the image of a block-diagonal matrix of the form ${\rm diag}(A,B) \in {\rm SL}_4(q)$, where $A,B \in {\rm GL}_2(q)$ are irreducible, with distinct eigenvalues in $\mathbb{F}_{q^2}$. But this implies that $|C_G(x)| = (q+1)(q^2-1)/d$ and we have reached a contradiction. This justifies the claim. From \cite[Example (a), p.484]{FJ93}, we see that $G$ has exactly  
\[
\frac{1}{4d}(q+1)(q^2+1-e)
\]
such conjugacy classes, where $e = (2,q-1)$, whence
\[
\frac{(q+1)(q^2+1-e)/4d}{(q^4-1)/d(q-1)} = \frac{1}{4} - \frac{e}{4(q^2+1)}
\]
is the contribution to $\delta(G)$ from these elements.

Next suppose $y \in G$ is regular semisimple with $|C_G(y)| = (q^3-1)/d$ and note that $(|H|,|C_G(x)|)$ divides $e'(q-1)$, where $e'=(3,q-1)$. If $e'=1$ then by repeating the argument above we deduce that $y$ is a derangement (indeed, we have $|C_G(z)| = (q-1)^3/d$ for every regular semisimple element $z \in G$ such that $|z|$ divides $q-1$). Now assume $e'=3$. As before, $y$ is a derangement if $|y|$ divides $q-1$, so we may assume that $|y| = 9m$ for some divisor $m$ of $(q-1)/3$. To analyse this situation, it will be convenient to switch to the matrix groups $L = {\rm SL}_4(q)$ and $J = (q-1)^3{:}S_4$, where $J$ is the stabiliser in $L$ of a direct sum decomposition
$V = \la v_1 \ra \oplus \la v_2 \ra \oplus \la v_3 \ra \oplus \la v_4 \ra$ of the natural module.

Write $\mathbb{F}_{q}^{\times} = \la \omega \ra$. In terms of the basis $\{v_1,v_2,v_3,v_4\}$ for $V$, we observe that each $a \in J$ with $|C_L(a)| = q^3-1$ is $J$-conjugate to an element of the form
\[
\left(\begin{array}{cccc}
\omega^i & & & \\
& 1 & & \\
& & 1 & \\
& & & \omega^{-i} \end{array}\right)(1,2,3) = \left(\begin{array}{cccc}
 & & \omega^i & \\
1 &  & & \\
& 1 &  & \\
& & & \omega^{-i} \end{array}\right) \in J,
\]
where $1 \leqs i <q$ is indivisible by $3$. Since $3$ divides $q-1$, there are precisely $2(q-1)/3$ such $J$-classes, none of which are fused in $L$.

By considering rational canonical forms, we see that each $L$-class $a^L$ with $|C_L(a)| = q^3-1$ corresponds to a monic irreducible cubic polynomial over $\mathbb{F}_q$. So there are precisely
\[
\frac{1}{3}\sum_{m|3}\mu(m)q^{\frac{3}{m}} = \frac{1}{3}q(q^2-1)
\]
such classes, where $\mu$ is the M\"{o}bius function, and thus the contribution to $\delta(G)$ from the regular semisimple elements $y$ with $|C_G(y)| = (q^3-1)/d$ is at least 
\[
\frac{q(q^2-1)/3-2(q-1)/3}{q^3-1} = \frac{1}{3} - \frac{1}{q^2+q+1}.
\]
Therefore,
\[
\delta(G) \geqs \frac{1}{4} - \frac{e}{4(q^2+1)} + \frac{1}{3} - \frac{1}{q^2+q+1} \geqs \frac{643}{1209} > \frac{1}{2}
\]
for all $q \geqs 5$ and the result follows.
\end{proof}

\begin{lem}\label{l:nparab30}
The conclusion to Proposition \ref{p:class_sol} holds if $G = {\rm U}_{4}(q)$ and $H$ is of type ${\rm GU}_1(q) \wr S_4$.
\end{lem}

\begin{proof}
This is essentially identical to the previous case. Set $d = (4,q+1)$ and first consider a regular semisimple element $x \in G$ with $|C_G(x)| = (q^4-1)/d(q+1)$. Since $(|H|,|C_G(x)|)$ divides $q-1$ we deduce that $x$ is a derangement (indeed, if $z \in G$ is regular semisimple and $|z|$ divides $q-1$, then $|C_G(z)| = (q^2-1)(q-1)/d$) and we calculate that there are precisely
\[
\frac{1}{4d}(q-1)(q^2+1-e)
\]
such classes in $G$, where $e = (2,q-1)$.

Next set $L = {\rm SU}_4(q)$ and $J = (q+1)^3{:}S_4$. Then $L$ has exactly $q(q^2-1)/3$ regular semisimple classes $y^L$ with $|C_L(y)| = q^3+1$, all of which contain derangements if $q+1$ is indivisible by $3$. However, if $q \equiv -1 \imod{3}$, then by arguing as in the proof of Lemma \ref{l:nparab3} we deduce that $2(q+1)/3$ of these classes meet $J$. So it follows that the contribution to $\delta(G)$ from regular semisimple elements $y \in G$ with $|C_G(y)| = (q^3+1)/d$ is at least
\[
\frac{q(q^2-1)/3 - 2(q+1)/3}{q^3+1} = \frac{1}{3} - \frac{1}{q^2-q+1}
\]
and this implies that
\[
\delta(G) \geqs \frac{1}{4} - \frac{e}{4(q^2+1)} + \frac{1}{3} - \frac{1}{q^2-q+1} \geqs \frac{47}{91}
\]
for all $q \geqs 5$. If $q=4$ we compute $\delta(G) = 3508723/4243200$, and similarly we get $\delta(G) = 20128/25515$ for $q=3$, hence $\delta(G) > 1/2$ for all $q \geqs 3$ and thus $G = \Delta(G)^2$ by Lemma \ref{l:easy2}. Finally, for $q=2$ we compute $\delta(G) = 31/80$ and it is straightforward to check that $G = C^2$, where $C$ is the unique conjugacy class of elements of order $5$.
\end{proof}

\begin{lem}\label{l:nparab6}
The conclusion to Proposition \ref{p:class_sol} holds if $G = {\rm Sp}_{4}(q)$ and $H$ is of type ${\rm O}_2^{\e}(q) \wr S_2$ or ${\rm O}_2^{-}(q^2)$, where $q \geqs 4$ is even.
\end{lem}

\begin{proof}
Here $H \in \mathcal{S}^{+}\setminus \mathcal{S}$ and so we just need to verify the bound $\delta(G) > 89/325$.

First assume $H$ is of type ${\rm O}_2^{\e}(q) \wr S_2$, so $H = D_{2(q-\e)} \wr S_2$. If $x \in G$ is a regular semisimple element with $|C_G(x)| = q^2-1$, then $(|H|,|C_G(x)|) = q-\e$ and we deduce that $x$ is a derangement (indeed, $|C_G(z)| = (q-\e)^2$ for every regular semisimple element $z \in G$ such that $|z|$ divides $q-\e$). By inspecting \cite[Table IV-1]{Eno2} we see that there are exactly $\frac{1}{2}q(q-2)$ such classes and thus 
\[
\delta(G) \geqs \frac{q(q-2)}{2(q^2-1)} \geqs \frac{8}{21}
\]
for all $q \geqs 8$. And if $q=4$ then $\e=-$ (see \cite[Table 8.14]{BHR}) and we compute $\delta(G) = 3481/4896$. 

Now assume $H$ is of type ${\rm O}_2^{-}(q^2)$, so $H = (q^2+1).4$. As above, each $x \in G$ with $|C_G(x)| = q^2-1$ is a derangement (indeed, we have $(|H|,|C_G(x)|) = 1$) and we deduce that $\delta(G) \geqs 8/21$ for $q \geqs 8$. Finally, for $q=4$ we compute $\delta(G) = 21011/28800$. 
\end{proof}

\begin{lem}\label{l:nparab8}
The conclusion to Proposition \ref{p:class_sol} holds if $G = {\rm P\O}_{8}^{+}(q)$ and $H$ is of type ${\rm O}_{2}^{-}(q^2) \times {\rm O}_{2}^{-}(q^2)$.
\end{lem}

\begin{proof}
Once again we have $H \in \mathcal{S}^{+}\setminus \mathcal{S}$ and so our goal is to show that $\delta(G) > 89/325$. Set $d=(2,q-1)$ and observe that
\[
H = N_G(P) = (D_{2(q^2+1)/d} \times D_{2(q^2+1)/d}).2^2,
\]
where $P$ is a Sylow $r$-subgroup of $G$ and $r$ is an odd prime divisor of $q^2+1$. For $q=2$ we compute $\delta(G) = 5740943/6967296$, so for the remainder we may assume $q \geqs 3$.

If $q$ is even, then every semisimple element $x \in H$ is contained in the torus $(q^2+1)^2$, which means that $|C_G(x)| = (q^2+1)^2$ for every regular semisimple element $x \in H$. The conjugacy classes in $G$ are available in \textsf{Chevie} \cite{Chevie} (also see \cite{Lub}) and it is a routine exercise to check that 
\[
\delta(G) \geqs \frac{1}{|G|}|\{x \in G \,:\, \mbox{$x$ regular semisimple, $|C_G(x)| \ne (q^2+1)^2$}\}| > \frac{1}{2}
\]
for all $q \geqs 4$.

Now assume $q$ is odd and let $x \in H$ be a regular semisimple element with $|C_G(x)| = (q^3-\e)(q-\e)/2$ and $\e = \pm$. Since $(|H|,|C_G(x)|) \in \{1,4,16\}$, Lemma \ref{l:cox} implies that $x$ is a derangement unless $|x| = 16$, which coincides with the order of a Sylow $2$-subgroup of $H$. But the Sylow $2$-subgroups of $H$ are non-cyclic, so $H$ does not contain any elements of order $16$. It follows that each $x \in G$ with $|C_G(x)| = (q^3-\e)(q-\e)/2$ is a derangement. 

Set $K = {\rm Spin}_8^{+}(q)$, so $G = K/Z$ with $Z = Z(K) = C_2 \times C_2$. By inspecting \cite{Lub}, we observe that $K$ has exactly 
\[
N = \frac{|K|}{(q^3-\e)(q-\e)} \cdot \frac{1}{6}q(q-\e)(q^2-1)
\]
elements $y$ with $|C_K(y)| = (q^3-\e)(q-\e)$. By considering the images of these elements, it follows that $G = K/Z$ has at least $N/4$ regular semisimple elements $x$ with $|C_G(x)| = (q^3-\e)(q-\e)/2$ and we deduce that 
\[
\delta(G) \geqs \frac{1}{6}\left(\frac{q(q-1)(q^2-1)}{(q^3-1)(q-1)}+\frac{q(q+1)(q^2-1)}{(q^3+1)(q+1)}\right) > \frac{89}{325}
\]
for all $q \geqs 3$.
\end{proof}

In order to complete the proof of Proposition \ref{p:class_sol}, we may assume 
$G = {\rm P\O}_{8}^{+}(q)$ and $H$ is of type ${\rm O}_{2}^{\e}(q) \wr S_4$.

\begin{lem}\label{l:nparab9}
The conclusion to Proposition \ref{p:class_sol} holds if $G = {\rm P\O}_{8}^{+}(q)$ and $H$ is of type ${\rm O}_{2}^{\e}(q) \wr S_4$.
\end{lem}

\begin{proof}
First observe that $q \geqs 5$ if $\e=+$ (see \cite[Table 8.50]{BHR}).

To begin with, we will assume $q$ is even. For $q \in \{2,4\}$ we have $e=-$ and with the aid of {\sc Magma} it is easy to check that $\delta(G)>1/2$. Now assume $q \geqs 8$.

If $x \in G$ is regular semisimple and $|C_G(x)| = (q+\e)(q^3+\e)$, then Lemma \ref{l:cox} implies that $x$ is a derangement since $(|H|,|C_G(x)|) \leqs 3$. From \cite{Lub} we read off that $G$ has precisely $q(q^2-1)(q+\e)/6$ conjugacy classes of such elements. 

Next let $y \in G$ be regular semisimple with $|C_G(y)| = q^4-1$, in which case $T = C_G(y)$ is cyclic and $(|H|,|T|)= (3,q+\e)(q-\e)$ divides $q^2-1$. Let $S$ be the unique subgroup of $T$ of order $q^2-1$ and note that we may embed $T$ in a subgroup $M = {\rm GL}_4(q)$. Then $T < C_M(S)$ and we deduce that $y$ is a derangement.

Since $G$ has $3q^2(q^2-2)/8$ distinct classes $y^G$ with $|C_G(y)| = q^4-1$ (see \cite{Lub}), it follows that 
\[
\delta(G) \geqs \frac{q(q^2-1)}{6(q^3+\e)} + \frac{3q^2(q^2-2)}{8(q^4-1)} > \frac{1}{2}
\]
for all $q \geqs 8$, hence $G = \Delta(G)^2$ by Lemma \ref{l:easy2}.

Now assume $q$ is odd. If $q = 3$ then $\e=-$ and $H = N_G(P)$, where $P$ is a normal subgroup of a Sylow $2$-subgroup of $G$ with $|P|=2^7$. Using {\sc Magma}, we compute $\delta(G)>1/2$. In the same way, one can check that $\delta(G)> 1/2$ when $q=5$, so for the remainder we may assume $q \geqs 7$. It will be convenient to work in the quasisimple group $L = \O_{8}^{+}(q)$, in which case $H$ is the image of the subgroup $J = \O_{2}^{\e}(q)^4.[2^6].S_4$ of $L$. In addition, set $K = {\rm Spin}_{8}^{+}(q)$ and $Z = Z(K) = C_2 \times C_2$, so $G = K/Z$ and $L = K/Z_1$ for some central subgroup $Z_1$ of order $2$.

Let $x \in L$ be a regular semisimple element with $|C_L(x)| = (q^4-1)/2$ or $q^4-1$. By arguing as above, we deduce that $x$ is a derangement. In addition, by inspecting \cite{Lub} we see that $K$ has exactly 
\[
N = \frac{|K|}{q^4-1} \cdot \frac{3}{8}(q^4-4q^2+3)
\]
elements $y$ with $|C_K(y)| = q^4-1$. It follows that there are at least $N/2$ regular semisimple elements $x \in L$ with $|C_L(x)| = q^4-1$ or $(q^4-1)/2$.

Next consider a regular semisimple element $y \in L$ with $|C_L(y)| = (q+\e)(q^3+\e)/2$ and note that $(|J|,|C_L(y)|)$ divides $(q+\e)^2/e$, where $e=(3,q+\e)$. Here $T = C_L(y) = \frac{1}{2}(q+\e) \times (q^3+\e)$ is non-cyclic and we may embed $T$ in a subgroup $M$ of type ${\rm GL}_4^{-\e}(q)$. Then $T<C_M(S)$ for every subgroup $S$ of $T$ such that $|S|$ divides $(q+\e)^2/e$, and it follows that $y$ is a derangement. As noted in the proof of Lemma \ref{l:nparab8}, there are at least 
\[
\frac{|L|}{(q^3+\e)(q+\e)} \cdot \frac{1}{6}q(q+\e)(q^2-1)
\]
such elements in $L$.

Putting this together, we deduce that
\[
\delta(G) \geqs \frac{3(q^4-4q^2+3)}{8(q^4-1)} + \frac{q(q^2-1)}{6(q^3+\e)} > \frac{1}{2}
\]
for all $q \geqs 7$. The result follows.
\end{proof}

\vs

This completes the proof of Proposition \ref{p:class_sol}. By combining this result with Propositions \ref{p:spor_sol}, \ref{p:alt_sol} and \ref{p:ex_sol}, we conclude that the proof of Theorem \ref{t:main_sol} is complete. In particular, we have now proved Theorems \ref{t:main2} and \ref{t:main4} in all cases.

\subsection{Primitive groups}\label{ss:prim_sol}
	
We now use Theorem \ref{t:main_sol}(i) to prove part (ii) of Theorem \ref{t:main3} (recall that Theorem \ref{t:main3}(i) was established in Section \ref{ss:prim}, so this will complete the proof of Theorem \ref{t:main3}).
	
\begin{proof}[Proof of Theorem \ref{t:main3}(ii)]
Let $G \leqs {\rm Sym}(\O)$ be a finite primitive permutation group with socle $N$ and soluble point stabiliser $H$. Recall that our goal is to show that $\delta(N) \geqs 89/325$, with equality if and only if $N = {}^2F_4(2)'$ and $H \cap N = 2^2.[2^8].S_3$.

As a consequence of the O'Nan-Scott theorem and our soluble point stabiliser hypothesis, $G$ is either affine, almost simple or a product type group $G \leqs L \wr S_b$, where $G$ acts on $\O = \Gamma^b$ via the product action and $L \leqs {\rm Sym}(\Gamma)$ is an almost simple primitive group with socle $S$ and soluble point stabilisers. 

If $G$ is affine, then $N$ is regular and $\delta(N) = 1 -|N|^{-1} \geqs 1/2$. And if $G$ is almost simple, then Theorem \ref{t:main_sol}(i) implies that $\delta(N) \geqs 89/325$, with equality if and only if $N = {}^2F_4(2)'$ and $H \cap N = 2^2.[2^8].S_3$. Finally, if $G \leqs L \wr S_b$ is a product type group, then $N = S^b$ with $S$ simple and by combining \eqref{e:sharp} with Theorem \ref{t:main_sol}(i) we deduce that 
\[
\delta(N) > \delta(S,\Gamma) \geqs \frac{89}{325}.
\]
The result follows.
\end{proof}

\section{Derangement generation}\label{s:gen}

Recall that every alternating group can be generated by two elements. For example, 
\[
A_n = \la (1,2,3), \, (\delta, \delta+1, \ldots, n) \ra,
\]
for all $n \geqs 4$, where $\delta = 1$ if $n$ is odd, otherwise $\delta = 2$. By a celebrated theorem of Steinberg \cite{St}, every finite simple group of Lie type can also be generated by a pair of elements, and the same is true for all the sporadic simple groups (see \cite{AG}). So this allows us to conclude, via CFSG, that every finite simple group is generated by two elements. In recent years, this theorem has been extended in many different directions and there is now a vast literature on a wide range of $2$-generation properties of simple groups (for example, see the survey articles \cite{B_sur, H_sur} for a sample of some of the remarkable results that have been established).

Now suppose $G \leqs {\rm Sym}(\O)$ is a finite transitive permutation group with  $|\O| \geqs 2$ and let $\Delta(G)$ be the set of derangements in $G$. The (normal) subgroup $D(G) = \la \Delta(G) \ra$ generated by $\Delta(G)$ has been studied by Bailey et al. in \cite{BCGR}, where several interesting results concerning the order and structure of the quotient group $G/D(G)$ are established. Of course, if $G$ is simple then $G = D(G)$ and it is natural to ask whether or not $G$ is generated by a pair of derangements. In this final section, our main theorem shows that this is indeed the case. Moreover, we will prove that such a group is generated by two \emph{conjugate} derangements. The following result is stated as Theorem \ref{t:main6} in Section \ref{s:intro}.

\begin{thm}\label{t:gen}
Let $G \leqs {\rm Sym}(\O)$ be a finite simple transitive permutation group. Then there exists a derangement $x \in G$ such that $G = \la x, x^g\ra$ for some $g \in G$.
\end{thm}

One of the key ingredients in our proof of Theorem \ref{t:gen} involves the concept of uniform spread. Following Breuer et al. \cite{BGK}, we say that a finite group $G$ has \emph{positive uniform spread} if there exists a conjugacy class $C = y^G$ with the property that for any non-identity $x \in G$, there exists an element $z \in C$ such that $G = \la x,z \ra$. In this situation, we call $y$, or the class $C$ itself, a \emph{witness}. The main theorem of \cite{GK} shows that every finite simple group has this remarkably strong $2$-generation property and we refer the reader to \cite{BGK,BGH,BH0,BH1} for more recent extensions and generalisations.

The next two lemmas will be applied repeatedly in the proof of Theorem \ref{t:gen}. In the first one, and for the remainder of the paper, we write 
\begin{equation}\label{e:My}
\mathcal{M}(y) = \{ H \,:\, \mbox{$H<G$ maximal and $y \in H$}\}
\end{equation}
for the set of maximal subgroups of $G$ containing the element $y \in G$. In addition, 
\[
{\rm fpr}(z,G/H) = \frac{|z^G \cap H|}{|z^G|}
\]
is the \emph{fixed point ratio} of $z \in G$ with respect to the natural transitive action of $G$ on $\O = G/H$, which is simply the proportion of points in $\O$ fixed by $z$. The following elementary observation was a key tool in \cite{GK} and we include a short proof for completeness.

\begin{lem}\label{l:us}
Let $G$ be a finite group. An element $y \in G$ is a witness if 
\begin{equation}\label{e:fpry}
\sum_{H \in \mathcal{M}(y)} {\rm fpr}(z,G/H) < 1
\end{equation}
for all $z \in G$ of prime order.
\end{lem}

\begin{proof}
Let $x \in G$ be a non-identity element and suppose $z = x^m$ has prime order. Let 
\[
\mathbb{P}(y,z) = \frac{|\{ v \in y^G \,:\, G \ne \la v,z \ra \}|}{|y^G|}
\]
be the probability that $z$ and a uniformly random conjugate of $y$ do not generate $G$. Now $G \ne \la v,z \ra$ if and only if $z^g \in H$ for some $g \in G$ and $H \in \mathcal{M}(y)$. And since $|z^G \cap H|/|z^G|$ is the probability that a random conjugate of $z$ is contained in $H$, it follows that
\[
\mathbb{P}(y,z) \leqs \sum_{H \in \mathcal{M}(y)} \frac{|z^G \cap H|}{|z^G|} =  \sum_{H \in \mathcal{M}(y)} {\rm fpr}(z,G/H).
\]
So if \eqref{e:fpry} holds then $\mathbb{P}(y,z) < 1$ and thus 
$G = \la z,y^g \ra = \la x, y^g \ra$
for some $g \in G$. Since $x$ is an arbitrary non-identity element, we conclude that $y$ is a witness.
\end{proof}

\begin{rem}\label{r:us}
Suppose $G$ is simple and $\{H_1, \ldots, H_t\}$ is a complete set of representatives of the conjugacy classes of maximal subgroups in $G$. Then each $H_i$ is self-normalising and we have
\begin{equation}\label{e:fpruseful}
\sum_{H \in \mathcal{M}(y)} {\rm fpr}(z,G/H) = \frac{1}{|G|}\sum_{i=1}^t |H_i|\chi_i(x)\chi_i(z),
\end{equation}
for all $z \in G$, where $\chi_i = 1^G_{H_i}$ is the permutation character for the action of $G$ on $G/H_i$.
\end{rem}

The connection between uniform spread and Theorem \ref{t:gen} is given by the following easy observation.

\begin{lem}\label{l:us2}
Let $G \leqs {\rm Sym}(\O)$ be a finite transitive permutation group and suppose $x,y \in G$ are witnesses such that
\begin{equation}\label{e:int}
\{ A^g \,:\, A \in \mathcal{M}(x), \, g \in G \} \cap \mathcal{M}(y) = \emptyset.
\end{equation}
Then there exists a derangement $z \in G$ such that $G = \la z,z^g\ra$ for some $g \in G$.
\end{lem}

\begin{proof}
Let $H$ be a point stabiliser. Since the intersection in \eqref{e:int} is empty, it follows that some element $z \in \{x,y\}$ is a derangement on $\O$. And since $z$ is a witness, we conclude that $G = \la z,z^g \ra$ for some $g \in G$. 
\end{proof}

\begin{rem}\label{r:usp}
Notice that a weaker form of Lemma \ref{l:us2} will be sufficient for establishing the conclusion in Theorem \ref{t:gen}. Indeed, it suffices to show that there exist elements $x,y \in G$ such that \eqref{e:int} holds and we have
\[
\sum_{H \in \mathcal{M}(z)} {\rm fpr}(z,G/H) < 1 
\]
for $z \in \{x,y\}$. In particular, it is not strictly necessary to show that $x$ or $y$ is a witness.
\end{rem}

\begin{rem}\label{r:invgen}
Let us also note that if $x,y \in G$ satisfy the condition in \eqref{e:int}, then $\{x,y\}$ is an \emph{invariable} generating set for $G$, in the sense that $G = \la x^a,y^b\ra$ for all $a,b \in G$.
\end{rem}

In view of Lemma \ref{l:us2} and the connection with uniform spread, we will make extensive use of \cite{BGK,GK}, where explicit witnesses for simple groups are identified. In several cases, we will need to prove that some additional elements are also witnesses and we will typically do this via Lemma \ref{l:us}, which involves determining the maximal overgroups and bounding the corresponding fixed point ratios. For example, if $G$ is a classical group then we will work closely with \cite{Ber,BHR,GPPS,KL} to study maximal overgroups and we will appeal to the fixed point ratio bounds in \cite{Bur1} and \cite[Section 3]{GK}. We will also use computational methods to handle some of the small simple groups that are amenable to direct computation in \textsf{GAP} \cite{GAP} or {\sc Magma} \cite{magma} (see Propositions \ref{p:sporgen} and \ref{p:gensmall}, for example).

\subsection{Alternating groups}\label{ss:gen_an}

Here we prove Theorem \ref{t:gen} in the case where $G = A_n$ is an alternating group with $n \geqs 5$. First we handle the groups of small degree.

\begin{lem}\label{l:smalldeg}
The conclusion to Theorem \ref{t:gen} holds if $G = A_n$ and $n \leqs 15$.
\end{lem}

\begin{proof}
This is a routine {\sc Magma} \cite{magma} computation. First, we construct a set of representatives of the conjugacy classes of maximal subgroups $H$ of $G$. Then by inspecting the conjugacy classes of elements in $H$ and $G$, it is straightforward to determine the set of derangements in $G$ (with respect to the action of $G$ on $\O = G/H$) and using random search we can easily find two conjugate derangements that generate $G$.
\end{proof}

For the remainder of this section, we may assume $n \geqs 16$. The following number-theoretic lemma will be useful (note that the conclusion is false if $n=15$).

\begin{lem}\label{l:nt1}
Let $n \geqs 16$ be an integer. Then there exist primes $p$ and $q$ such that $n/2 < p < q \leqs n-3$ and $q \ne (n+p)/2$.
\end{lem}

\begin{proof}
The cases with $16 \leqs n \leqs 40$ can be checked directly, so let us assume $n \geqs 41$. Then by a theorem of Ramanujan \cite{Ram}, there are at least $5$ primes in the interval $(n/2,n]$, which immediately implies that there are primes $p_i$ such that $n/2 < p_1 < p_2 < p_3 \leqs n-4$. The result now follows since we can set $p = p_1$ and $q = p_2$ or $p_3$, ensuring that $q \ne (n+p)/2$. 
\end{proof}

\begin{prop}\label{p:gen_an}
The conclusion to Theorem \ref{t:gen} holds if $G$ is an alternating group.
\end{prop}

\begin{proof}
Write $G = A_n$. In view of Lemma \ref{l:smalldeg}, we may assume $n \geqs 16$ and we fix primes $p$ and $q$ as in Lemma \ref{l:nt1}. Let $H$ be a point stabiliser and set $[n] = \{1, \ldots, n\}$.

First assume $H$ acts transitively on $[n]$ and set $L = \la x, y \ra$, where $x$ and $y$ are the $p$-cycles
\[
x=(1,2,\ldots,p),\;\; y=(n-p+1,\ldots,n).
\]
Since $p \geqs n-p+1$, it follows that $L$ acts transitively on $[n]$. In addition, since $p > n/2$, we observe that no transitive imprimitive subgroup of $G$ contains an element of order $p$, so $L$ is primitive. Finally, a classical theorem of Jordan (see \cite[Theorem 13.9]{Wielandt}) implies that no proper primitive subgroup of $G$ contains a $p$-cycle, whence $L = G$. In addition, this argument shows that neither $x$ nor $y$ is contained in a proper transitive subgroup of $G$, so $x$ and $y$ are conjugate derangements and the result follows.

For the remainder, we may assume $H$ is intransitive. Moreover, without loss of generality, we may assume that $H = (S_k \times S_{n-k}) \cap G$ for some $1 \leqs k < n/2$, which allows us to identify $\O$ with the set of $k$-element subsets of $[n]$. 

First assume $n \geqs 16$ is even. Define the elements
\begin{align*}
x & = (1,\ldots,p)(p+1,\ldots, n) \\
y & = (1,\ldots,p-1,p+1)(p+2,p,p+3,\ldots, n)
\end{align*}
in $G$ and set $L = \langle x,y\rangle$. Note that $y = x^z$, where $z = (p,p+1,p+2) \in G$, so $x$ and $y$ are $G$-conjugate. In addition, note that $L$ is transitive on $[n]$ and $x^{n-p}$ is a $p$-cycle, so by arguing as above we deduce that $L = G$. Clearly, $x$ and $y$ are derangements unless $k = n-p$. And if $k =n-p$ then we can redefine $x$ and $y$ with $p$ replaced by $q$ and repeat the argument.

Finally, let us assume $n \geqs 17$ is odd and set $L = \la x,y \ra$, where
\begin{align*}
x & = (1,\ldots,p)(p+1, \ldots, \ell)(\ell+1,\ldots,n) \\
y & = (1, \ldots, p-1,p+1)(p,p+2, \ldots, \ell-1,\ell+1)(\ell,\ell+2, \ldots, n)
\end{align*}
and $\ell = (n+p)/2$. Note that $x$ and $y$ both have cycle-type $[p,(n-p)/2, (n-p)/2]$, so they are conjugate in $G$. Then $L$ acts transitively on $[n]$ and it contains a $p$-cycle, so as before we deduce that $L=G$. Moreover, $x$ and $y$ are derangements unless $k \in \{(n-p)/2, n-p\}$, so let us assume we are in one of these two cases. Here we can repeat the argument, replacing $p$ by $q$ in the construction of $x$ and $y$, noting that the condition 
$q \ne (n+p)/2$ in Lemma \ref{l:nt1} implies that $x$ and $y$ are derangements on $\O$ when $k \in \{(n-p)/2, n-p\}$.
\end{proof}

\subsection{Sporadic groups}\label{ss:gen_spor}

Next let us turn to the sporadic groups. Here we adopt a computational approach, working  extensively with Lemma \ref{l:us2}. As before, we write 
$\mathbb{B}$ and $\mathbb{M}$ for the Baby Monster and Monster sporadic groups, respectively.

\begin{prop}\label{p:sporgen}
The conclusion to Theorem \ref{t:gen} holds if $G$ is a sporadic group.
\end{prop}

\begin{proof}
First assume $G \ne \mathbb{B}, \mathbb{M}$. Using \textsf{GAP} \cite{GAP} and the information available in the \textsf{GAP} Character Table Library \cite{GAPCTL}, it is straightforward to evaluate the expression in \eqref{e:fpruseful} for all $y,z \in G$ (of course, to do this we only need to work with a set of conjugacy class representatives). In every case, we can combine this with Lemma \ref{l:us} to identify witnesses $x$ and $y$ satisfying all the conditions in Lemma \ref{l:us2}, which gives the desired result. Two such elements are recorded in Table \ref{tab:sporgen} (using the standard Atlas \cite{Atlas} labelling for conjugacy classes), together with the maximal overgroups $\mathcal{M}(x)$ and $\mathcal{M}(y)$ (see \eqref{e:My}, as well as Remark \ref{r:spor_tab} below).

Next assume $G = \mathbb{B}$ and fix elements $x \in \texttt{55A}$ and $y \in \texttt{47A}$. The character table of every maximal subgroup of $G$ is available in \cite{GAPCTL} and just by considering element orders it is easy to check that every maximal overgroup of $y$ is of the form $H = 47{:}23$. Moreover, we can compute the permutation character $\chi = 1^G_H$ and we find that $\chi(y) = 1$, which implies that $y$ is contained in a unique maximal subgroup of $G$ (namely, $N_G(\la y \ra) = 47{:}23$). Then by applying Lemma \ref{l:us}, we  deduce that $y$ is a witness. Similarly, one can check that $\mathcal{M}(x) = \{ 5{:}4 \times {\rm HS}.2, S_5 \times {\rm M}_{22}{:}2\}$ and once again we see that $x$ is a witness. The result now follows from Lemma \ref{l:us2}.

Finally, let us assume $G = \mathbb{M}$ is the Monster. Fix $x,y \in G$, where $x \in \texttt{71A}$ and $y \in \texttt{59A}$. By inspecting the list of maximal subgroups of $G$ (see \cite[Table 1]{DLP}), it is easy to see that ${\rm L}_2(71)$ is the only maximal subgroup containing an element of order $71$. And similarly, ${\rm L}_2(59)$ is the only maximal subgroup with an element of order $59$. In both cases we can use \cite{GAPCTL} to compute the corresponding permutation characters and we conclude that $x$ and $y$ both have unique maximal overgroups. As above, the result now follows from Lemma \ref{l:us2}.
\end{proof}

\begin{rem}\label{r:spor_tab}
In the final two columns of Table \ref{tab:sporgen}, we list the maximal subgroups in the sets $\mathcal{M}(x)$ and $\mathcal{M}(y)$. In order to clarify our notation, note that if $G = {\rm M}_{12}$ and $x \in \texttt{11A}$, then $\mathcal{M}(x)$ contains two non-conjugate subgroups isomorphic to ${\rm M}_{11}$. On the other hand, if $G = {\rm J}_2$ and $y \in \texttt{7A}$, then $\mathcal{M}(y)$ contains two conjugate subgroups isomorphic to ${\rm U}_{3}(3)$.
\end{rem}

{\footnotesize
\begin{table}
\[
\begin{array}{lllll} \hline
G & x & y & \mathcal{M}(x) & \mathcal{M}(y) \\ \hline 
{\rm M}_{11} &  \texttt{11A} & \texttt{8A}  & {\rm L}_2(11) & {\rm M}_{10},\, 3^2{:}{\rm SD}_{16},\, 2.S_4 \\ 

{\rm M}_{12} & \texttt{11A} & \texttt{10A} & {\rm M}_{11}, \, {\rm M}_{11}, \, {\rm L}_2(11) &  A_6.2^2, \, A_6.2^2, \, 2 \times S_5 \\

{\rm M}_{22} & \texttt{11A} & \texttt{8A}  & {\rm L}_2(11) & 2^4{:}A_6, \, 2^4{:}S_5, \, {\rm M}_{10} \\
 
{\rm M}_{23} &  \texttt{23A} & \texttt{15A} & 23{:}11 & A_8, \, 2^4{:}(3 \times A_5).2 \\  

{\rm M}_{24} &  \texttt{23A} & \texttt{21A} & {\rm M}_{23},\, {\rm L}_2(23) & {\rm L}_3(4){:}S_3, \, 2^6{:}({\rm L}_3(2) \times S_3) \\
   
{\rm J}_1 & \texttt{19A} & \texttt{15A} & 19{:}6 & D_6 \times D_{10} \\ 
  
{\rm J}_{2} &  \texttt{10C} & \texttt{7A} & 2^{1+4}{:}A_5, \, A_5\times D_{10}, \, 5^2{:}D_{12}  & {\rm U}_3(3) \mbox{ (two)}, \, {\rm L}_3(2).2 \\
  
{\rm J}_{3} &  \texttt{19A} & \texttt{17A} & {\rm L}_2(19), \, {\rm L}_2(19) & {\rm L}_2(16).2, \, {\rm L}_2(17) \\
  
{\rm J}_{4} &  \texttt{43A} & \texttt{29A} & 43{:}14 & 29{:}28 \\
  
{\rm HS} &  \texttt{15A} & \texttt{11A}  & S_8, \, 5{:}4 \times A_5 & {\rm M}_{22},\, {\rm M}_{11}, \, {\rm M}_{11} \\
  
{\rm He} & \texttt{17A} & \texttt{14C}  & {\rm Sp}_4(4).2 & 2^{1+6}.{\rm L}_3(2),\, 
7^2{:}2.{\rm L}_2(7),\, 7^{1+2}{:}(S_3\times 3) \\

{\rm McL} & \texttt{15A} & \texttt{11A} & 3^{1+4}{:}2S_5, \, 2.A_8, \, 5^{1+2}{:}3{:}8 & {\rm M}_{22},\, {\rm M}_{22},\, {\rm M}_{11} \\
  
{\rm Suz}  & \texttt{14A} & \texttt{13A} & {\rm J}_2.2 \mbox{ (two)}, \, (A_4\times {\rm L}_3(4)){:}2 & G_2(4), \, {\rm L}_3(3).2, \, {\rm L}_3(3).2, \, {\rm L}_2(25) \mbox{ (three)} \\
    
{\rm Ru} & \texttt{29A} & \texttt{26A} & {\rm L}_2(29) & (2^2 \times {}^2B_2(8)){:}3, \,  {\rm L}_2(25).2^2 \mbox{ (two)} \\
  
{\rm O'N} & \texttt{31A} & \texttt{19A} & {\rm L}_2(31), \, {\rm L}_2(31) & {\rm L}_3(7).2, \, {\rm L}_3(7).2, \, {\rm J}_1 \\

{\rm Co}_1 & \texttt{26A} & \texttt{23A} & (A_4\times G_2(4)){:}2 & {\rm Co}_2, \, {\rm Co}_3, \, {\rm M}_{24} \\
   
{\rm Co}_2 &  \texttt{30A} & \texttt{23A} & {\rm U}_6(2).2, \, 2^{1+8}{:}{\rm Sp}_6(2), \, {\rm HS}.2 & {\rm M}_{23} \\
  
{\rm Co}_3 & \texttt{30A} & \texttt{23A} & {\rm McL}.2,\, 2.{\rm Sp}_6(2),\, {\rm U}_3(5){:}S_3, \, 3^{1+4}{:}4S_6 & {\rm M}_{23} \\
  
{\rm HN} & \texttt{22A} & \texttt{19A} & 2.{\rm HS}.2 & {\rm U}_3(8).3 \\

{\rm Ly} & \texttt{67A} & \texttt{37A} & 67{:}22 & 37{:}18 \\
       
{\rm Th} &  \texttt{39A} & \texttt{36A} & (3 \times G_2(3)){:}2 & 2^{1+8}.A_9, \, 3.[3^8].2S_4 \\

{\rm Fi}_{22} & \texttt{22A} & \texttt{13A} & 2.{\rm U}_6(2) & \O_7(3), \, \O_7(3), \, {}^2F_4(2) \\
  
{\rm Fi}_{23} & \texttt{35A} & \texttt{23A} & S_{12} & 2^{11}.{\rm M}_{23}, \, {\rm L}_{2}(23) \\
  
{\rm Fi}_{24}' & \texttt{29A} & \texttt{23A} & 29{:}14 & {\rm Fi}_{23}, \, 2^{11}.{\rm M}_{24} \\
    
\mathbb{B} & \texttt{55A} & \texttt{47A} & 5{:}4 \times {\rm HS}.2, \, S_5 \times {\rm M}_{22}{:}2 & 47{:}23 \\
  
\mathbb{M} & \texttt{71A} & \texttt{59A} & {\rm L}_2(71) & {\rm L}_2(59) \\ \hline
\end{array}
\]
\caption{The elements $x,y \in G$ in the proof of Proposition \ref{p:sporgen}}\label{tab:sporgen}
\end{table}
}

\subsection{Exceptional groups}\label{ss:gen_ex}

In this section we assume $G$ is a finite simple exceptional group of Lie type over $\mathbb{F}_{q}$, where $q = p^f$ and $p$ is a prime. We begin by handling the groups in $\mathcal{A}$, which is defined as follows:
\[
\mathcal{A} = \{ {}^2G_2(3)', G_2(2)',  G_2(3), G_2(4), {}^2F_4(2)', {}^3D_4(2), F_4(2), {}^2E_6(2) \}.
\]

\begin{lem}\label{l:ex_gensmall}
The conclusion to Theorem \ref{t:gen} holds if $G \in \mathcal{A}$.
\end{lem}

\begin{proof}
For $G \ne {}^2E_6(2)$ we can proceed as in the proof of Proposition \ref{p:sporgen}, working the \textsf{GAP} Character Table Library \cite{GAPCTL} to evaluate the expression in \eqref{e:fpruseful} for all $y,z \in G$. In this way, using Lemma \ref{l:us}, we can find two elements $x,y \in G$ satisfying the conditions in Lemma \ref{l:us2} and the result follows (two such elements are presented in Table \ref{tab:exgensmall} with respect to the standard Atlas \cite{Atlas} notation for conjugacy classes, together with the respective maximal overgroups $\mathcal{M}(x)$ and $\mathcal{M}(y)$).

To complete the proof, we may assume $G = {}^2E_6(2)$. The maximal subgroups of $G$ have been determined up to conjugacy by Wilson in \cite{Wil_max} (also see Tables 3 and 10 in \cite{Craven}). Fix elements $x \in \texttt{19A}$ and $y \in \texttt{13A}$. 

By inspecting the orders of the maximal subgroups of $G$, we find that $H = {\rm U}_3(8).3$ is the only one containing an element of order $19$. Moreover, we can use \textsf{GAP} to compute the character $\chi = 1^G_H$ and we deduce that $\chi(x) = 1$. Therefore, $\mathcal{M}(x) = \{ H \}$ and thus $x$ is a witness by Lemma \ref{l:us}.

Now let us consider $y$, which represents the unique conjugacy class of elements of order $13$ in $G$. By inspection, we see that each $H \in \mathcal{M}(y)$ is isomorphic to $F_4(2)$ or ${\rm Fi}_{22}$, noting that there are three conjugacy classes of each type. By working with the corresponding permutation characters, we deduce that $\mathcal{M}(y)$ comprises three subgroups isomorphic to $F_4(2)$, which represent the three classes of maximal subgroups of this type, together with six subgroups isomorphic to ${\rm Fi}_{22}$ (two subgroups from each of the three classes of this type). By \cite[Theorem 1]{LLS} we have ${\rm fpr}(z,G/H) \leqs 1/57$ for all $z \in G$ of prime order and all maximal subgroups $H$ of $G$, whence 
\[
\sum_{H \in \mathcal{M}(y)} {\rm fpr}(z,G/H) \leqs \frac{9}{57}
\]
for all $z \in G$ of prime order. Therefore, $y$ is a witness by Lemma \ref{l:us} and we can now conclude by appealing to Lemma \ref{l:us2}.
\end{proof}

{\footnotesize
\begin{table}
\[
\begin{array}{lllll} \hline
G & x & y & \mathcal{M}(x) & \mathcal{M}(y) \\ \hline

{}^2G_2(3)' & \texttt{7A} & \texttt{3A} & 2^3{:}7 \mbox{ (two)}, \, D_{14} & D_{18} \\

G_2(2)' & \texttt{12A} & \texttt{7A} & 3^{1+2}.D_8,\, 4.S_4 & {\rm L}_3(2) \\

G_2(3) & \texttt{13A} & \texttt{9A} & {\rm L}_3(3){:}2, \, {\rm L}_3(3){:}2, \, {\rm L}_2(13) & (3^2 \times 3^{1+2}){:}2S_4, \, (3^2 \times 3^{1+2}){:}2S_4, \, {\rm L}_2(8){:}3 \mbox{ (three)}\\

G_2(4) & \texttt{21A} & \texttt{13A} & 3.{\rm L}_3(4){:}2 & {\rm U}_3(4){:}2, \, {\rm L}_2(13) \\

{}^2F_4(2)' & \texttt{16A} & \texttt{13A} & 2.[2^8].5.4,\, 2^2.[2^8].S_3 & {\rm L}_3(3).2, \,  {\rm L}_3(3).2, \, {\rm L}_2(25) \mbox{ (three)} \\

{}^3D_4(2) & \texttt{13A} & \texttt{7D} & 13{:}4 & G_2(2) \mbox{ (seven)}, \, (7 \times {\rm SL}_3(2)).2 \mbox{ (four)}, \, 7^2.{\rm SL}_2(3) \\

F_4(2) & \texttt{28A} & \texttt{17A} & [2^{15}].{\rm Sp}_6(2), \, {}^3D_4(2).3, \, {\rm L}_3(2)^2.2,  & {\rm Sp}_8(2), \, {\rm Sp}_8(2) \\
& & & [2^{20}].(S_3 \times {\rm L}_3(2)) \mbox{ (two)} & \\

{}^2E_6(2) & \texttt{19A} & \texttt{13A} & {\rm U}_3(8).3 & F_4(2),  F_4(2),  F_4(2), \\
& & & & {\rm Fi}_{22} \, \mbox{(two)},  {\rm Fi}_{22} \, \mbox{(two)},  {\rm Fi}_{22} \, \mbox{(two)} \\ \hline
\end{array}
\]
\caption{The elements $x,y \in G$ in the proof of Lemma \ref{l:ex_gensmall}, $G \in \mathcal{A}$}
\label{tab:exgensmall}
\end{table}
}

For the remainder of this section, we may assume $G \not\in \mathcal{A}$. The following lemma will be a useful observation.

\begin{lem}\label{l:ex_gen1}
Suppose $G \not\in \mathcal{A}$ and $x \in G$ satisfies $|\mathcal{M}(x)| \leqs 12$. Then $x$ is a witness.
\end{lem}

\begin{proof}
This follows by combining Lemma \ref{l:us} with \cite[Theorem 1]{LLS}, which implies that  ${\rm fpr}(z,G/H) \leqs 1/13$ for all $z \in G$ of prime order and all maximal subgroups $H$ of $G$. 
\end{proof}

\begin{prop}\label{p:ex_gen}
The conclusion to Theorem \ref{t:gen} holds if $G$ is an exceptional group of Lie type.\end{prop}

\begin{proof}
In view of Lemma \ref{l:ex_gensmall}, we may assume $G \not\in \mathcal{A}$. In every case, we will show that there exist witnesses $x,y \in G$ satisfying the conditions in Lemma \ref{l:us2}. 

For now, let us assume $G \ne {}^3D_4(q)$ and choose $x,y \in G$ as in Table \ref{tab:exgen}. In the table, we write 
$\Phi_i$ for the $i$-th cyclotomic polynomial evaluated at $q$ and we use the notation $d = (2,q-1)$, $e = (3,q-1)$ and $e'=(3,q+1)$. In addition, for $G = {}^2F_4(q)$ we define
\[
\a^{\e} = q^2 + \e\sqrt{2q^3} + q + \e\sqrt{2q} + 1
\]
for $\e = \pm$. For $G = E_7(q)$, we have
\[
H_1 = e'.({}^2E_6(q) \times (q+1)/de').e'.2,\;\; H_2 = e.(E_6(q) \times (q-1)/de).e.2
\]
and $H_3$ is an $E_6$-parabolic subgroup of type $P_7$ (in terms of the standard labelling of maximal parabolic subgroups of $G$).

Observe that in every case, $x$ and $y$ generate distinct maximal tori of $G$, which we denote by $T_x$ and $T_y$, respectively. The existence of cyclic tori of the given orders, and hence the existence of $x$ and $y$, follows from the general theory of tori in finite reductive groups (see \cite[Section 3.3]{Carter}, for example). So in view of Lemmas \ref{l:us2} and \ref{l:ex_gen1}, it just remains to justify the given description of the sets $\mathcal{M}(x)$ and $\mathcal{M}(y)$ in Table \ref{tab:exgen}.
 
{\footnotesize
\begin{table}
\[
\begin{array}{llllll} \hline
G & & |x| & |y| & \mathcal{M}(x) &  \mathcal{M}(y) \\ \hline
{}^2B_2(q) & & q+\sqrt{2q}+1 & q-\sqrt{2q}+1 & T_x.4 & T_y.4 \\

{}^2G_2(q) & & q + \sqrt{3q} + 1 & q - \sqrt{3q} + 1 & T_x.6 & T_y.6 \\

G_2(q) & q \equiv \e \imod{3} & q^2 - \epsilon q + 1 & q^2 + \epsilon q + 1 & {\rm SL}_3^{-\epsilon}(q).2  & {\rm SL}_3^{\epsilon}(q).2 \\ 

 & p = 3 & q^2 + q + 1 & q^2 - q + 1 & {\rm SL}_3(q).2 \mbox{ (two)} & {\rm SU}_3(q).2 \mbox{ (two)}  \\

{}^2F_4(q) & & \a^+ & \a^- & T_x.12 & T_y.12 \\

F_4(q) & p \ne 2 & q^4 - q^2 + 1 & q^4+1 & {}^3D_4(q).3 & 2.\O_9(q) \\

& p = 2 & q^4 - q^2 + 1 & q^4+1 & {}^3D_4(q).3 \mbox{ (two)} & {\rm Sp}_8(q) \mbox{ (two)} \\

E_6(q) & & \Phi_9/e & \Phi_{12}\Phi_3/e & {\rm L}_3(q^3).3 & ({}^3D_4(q) \times (q^2+q+1)/e).3 \\ 

{}^2E_6(q) & & \Phi_{18}/e' & \Phi_{12}\Phi_6/e' & {\rm U}_3(q^3).3 & ({}^3D_4(q) \times (q^2-q+1)/e').3 \\

E_7(q) & q \geqs 3 & \Phi_{18}\Phi_2/d & \Phi_9\Phi_1/d & H_1 & H_2, \, H_3 \mbox{ (two)} \\
& q = 2 & 129 & 73 & {\rm SU}_8(2) & H_2, \, H_3 \mbox{ (two)} \\
E_8(q) & & \Phi_{30} & \Phi_{15} & T_x.30 & T_y.30 \\ \hline
\end{array}
\]
\caption{The elements $x,y \in G$ in the proof of Proposition \ref{p:ex_gen}, $G \not\in \mathcal{A}$, $G \ne {}^3D_4(q)$}\label{tab:exgen}
\end{table}
}

For now we will assume $G \ne G_2(q)$. If we exclude the groups
\begin{equation}\label{e:exc}
\{ F_4(q), \, {}^2E_6(q),\, E_7(q) \,:\, q=2,3 \}
\end{equation}
then the maximal overgroups of $T_x$ are determined by Weigel \cite{Weigel} and we read off the subgroups in $\mathcal{M}(x)$ recorded in the fourth column of the table. And for the groups in \eqref{e:exc} we can appeal to the proof of \cite[Proposition 6.2]{GK}. Similarly, the description of $\mathcal{M}(y)$ in Table \ref{tab:exgen} follows from \cite[Theorem 2.1]{GM2}.

Now suppose $G = G_2(q)$ and recall that $q \geqs 5$ since we are assuming $G \not\in \mathcal{A}$. As above, we can read off $\mathcal{M}(y)$ from \cite[Theorem 2.1]{GM2}. Similarly, our description of $\mathcal{M}(x)$ follows from \cite{Weigel} unless $q \equiv -1 \imod{3}$. Here $|x| = q^2+q+1$ and by inspecting the list of maximal subgroups of $G$ (see \cite[Tables 8.30, 8.41]{BHR}, for example) it is easy to see that each subgroup in $\mathcal{M}(x)$ is conjugate to $H = {\rm SL}_3(q).2$. Moreover, we note that $x^G \cap H$ is the union of two ${\rm SL}_3(q)$-classes, both of which have size $|{\rm SL}_3(q)|/(q^2+q+1)$. Therefore,
\[
{\rm fpr}(x,G/H) \cdot |G:N_G(H)| = \frac{|x^G\cap H|}{|x^G|} \frac{|G|}{|H|} = 1
\]
and we conclude that $x$ is contained in a unique conjugate of $H$, as recorded in Table \ref{tab:exgen}.

To complete the proof, we may assume $G = {}^3D_4(q)$ with $q \geqs 3$ (recall that ${}^3D_4(2) \in \mathcal{A}$). Let $x,y \in G$ be regular semisimple elements with respective orders $q^4-q^2+1$ and $q^2+q+1$, and corresponding centralisers of order $q^4-q^2+1$ and $(q^2+q+1)^2$. By \cite[Proposition 2.11]{GM} we have $\mathcal{M}(x) = \{ T.4 \}$, where $T = \la x \ra$ is a maximal torus, and thus $x$ is a witness. So to complete the proof, we may assume $G$ has point stabiliser $H = T.4$. 

Now let us consider $y$. As noted in the proof of \cite[Proposition 2.3]{GM2}, each $H \in \mathcal{M}(y)$ is conjugate to one of the following
\[
H_1 = G_2(q),\; H_2 = ((q^2+q+1) \circ {\rm SL}_3(q)).(q^2+q+1,3).2,
\]
\[
H_3 = (q^2+q+1)^2.{\rm SL}_2(3),\; H_4 = {\rm PGL}_3(q),
\]
where $H_4$ only arises when $q \equiv 1 \imod{3}$. In particular, $y$ is a derangement on $\O = G/H$ and so it just 
remains to show that $G = \la y, y^g \ra$ for some $g \in G$.

For $q=3$ we can use {\sc Magma} to show that $|\mathcal{M}(y)| = 18$ and the result follows via Lemmas \ref{l:us} and \ref{l:us2} since \cite[Theorem 1]{LLS} gives ${\rm fpr}(z,G/H) \leqs 1/73$ for all $z \in G$ of prime order and all maximal subgroups $H$ of $G$. 
Now assume $q \geqs 4$. As noted in Remark \ref{r:usp}, we just need to show that
\[
\sum_{H \in \mathcal{M}(y)} {\rm fpr}(y,G/H) = \sum_{i=1}^4 n_i \cdot {\rm fpr}(y,G/H_i)  = |G|\sum_{i=1}^4 |H_i|^{-1} \cdot {\rm fpr}(y,G/H_i)^2 < 1,
\]
where $n_i$ is the number of conjugates of $H_i$ containing $y$. Since 
$|G|<q^{28}$ and $|y^G|>\frac{1}{2}q^{24}$, the crude bound 
\[
{\rm fpr}(y,G/H_i) < \frac{|H_i|}{|y^G|} < 2|H_i|q^{-24}
\]
yields
\[
|G|\sum_{i=1}^4 |H_i|^{-1} \cdot {\rm fpr}(y,G/H_i)^2 < 4q^{-20}\sum_{i=1}^4 |H_i| < 1
\]
and the result follows.
\end{proof}

\subsection{Classical groups}\label{ss:gen_class}

In order to complete the proof of Theorem \ref{t:gen}, it remains to handle the classical groups. So let $G$ be a finite simple classical group over $\mathbb{F}_q$, where $q = p^f$ and $p$ is a prime. Let $V$ be the natural module and set $n = \dim V$. 

As explained at the start of Section \ref{ss:class_sol}, we may assume that $G$ is one of the groups in \eqref{e:class_list}. In addition, we may exclude the groups in \eqref{e:omit} due to the existence of exceptional isomorphisms involving some of the low dimensional classical groups.

Recall that the main theorem on the subgroup structure of finite classical groups is due to Aschbacher \cite{Asch} and we will refer repeatedly to the detailed information on maximal subgroups in \cite{BHR,KL}. Also recall that the \emph{type} of a maximal subgroup gives an approximate description of its structure. For example, if $G = {\rm L}_6(q)$ and $H$ is the stabiliser of a direct sum decomposition $V = V_1 \oplus V_2 \oplus V_3$ of the natural module into $2$-spaces, then we say that $H$ is a subgroup of type ${\rm GL}_2(q) \wr S_3$, noting that the precise structure can be read off from \cite[Proposition 4.2.9]{KL} (also see \cite[Table 8.24]{BHR}). Detailed information on the maximal subgroups of classical groups is presented in \cite{BHR,KL}, including a complete classification (up to conjugacy) for the groups with $n \leqs 12$.

Let us also recall that if $e,q \geqs 2$ are integers, then a prime divisor $r$ of $q^e-1$ is a \emph{primitive prime divisor} (ppd for short) if $r$ does not divide $q^i-1$ if $i<e$. By a theorem of Zsigmondy \cite{Zsig}, a ppd exists unless $(e,q) = (6,2)$, or if $e=2$ and $q$ is a Mersenne prime. If $G$ is a classical group as above and $x \in G$ has prime order $r$, where $r$ is a ppd of $q^e-1$ with $e>n/2$, then the maximal overgroups of $x$ (up to conjugacy) are described in \cite{GPPS} and we will use this important result in several proofs.

Let $G$ be one of the groups listed in Table \ref{tab:sing} and recall that an element $y \in G$ is a \emph{Singer cycle} if the order of $y$ is as given in the table. Note that every Singer cycle acts irreducibly on the natural module $V$. In the proof of the following result, we write $\pi(\ell)$ for the set of prime divisors of the integer $\ell \geqs 2$.

{\footnotesize
\begin{table}
\[
\begin{array}{lll} \hline
G & |y| & \mbox{Conditions} \\ \hline
{\rm L}_n(q) & (q^n-1)/e & \mbox{$n \geqs 2$, $e = (n,q-1)(q-1)$} \\
{\rm U}_n(q) & (q^n+1)/e' & \mbox{$n \geqs 3$ odd, $e' = (n,q+1)(q+1)$} \\  
{\rm PSp}_{n}(q) & (q^{n/2}+1)/d & \mbox{$n \geqs 4$, $d = (2,q-1)$} \\
{\rm P\O}_{n}^{-}(q) & (q^{n/2}+1)/d & \mbox{$n \geqs 8$, $d = (2,q-1)$} \\ \hline
\end{array}
\]
\caption{Singer cycles}
\label{tab:sing}
\end{table}
}

\begin{prop}\label{p:singer}
Every Singer cycle is a witness.
\end{prop}

\begin{proof}
Let $G$ be one of the groups in Table \ref{tab:sing} and let $y \in G$ be a Singer cycle. We will consider each possibility for $G$ in turn.

First assume $G = {\rm L}_n(q)$ is a linear group, so $n \geqs 2$ and $(n,q) \ne (2,2), (2,3)$. We begin by claiming that 
\begin{equation}\label{e:sing_lin}
\mathcal{M}(y) = \left\{ \begin{array}{ll}
\{ S_3, \, A_4 \mbox{ {\rm (two)}} \} & \mbox{if $(n,q) = (2,5)$} \\
\{ S_4, \, S_4 \}   & \mbox{if $(n,q) = (2,7)$} \\
\{ A_5, \, A_5 \} & \mbox{if $(n,q) = (2,9)$} \\
\{ {\rm L}_3(2),\, {\rm L}_3(2), \, {\rm L}_3(2) \} & \mbox{if $(n,q) = (3,4)$} \\
\{ H_k \, :\, k \in \pi(n) \} & \mbox{otherwise,}
\end{array}\right.
\end{equation}
where $H_k$ is a field extension subgroup of type ${\rm GL}_{n/k}(q^k)$. Here our notation indicates that $\mathcal{M}(y)$ contains two conjugate subgroups isomorphic to $A_4$ when $G = {\rm L}_2(5)$, and it contains two non-conjugate subgroups isomorphic to $S_4$ when $G = {\rm L}_2(7)$.

The groups with  $(n,q) = (2,5), (2,7), (2,9)$ and $(3,4)$ can be handled using \textsc{Magma}, so we may  assume that we are not in one of these cases. Then by the main theorem of \cite{Ber}, each $H \in \mathcal{M}(y)$ is a field extension subgroup of type ${\rm GL}_{n/k}(q^k)$, where $k$ is a prime divisor of $n$. Moreover, \cite[Lemma 2.12]{BGK} implies that $y$ is contained in a unique subgroup of type ${\rm GL}_{n/k}(q^k)$ for each prime $k$, so this justifies \eqref{e:sing_lin} and it just remains to show that $y$ is a witness. We will do this by establishing the inequality \eqref{e:fpry} in Lemma \ref{l:us}.

If $n \leqs 5$ then $|\mathcal{M}(y)| = 1$ and thus \eqref{e:fpry} holds for all $z \in G$ of prime order. Now assume $n \geqs 6$ and $H \in \mathcal{M}(y)$. Fix an element $z \in G$ of prime order and note that 
\[
{\rm fpr}(z,G/H) < |z^G|^{-\frac{1}{2}+\frac{1}{n}}
\]
by the main theorem of \cite{Bur1}. Now $|\mathcal{M}(y)|  = |\pi(n)| < \log_2n$ and it is straightforward to show that $|z^G| > q^{2n-2}$ (see \cite[Section 3]{Bur2}, for example). Therefore 
\[
\sum_{H \in \mathcal{M}(y)} {\rm fpr}(z,G/H) < \log_2 n\left(q^{2n-2}\right)^{-\frac{1}{2}+\frac{1}{n}}
\]
and it is easy to check that this upper bound is less than $1$ for all $n \geqs 6$ and $q \geqs 2$.

\vs

Next assume $G = {\rm U}_n(q)$, where $n \geqs 3$ is odd and $(n,q) \ne (3,2)$. Here we claim that
\begin{equation}\label{e:sing_unit}
\mathcal{M}(y) = \left\{ \begin{array}{ll}
\{ {\rm L}_2(7) \} & \mbox{if $(n,q) = (3,3)$} \\
\{ A_7, \, A_7, \, A_7\}   & \mbox{if $(n,q) = (3,5)$} \\
\{ {\rm L}_2(11) \} & \mbox{if $(n,q) = (5,2)$} \\
\{ H_k \, :\, k \in \pi(n) \} & \mbox{otherwise,}
\end{array}\right.
\end{equation}
where $H_k$ is a field extension subgroup of type ${\rm GU}_{n/k}(q^k)$. For $(n,q) = (3,3), (3,5)$ or $(5,2)$ we can use {\sc Magma} to verify the result (note that if $(n,q) = (3,5)$ then $G$ has three conjugacy classes of maximal subgroups isomorphic to $A_7$ and $\mathcal{M}(y)$ contains a representative from each class). And in each of the remaining cases, we apply \cite[Proposition 5.21]{BGK}.

\vs

Now suppose $G = {\rm PSp}_n(q)$ is a symplectic group, where $n=2m \geqs 4$ and $(n,q) \ne (4,2)$. We claim that
\begin{equation}\label{e:sing_symp}
\mathcal{M}(y) = \left\{ \begin{array}{ll}
\{ H_2, \, 2^4.A_5 \mbox{ {\rm  (two)}} \} & \mbox{if $(n,q) = (4,3)$} \\
\{ H_3 \mbox{ {\rm (three)}}, \, {\rm O}_6^{-}(2) \} & \mbox{if $(n,q) = (6,2)$} \\
\{ H_2, \, {\rm O}_8^{-}(2), \, {\rm L}_2(17) \}   & \mbox{if $(n,q) = (8,2)$} \\
\{ H_k \, :\, k \in \pi(m) \} \cup \mathcal{J} & \mbox{otherwise,}
\end{array}\right.
\end{equation}
where $H_k$ is a field extension subgroup of type ${\rm Sp}_{n/k}(q^k)$ and 
\[
\mathcal{J} = \left\{\begin{array}{ll}
\{ H  \} & \mbox{if $q$ is even} \\
\{ K \} & \mbox{if $qm$ is odd} \\ 
\emptyset & \mbox{otherwise}
\end{array}\right.
\]
with $H,K$ of type ${\rm O}_{n}^{-}(q)$ and ${\rm GU}_m(q)$, respectively.

In order to justify the claim, first assume $q$ is even. The cases $(n,q) = (6,2), (8,2)$ can be checked directly, and for each of the remaining groups we can appeal to \cite[Proposition 5.8]{BGK} (and its proof). As indicated in \eqref{e:sing_symp}, if $G = {\rm Sp}_6(2)$ then $y$ is contained in exactly $3$ conjugate field 
extension subgroups of type ${\rm Sp}_2(2^3)$ (this observation was not noted in the proof of \cite[Proposition 5.8]{BGK}, and it was also incorrectly omitted in the statement of \cite[Lemma 8.7]{BH1}). 

Now assume $q$ is odd. The case $(n,q) = (4,3)$ can be handled directly. For the remaining groups, the main theorem of \cite{Ber} implies that each $H \in \mathcal{M}(y)$ is a field extension subgroup of type ${\rm Sp}_{n/k}(q^k)$ or ${\rm GU}_m(q)$, noting in the latter case that $H$ contains a Singer cycle if and only if $m$ is odd. The description of $\mathcal{M}(y)$ in \eqref{e:sing_symp} now follows from \cite[Lemma 2.12]{BGK}. 

To complete the argument for symplectic groups, we need to show that $y$ is a witness when $q$ is odd and $(n,q) \ne (4,3)$. If $n=4$ then $|\mathcal{M}(y)| = 1$ and the result follows from Lemma \ref{l:us}, so we may assume $n \geqs 6$. Fix a subgroup $H \in \mathcal{M}(y)$ and let $z \in G$ be an element of prime order. Now $|\mathcal{M}(y)| < \log_2n$ and $|z^G| \geqs (q^n-1)/2$, so by applying the main theorem of \cite{Bur1} we deduce that
\[
\sum_{H \in \mathcal{M}(y)} {\rm fpr}(z,G/H) < \log_2n\left(\frac{1}{2}(q^n-1)\right)^{-\frac{1}{2}+\frac{1}{n}+\frac{1}{n+2}}.
\]
It is routine to check that this upper bound is less than $1$ for all $n \geqs 6$, $q \geqs 3$ and the result follows.

\vs

To complete the proof, let $G = {\rm P\O}_{n}^{-}(q)$ with $n=2m \geqs 8$. By combining the main theorem of \cite{Ber} with \cite[Lemma 2.12]{BGK}, we deduce that
\begin{equation}\label{e:sing_ominus}
\mathcal{M}(y) = \{ H_k \,: \, k \in \pi(m) \} \cup \mathcal{J},
\end{equation}
where $H_k$ is a field extension subgroup of type ${\rm O}_{n/k}^{-}(q^k)$ and 
\[
\mathcal{J} = \left\{ \begin{array}{ll}
\{ H \} & \mbox{if $m$ is odd} \\
\emptyset & \mbox{otherwise}
\end{array}\right.
\]
with $H$ of type ${\rm GU}_m(q)$. If $z \in G$ has prime order, then by combining the bounds $|\mathcal{M}(x)| < \log_2n$ and $|z^G| > q^{2n-6}$ with the main theorem of \cite{Bur1}, we get
\[
\sum_{H \in \mathcal{M}(y)} {\rm fpr}(z,G/H) < \log_2n \left(q^{2n-6}\right)^{-\frac{1}{2}+\frac{1}{n} + \frac{1}{n-2}} < 1
\]
for all $n \geqs 8$, $q \geqs 2$ and the result follows via Lemma \ref{l:us}.
\end{proof}

We are now ready to prove Theorem \ref{t:gen} for classical groups. It is convenient to use 
{\sc Magma} to verify the result for some of the low dimensional classical groups defined over small fields. With this aim in mind, let $\mathcal{B}$ denote the following collection of groups:
\[
\def\arraystretch{1.5}
\begin{array}{l}
{\rm L}_2(11), {\rm L}_3(2), {\rm L}_3(3), {\rm L}_3(4), {\rm L}_4(3) \\
{\rm U}_3(4), {\rm U}_3(5), {\rm U}_4(2), {\rm U}_4(3), {\rm U}_4(4), {\rm U}_4(5), {\rm U}_5(2), {\rm U}_5(3), {\rm U}_5(4), {\rm U}_6(2), {\rm U}_6(3), {\rm U}_7(2) \\
{\rm PSp}_4(3), {\rm Sp}_4(4), {\rm PSp}_4(5), {\rm Sp}_6(2), {\rm PSp}_6(3), {\rm Sp}_6(4), {\rm Sp}_8(2), {\rm Sp}_{10}(2), {\rm Sp}_{12}(2) \\
\O_7(3), \O_{8}^{\pm}(2), {\rm P\O}_{8}^{\pm}(3), \O_{8}^{+}(4), {\rm P\O}_8^{+}(5), \O_{10}^{\pm}(2), {\rm P\O}_{10}^{\pm}(3), \O_{12}^{\pm}(2), {\rm P\O}_{12}^{\pm}(3)
\end{array}
\]

\begin{prop}\label{p:gensmall}
The conclusion to Theorem \ref{t:gen} holds if $G \in \mathcal{B}$.
\end{prop}

\begin{proof}
We can proceed as in the proof of Lemma \ref{l:smalldeg}, working with {\sc Magma} \cite{magma}. Set $\O = G/H$, where $H$ is a maximal subgroup of $G$. For the computation, it is convenient to work in the corresponding matrix group $L \in \{ {\rm SL}_{n}^{\e}(q), {\rm Sp}_n(q), \O_n^{\e}(q)\}$, so we have $G = L/Z$ and $H = J/Z$, where $Z = Z(L)$. We can use the function \textsf{ClassicalMaximals} to construct a conjugate of $J$ and by taking conjugacy classes we can determine the set $\Delta(L,J)$ of derangements in $L$ with respect to the action on $L/J$. Finally, we use random search to find elements $x \in \Delta(L,J)$ and $y \in L$ such that $L = \la x,x^y \ra$, which in turn implies that $G$ is generated by two conjugate derangements on $\O$.
\end{proof}

\begin{rem}\label{r:comm}
For the groups in $\mathcal{B}$ that are handled in Proposition \ref{p:gensmall}, there is often no need to determine the complete set of derangements in order to reach the desired conclusion. For example, in many cases we can simply construct $L$ and $J$ as in the proof of the proposition, and then randomly search for elements $x,y \in L$ such that $L = \la x,x^y \ra$ and $|J|$ is indivisible by $|x|$, which means that $x \in \Delta(L,J)$. Alternatively, we can construct a set representatives of the conjugacy classes in $J$ and then seek random elements $x,y \in L$ such that $L = \la x,x^y\ra$ and $J$ does not contain a conjugate of $x$ (and in many cases we can simply check that $J$ has no element of order $|x|$).
\end{rem}

We begin by considering the linear groups.

\begin{prop}\label{p:lingen}
The conclusion to Theorem \ref{t:gen} holds if $G = {\rm L}_n(q)$ is a linear group.
\end{prop}

\begin{proof}
As noted above (see \eqref{e:omit}), we may assume $(n,q) \ne (4,2)$ and $q \geqs 11$ if $n=2$. And in view of Proposition \ref{p:gensmall}, we may assume $G \not\in \mathcal{B}$. Set $d = (n,q-1)$ and $e=d(q-1)$. As usual, we write $P_k$ for the stabiliser in $G$ of a $k$-dimensional subspace of $V$.
 
Let $x \in G$ be a Singer cycle. By Proposition \ref{p:singer}, $x$ is a witness and each $H \in \mathcal{M}(x)$ is a field extension subgroup of type ${\rm GL}_{n/k}(q^k)$, one for each prime divisor $k$ of $n$ (see \eqref{e:sing_lin}).

Let $y \in G$ be an element of order $(q^{n-1}-1)/d$ and note that $y$ fixes a decomposition $V = U \oplus W$ of the natural module, where $W$ is an $(n-1)$-dimensional subspace on which $y$ acts irreducibly. We claim that $y$ is a witness and
\[
\mathcal{M}(y) = \left\{ 
\begin{array}{ll}
\{ D_{2(q-1)/d}, \, P_1 \mbox{ {\rm (two)}} \} & \mbox{if $n=2$ and $q \geqs 13$} \\
\{ P_1, \, P_{n-1} \} & \mbox{if $n \geqs 3$ and $(n,q) \ne (3,2), (3,3), (3,4), (4,2)$,} 
\end{array}\right.
\]
so the desired result follows from Lemma \ref{l:us2}. 

To justify the claim, first assume $n=2$ and $q \geqs 13$. By inspecting \cite[Tables 8.1, 8.2]{BHR}, we deduce that $\mathcal{M}(y)$ comprises $N_G(\la y \ra) = D_{2(q-1)/d}$, together with two $P_1$ parabolic subgroups (note that $y$ is only contained in two such subgroups because it fixes exactly two $1$-dimensional subspaces of $V$). By applying the main theorem of \cite{LS} we deduce that ${\rm fpr}(z,G/H) < 1/3$ for all $z \in G$ of prime order and all $H \in \mathcal{M}(y)$, so  Lemma \ref{l:us} implies that  $y$ is a witness.

Now assume $n \geqs 3$. Clearly, $U$ and $W$ are the only proper nonzero subspaces of $V$ fixed by $y$, which implies that $y$ is contained in unique $P_1$ and $P_{n-1}$ parabolic subgroups. If we can show that $\mathcal{M}(y) = \{P_1, P_{n-1}\}$, then the result will follow via Lemmas \ref{l:us} and \ref{l:us2} since 
\[
{\rm fpr}(z,G/H) \leqs \frac{2^{n-1}-1}{2^n-1} < \frac{1}{2}
\]
for all $z \in G$ of prime order and $H \in \{P_1, P_{n-1}\}$ (maximal when $q=2$ and $z$ is a transvection). So it just remains to show that $y$ is not contained in a proper irreducible subgroup.

For $n=3$ this can be checked by inspecting the list of maximal subgroups of $G$ presented in \cite[Tables 8.3, 8.4]{BHR}, so let us assume $n \geqs 4$. The case $(n,q) = (7,2)$ can be verified directly, which means that we may assume $|y|$ is divisible by a primitive prime divisor $r$ of $q^{n-1}-1$. The maximal subgroups of $G$ containing an element of order $r$ are described in \cite{GPPS}. Using this, together with the fact that $|y| = (q^{n-1}-1)/d$, it is a straightforward exercise to show that $\mathcal{M}(y) = \{P_1, P_{n-1}\}$ as required (see the proof of Theorem \ref{t:psl}, where we used \cite{GPPS} to study the maximal overgroups of the element $y$ defined in Proposition \ref{p:msw}). This justifies the claim and the proof is complete.
\end{proof}

Next we prove Theorem \ref{t:gen} for the unitary groups.

\begin{prop}\label{p:unitgen}
The conclusion to Theorem \ref{t:gen} holds if $G = {\rm U}_n(q)$ is a unitary group.
\end{prop}

\begin{proof}
We may assume $n \geqs 3$. Recall that ${\rm U}_3(2)$ is soluble and ${\rm U}_3(3) \cong G_2(2)'$ was handled in Lemma \ref{l:ex_gensmall}, so we may assume $q \geqs 4$ when $n=3$. In addition, we may assume $G \not\in \mathcal{B}$ (see Proposition \ref{p:gensmall}). We will write $P_k$ (respectively, $N_k$) for the stabiliser of a $k$-dimensional totally isotropic (respectively, non-degenerate) subspace of the natural module $V$. Set $d = (n,q+1)$ and $e=d(q+1)$.

\vs

\noindent \emph{Case 1.} $n$ is odd

\vs

First assume $n$ is odd. Let $x \in G$ be a Singer cycle. Then Proposition  \ref{p:singer} implies that $x$ is a witness and \eqref{e:sing_unit} indicates that each $H \in \mathcal{M}(x)$ is a field extension subgroup of type ${\rm GU}_{n/k}(q^k)$. In particular, each maximal overgroup of $x$ acts irreducibly on $V$.

For $n = 2m+1 \geqs 9$, we define $y \in G$ as in \cite[Table II]{GK}. Then $y$ is a witness by \cite[Proposition 4.1]{GK} and the final column of \cite[Table II]{GK} indicates that 
each $H \in \mathcal{M}(y)$ acts reducibly on $V$. We now conclude by applying Lemma \ref{l:us2}.

Now assume $n \in \{3,5,7\}$. Let $y \in G$ be a regular semisimple element of order $q^{n-2}+1$ with $|C_G(y)| = (q^{n-2}+1)(q+1)/d$. More specifically, for $n=3$ we take $y$ to be the image of the diagonal matrix ${\rm diag}(\l,\l^2,\l^{q-2}) \in {\rm SU}_3(q)$ with respect to an orthonormal basis for $V$, where $\mathbb{F}_{q^2}^{\times} = \la \mu \ra$ and $\l = \mu^{q-1}$. Note that $y$ fixes an orthogonal decomposition 
\[
V = U \perp W_1 \perp W_2
\]
of the natural module, where each summand is non-degenerate and the $W_i$ are $1$-dimensional. 

We claim that 
\[
\mathcal{M}(y) = \left\{ \begin{array}{ll}
\{ G_{W_1}, \, G_{W_2}, \, G_{U^{\perp}} \} = \{ N_1 \mbox{ (two)}, \, N_2\} & \mbox{if $n \in \{5,7\}$ and $(n,q) \ne (5,2)$} \\
\{ G_{W_1}, \, G_{W_2}, \, G_{U}, \, K \} = \{ N_1 \mbox{ (three)}, \, K \} & \mbox{if $n=3$ and $q \geqs 7$,}
\end{array}\right.
\]
where $K$ is a subgroup of type ${\rm GU}_1(q) \wr S_3$.

Since $y$ acts irreducibly on $U$, it is clear that the reducible subgroups in $\mathcal{M}(y)$ are as described. For $n \in \{5,7\}$ we note that some power of $y$ has order $r$, where $r$ is a primitive prime divisor of $q^{2n-4}-1$, and it is easy to verify the above claim by inspecting the maximal subgroups of $G$ in \cite{BHR}. Now assume $n=3$, so $q \geqs 7$ and representatives of the conjugacy classes of maximal subgroups of $G$ are listed in \cite[Tables 8.5, 8.6]{BHR}. Visibly, $y$ fixes a unique orthogonal decomposition of $V$ into non-degenerate $1$-spaces, so it is contained in a unique subgroup of type ${\rm GU}_1(q) \wr S_3$. In addition, if $q$ is odd then $y$ is not contained in a subgroup of type ${\rm O}_3(q)$ (since by construction, $y$ does not lift to an element in ${\rm SU}_3(q)$ with a nonzero $1$-eigenspace). None of the remaining irreducible maximal subgroups of $G$ contain an element of order $|y|$ and this establishes the claim. 

We now complete the proof of the proposition for $n \in \{3,5,7\}$. Recall that we may assume $G \not\in \mathcal{B}$, so $q \geqs Q$ where $Q = 7,5,3$ for $n = 3,5,7$. In view of Lemma \ref{l:us2}, it suffices to show that $y$ is a witness. If $n=3$ then  $|\mathcal{M}(y)| = 4$ and the main theorem of \cite{LS} yields ${\rm fpr}(z,G/H) \leqs 4/21$ for all $H \in \mathcal{M}(y)$ and all $z \in G$ of prime order. Now apply Lemma \ref{l:us}. Similarly, if $n \in \{5,7\}$ and $q \geqs 5$, then the bound ${\rm fpr}(z,G/H) \leqs 4/15$ from \cite{LS} is good enough. So this just leaves $(n,q) = (7,3), (7,4)$. Here \cite[Proposition 3.16]{GK} implies that ${\rm fpr}(z,G/H) < 1/3$ for all $H \in \mathcal{M}(y)$ and once again the result follows from Lemma \ref{l:us}.

\vs

\noindent \emph{Case 2.} $n$ is even 

\vs

For the remainder, we may assume $n = 2m \geqs 4$ is even. In view of Proposition \ref{p:gensmall}, we may assume $q \geqs 7$ if $n=4$, and $q \geqs 4$ if $n=6$. Let $x \in G$ be an element of order $(q^{n-1}+1)/d$. Then \cite[Proposition 5.22]{BGK} gives $\mathcal{M}(x) = \{N_1\}$ and thus $x$ is a witness.
  
Suppose $n \geqs 8$ and let $y \in G$ be an element of order $(q^{m+\delta}+1)(q^{m-\delta}+1)/e$, where $\delta = 1$ if $m$ is even, otherwise $\delta = 2$. By inspecting \cite[Table II]{GK} we see that $\mathcal{M}(y) = \{ N_{m-\delta} \}$, so $y$ is a witness and we conclude via Lemma \ref{l:us2}. So to complete the proof, we may assume $n=4$ or $6$.

Suppose $n = 6$ with $q \geqs 4$ and let $y \in G$ be a regular semisimple element of order $(q^6-1)/e$. Now $y$ fixes a decomposition $V = U \oplus W$ of the natural module, where $U$ and $W$ are totally isotropic $3$-spaces on which $y$ acts irreducibly. In particular, $U$ and $W$ are the only proper nonzero subspaces of $V$ fixed by $y$, so 
$\mathcal{M}(y)$ contains two $P_3$ parabolic subgroups (namely, the stabilisers of $U$ and $W$), together with a unique subgroup of type ${\rm GL}_3(q^2)$ corresponding to the stabiliser of the decomposition $V = U \oplus W$. By inspecting \cite[Tables 8.26, 8.27]{BHR}, it is clear that any additional subgroup in $\mathcal{M}(y)$ must be a field extension subgroup of type ${\rm GU}_{2}(q^3)$. By arguing as in the proof of \cite[Lemma 8.7]{BH1}, we deduce that $y$ is contained in exactly one such subgroup and thus $|\mathcal{M}(y)| = 4$.
For $H = P_3$, \cite[Proposition 3.15]{GK} implies that ${\rm fpr}(z,G/H) < 1/6$ for all $z \in G$ of prime order, while the main theorem of \cite{LS} yields ${\rm fpr}(z,G/H) \leqs 1/3$ for $H$ of type ${\rm GL}_3(q^2)$ or ${\rm GU}_{2}(q^3)$. Therefore, \eqref{e:fpry} holds and thus $y$ is a witness by Lemma \ref{l:us}. As before, we now complete the argument by appealing to Lemma \ref{l:us2}.

Finally, suppose $n=4$ with $q \geqs 7$. Let $y \in G$ be a regular semisimple element of order $(q^4-1)/e$. By inspecting \cite[Tables 8.10, 8.11]{BHR}, it is easy to see that $\mathcal{M}(y)$ comprises two $P_2$ parabolic subgroups and a unique subgroup of type ${\rm GL}_2(q^2)$. As above, it is straightforward to show that \eqref{e:fpry} holds for all $z \in G$ of prime order and we now conclude via Lemmas \ref{l:us} and \ref{l:us2}.
\end{proof}

\begin{prop}\label{p:sympgen}
The conclusion to Theorem \ref{t:gen} holds if $G = {\rm PSp}_n(q)$ is a symplectic group.
\end{prop}

\begin{proof}
We may assume $n = 2m \geqs 4$, with $(n,q) \ne (4,2),(4,3)$ and $G \not\in \mathcal{B}$ (see \eqref{e:omit} and Proposition \ref{p:gensmall}). As before, we will write $P_k$ (respectively, $N_k$) for the stabiliser of a $k$-dimensional totally isotropic (respectively, non-degenerate) subspace of the natural module $V$. Set $d = (2,q-1)$. 

First assume $m \geqs 4$. Let $x \in G$ be a Singer cycle, so $x$ is a witness by Proposition \ref{p:singer} and the maximal overgroups of $x$ are listed in \eqref{e:sing_symp}. If $q$ is odd then we can choose $y \in G$ as in \cite[Table II]{GK}. Then $y$ is a witness by \cite[Proposition 4.1]{GK} and we note that every subgroup in $\mathcal{M}(y)$ acts reducibly on $V$. Since the maximal overgroups of $x$ are irreducible, the desired result follows via Lemma \ref{l:us2}.

Now assume $q$ is even (with $m \geqs 4$) and fix $y \in G$ with $|y| = {\rm lcm}(q^{m-1}+1,q+1)$. Then $y$ fixes an orthogonal decomposition $V = U \perp W$ of the natural module, where $\dim U = 2$ and $y$ acts irreducibly on both summands.  We claim that $y$ is a witness and $\mathcal{M}(y) = \{ N_2, {\rm O}_{n}^{+}(q) \}$, which allows us to conclude via Lemma \ref{l:us2} once again. 

In order to prove the claim, first observe that $U$ and $W$ are the only proper nonzero subspaces of $V$ fixed by $y$, so the stabiliser $N_2$ of $U$ is the only reducible subgroup in $\mathcal{M}(y)$. Let us also note that we may embed
\[
y \in {\rm O}_2^{-}(q) \times {\rm O}_{n-2}^{-}(q) < {\rm O}_{n}^{+}(q) = H < G
\]
and we deduce that $y$ is contained in a unique conjugate of $H$ since $y^G \cap H = y^H$ and $C_H(y) = C_G(y)$. In addition, we observe that the action of $y$ on $V$ is not compatible with containment in an orthogonal subgroup ${\rm O}_n^{-}(q)$. At this point, we can now work with \cite{GPPS} to rule out the existence of any additional subgroups in $\mathcal{M}(y)$, noting that $|y|$ is divisible by a primitive prime divisor of $q^{n-2}-1$ (recall that ${\rm Sp}_8(2)$ is in $\mathcal{B}$).

Here it is worth pausing to highlight the special case where $q = 2$, $m$ is even and $H = S_{n+2}$ is irreducibly embedded in $G$ via the fully deleted permutation module (see \cite[Example 2.6(a)]{GPPS}). Now $|y| = 2^{m-1}+1$ and \cite[Theorem 2]{MNR} gives
\[
\log |z| \leqs \sqrt{\ell\log \ell}\left(1 + \frac{\log\log \ell - \a}{2\log \ell}\right)
\]
for all $z \in H$, where $\ell =n+2$ and $\a = 0.975$ (where $\log$ denotes the natural logarithm). It is straightforward to check that this upper bound implies that $|z|< |y|$ for all $m \geqs 20$. And one can check directly that $H$ does not contain an element of order $|y|$ if $8 \leqs m \leqs 18$, so there are no such subgroups in $\mathcal{M}(y)$. However, if $m = 6$ then $H = S_{14}$ has a unique conjugacy class of elements of order $2^5+1 = 33$, so for $(n,q) = (12,2)$ one has to carefully choose $y$ in order to avoid containment in a conjugate of $H$. This can always be done since $G$ has three classes of such elements, only one of which meets $H$. But in any case, we should observe that the group ${\rm Sp}_{12}(2)$ was handled 
computationally in Proposition \ref{p:gensmall}.

This leaves us with the cases $n = 4,6$. For $n=6$ with $q \geqs 4$ we take $x$ to be a  Singer cycle and we define $y$ to be a regular semisimple element of order $(q^2+1)(q+1)/d$. Here $y$ fixes an orthogonal decomposition $V = U \perp W$, where $U$ and $W$ are non-degenerate spaces on which $y$ acts irreducibly, with $\dim U = 2$. Then $\mathcal{M}(y) = \{ N_2, {\rm O}_6^{+}(q)\}$ if $q$ is even, and $\mathcal{M}(y) = \{ N_2 \}$ if $q$ is odd (see \cite[p.767]{GK}), so the result follows by combining \eqref{e:sing_symp} with Lemma \ref{l:us2}.

Finally, let us assume $n = 4$ and $q \geqs 7$. Let $x \in G$ be a Singer cycle and note that $\mathcal{M}(x) = \{ {\rm Sp}_2(q^2).2, {\rm O}_4^{-}(q)\}$ if $q$ is even and $\mathcal{M}(x) = \{ H \}$ if $q$ is odd, where $H$ is a field extension subgroup of type ${\rm Sp}_2(q^2)$ (see \eqref{e:sing_symp}). We define $y \in G$ as follows. For $q$ even, we take $y$ to be a regular semisimple element of order $q+1$ with $|C_G(y)| = (q+1)^2$. More precisely, we take $y = {\rm diag}(A,A^2)$ with respect to a standard symplectic basis $\{e_1,f_1,e_2,f_2\}$ of the natural module, where $A \in {\rm Sp}_2(q)$ is a Singer cycle. 
And for $q$ odd, we choose $y$ of order $p(q+1)$ with $|C_G(y)| = q(q+1)$. We claim that $y$ is a witness and $\mathcal{M}(y) = \{ H,K \}$, where $H$ is of type ${\rm Sp}_2(q) \wr S_2$, and $K$ is of type ${\rm O}_4^{+}(q)$ for $q$ even and type $P_1$ for $q$ odd. Given the claim, the main theorem of \cite{LS} implies that \eqref{e:fpry} holds for all $z \in G$ of prime order and we conclude by applying Lemmas \ref{l:us} and \ref{l:us2}. So it just remains to determine the subgroups in $\mathcal{M}(y)$.

Suppose $q$ is even and note that the maximal subgroups of $G$ are listed in \cite[Table 8.14]{BHR}. Here $y$ fixes an orthogonal decomposition $V = U \perp W$ of the natural module, where $U$ and $W$ are non-degenerate $2$-spaces on which $y$ acts irreducibly. Since $U$ and $W$ are the only proper nonzero subspaces of $V$ fixed by $y$, it follows that $y$ is contained in a unique subgroup of type ${\rm Sp}_2(q) \wr S_2$. In addition, $y$ is not contained in a parabolic subgroup. Subgroups of type ${\rm Sp}_2(q^2)$ or ${\rm O}_4^{-}(q)$ do not contain regular semisimple elements of order $q+1$ with a trivial $1$-eigenspace, and we note that $|{}^2B_2(q)|$ is indivisible by $q+1$. Next suppose $L = {\rm Sp}_4(q_0)$ is a subfield subgroup with $q=q_0^k$ and $k$ is a prime. If $k$ is odd then $L$ does not contain any elements of order $q+1$. And if $k=2$ then each $z \in L$ of order $q+1 = q_0^2+1$ is a Singer cycle and the eigenvalues on $V \otimes \mathbb{F}_{q^2}$ of such an element are incompatible with those of $y$, so there are no subfield subgroups in $\mathcal{M}(y)$. Finally, we note that 
\[
y \in {\rm O}_2^{-}(q) \times {\rm O}_2^{-}(q) < {\rm O}_4^{+}(q) = H < G
\]
and we calculate that $y$ is contained in a unique conjugate of $H$. 

Now assume $q$ is odd. Here $y$ is the image (modulo scalars) of ${\rm diag}(A,B) \in {\rm Sp}_4(q)$ with respect to a symplectic basis $\{e_1,f_1,e_2,f_2\}$, where $A \in {\rm Sp}_2(q)$ is a Singer cycle and $B \in {\rm Sp}_2(q)$ is a transvection. Note that $y$ fixes a unique $1$-dimensional subspace of $V$, so it is contained in a unique $P_1$ parabolic subgroup. It also fixes a unique orthogonal decomposition $V = U \perp W$ with non-degenerate summands, so $y$ is contained in a unique subgroup of type ${\rm Sp}_2(q) \wr S_2$. Finally, by inspecting \cite[Tables 8.12, 8.13]{BHR} we observe that no other maximal subgroup of $G$ contains an element of order $p(q+1)$ and the proof is complete.
\end{proof}

In order to complete the proof of Theorem \ref{t:gen}, we may assume $G = {\rm P\O}_{n}^{\e}(q)$ is an orthogonal group with $n \geqs 7$. As before, $P_k$ will denote the stabiliser in $G$ of a $k$-dimensional totally singular subspace of $V$. In addition, if $k$ is even then we use $N_k^{\delta}$ to denote the stabiliser of a non-degenerate $k$-space of type $\delta \in \{+,-\}$. (Recall that a non-degenerate $k$-space has \emph{plus-type} if it contains a totally singular subspace of dimension $k/2$, otherwise it has \emph{minus-type}.)

\begin{prop}\label{p:orthoddgen}
The conclusion to Theorem \ref{t:gen} holds if $G = \O_n(q)$ is an orthogonal group with $n$ odd.
\end{prop}

\begin{proof}
Here $q$ is odd and $n=2m+1 \geqs 7$. In view of Proposition \ref{p:gensmall}, we may assume $(n,q) \ne (7,3)$. Let $x \in G$ be a regular semisimple element of order $(q^m+1)/2$ and note that $x$ fixes an orthogonal decomposition $V = U \perp W$ of $V$, where $U$ is a non-degenerate minus-type space of dimension $n-1$. By \cite[Proposition 5.20]{BGK} we have $\mathcal{M}(x) = \{ N_{n-1}^{-} \}$ (the proposition in \cite{BGK} is stated for $q \geqs 5$, but the given conclusion still holds when $q = 3$ and $n \geqs 9$). In particular, $x$ is a witness by Lemma \ref{l:us}.

We define $y = {\rm diag}(A,B)$, where $A \in \O_{n-3}^{-}(q)$ is a Singer cycle and $B \in \O_3(q)$ is a regular unipotent element (so $B$ has Jordan form $(J_3)$ on the natural $3$-dimensional module). Note that $|y| = p(q^{m-1}+1)/2$ and $y$ fixes an orthogonal decomposition $V = U \perp W$, where $U$ is a non-degenerate minus-type space of dimension $n-3$. We claim that $\mathcal{M}(y) = \{ P_1, N_{n-3}^{-} \}$, in which case the main theorem of \cite{LS} implies that \eqref{e:fpry} holds for all $z \in G$ of prime order and we conclude in the usual manner via Lemmas \ref{l:us} and \ref{l:us2}. So it remains to determine the subgroups in $\mathcal{M}(y)$.

To do this, first note that $y$ fixes exactly three proper nonzero subspaces of $V$, namely $U$ and $W$, together with the totally singular $1$-dimensional $1$-eigenspace of $B$ on $W$. Therefore, $P_1$ and $N_{n-3}^{-}$ are the only reducible subgroups in $\mathcal{M}(y)$. The existence of an irreducible subgroup in $\mathcal{M}(y)$ can be ruled out by appealing to \cite{GPPS}, noting that $|y|$ is divisible by a primitive prime divisor of $q^{n-3}-1$. We leave the reader to check the details.
\end{proof}

\begin{prop}\label{p:orthminusgen}
The conclusion to Theorem \ref{t:gen} holds if $G = {\rm P\O}_n^{-}(q)$ is an orthogonal group.
\end{prop}

\begin{proof}
We may assume $n = 2m \geqs 8$. Let $x \in G$ be a Singer cycle, so $x$ is a witness by Proposition \ref{p:singer} and the subgroups in $\mathcal{M}(x)$ are recorded in \eqref{e:sing_ominus}. For $n \geqs 14$ we can define $y \in G$ as in \cite[Table II]{GK}. Then $y$ is a witness by \cite[Proposition 4.1]{GK}, while  \cite[Table II]{GK} indicates that every subgroup in $\mathcal{M}(y)$ acts reducibly on $V$. So for $n \geqs 14$, we conclude by applying Lemma \ref{l:us2}.

To complete the proof, we may assume $n \in \{8,10,12\}$. For $q \in \{2,3\}$ we can use {\sc Magma} to verify Theorem \ref{t:gen} directly (see Proposition \ref{p:gensmall}), so we are free to assume $q \geqs 4$. Let $y \in G$ be a regular semisimple element of type $(n-2)^{-} \perp (1 \oplus 1)$, so $y$ fixes a decomposition $V = U \perp (W_1 \oplus W_2)$ of the natural module, where $U$ is a non-degenerate minus-type space of dimension $n-2$ on which $y$ acts as a Singer cycle, and the $W_i$ are totally singular $1$-spaces. We claim that $\mathcal{M}(y) = \{H,K,L\}$, where $H$ and $K$ are $P_1$ parabolic subgroups and $L = N_{n-2}^-$.  

By construction, $y$ fixes exactly $4$ proper nonzero subspaces of $V$, namely $U$, $W_1$, $W_2$ and $U^{\perp} = W_1 \oplus W_2$. Therefore, $H$, $K$ and $L$ are the only reducible subgroups in $\mathcal{M}(y)$. By inspecting the relevant tables in \cite[Chapter 8]{BHR}, it is easy to see that there are no additional subgroups in $\mathcal{M}(y)$ and the claim follows.

Finally, let $z \in G$ be an element of prime order. Then by applying the upper bounds on ${\rm fpr}(z,G/H)$ in \cite[Propositions 3.15, 3.16]{GK}, we deduce that \eqref{e:fpry} holds  and the result follows via Lemma \ref{l:us2}.
\end{proof}

\begin{prop}\label{p:orthplusgen}
The conclusion to Theorem \ref{t:gen} holds if $G = {\rm P\O}_n^{+}(q)$ is an orthogonal group with $n \geqs 10$.
\end{prop}

\begin{proof}
Write $n=2m$. In view of Proposition \ref{p:gensmall}, we may assume $q \geqs 4$ if $n \in  \{10,12\}$. Define $x$ as in \cite[Proposition 5.13]{BGK} for $m$ odd, and as in \cite[Theorem 2.14]{Saul} for $m$ even. Then as explained in \cite{BGK,Saul}, we have  
\[
\mathcal{M}(x) = \left\{ \begin{array}{ll}
\{ N_{m-1}^{-} \} & \mbox{if $m$ is odd} \\
\{ N_{m-2}^{-}, H,K\} & \mbox{if $m$ is even,}
\end{array}\right.
\]
where $H$ and $K$ are non-conjugate field extension subgroups of type ${\rm O}_{m}^{+}(q^2)$ if $m \equiv 2 \imod{4}$ and type ${\rm GU}_m(q)$ if $m \equiv 0 \imod{4}$. Clearly, if $m$ is odd then $x$ is a witness. And the same conclusion holds for $m$ even, in view of Lemma \ref{l:us} and the fixed point ratio bounds in \cite{Bur1} and \cite[Proposition 3.16]{GK}.

Let $y \in G$ be a regular semisimple element fixing an orthogonal decomposition $V = U \perp W$ into minus-type non-degenerate spaces, where $\dim U = 2$ and $y$ acts irreducibly on both summands. More precisely, we take $y$ to be the image (modulo scalars) of an element in $\O_n^{+}(q)$ of order ${\rm lcm}(q^{m-1}+1,q+1)/d$, where $d = (2,q-1)$. We claim that $\mathcal{M}(y)$ comprises $N_2^{-}$ (namely, the stabiliser $G_U = G_W$), together with two non-conjugate field extension subgroups of type ${\rm GU}_m(q)$ when $m$ is even. 

To see this, first note that $U$ and $W$ are the only proper nonzero subspaces of $V$ fixed by $y$, so $G_U$ is the only reducible subgroup in $\mathcal{M}(y)$. Since $|y|$ is divisible by a primitive prime divisor of $q^{n-2}-1$, we can work with \cite{GPPS} to show that the only additional subgroups in $\mathcal{M}(y)$ are of type ${\rm GU}_m(q)$ with $m$ even. So let us assume $m$ is even and note that $G$ has two conjugacy classes of subgroups $H$ of type ${\rm GU}_m(q)$ (see \cite[Proposition 4.3.18]{KL}). We calculate that $y^G \cap H = y^H$ and $C_G(y) = C_H(y)$, so $y$ is contained in a unique conjugate of $H$ and the claim follows. 

Finally, we can use Lemma \ref{l:us} to show that $y$ is a witness and we conclude by applying Lemma \ref{l:us2}.
\end{proof}

The following proposition completes the proof of Theorem \ref{t:gen}.

\begin{prop}\label{p:orthplus8gen}
The conclusion to Theorem \ref{t:gen} holds if $G = {\rm P\O}_8^{+}(q)$.
\end{prop}

\begin{proof}
The groups with $q \leqs 5$ can be handled using {\sc Magma} (see Proposition \ref{p:gensmall}), so we may assume $q \geqs 7$. Define $x \in G$ as in \cite[Theorem 2.14]{Saul}, which tells us that $\mathcal{M}(x) = \{H,K,L\}$, where $H$ and $K$ are field extension subgroups of type ${\rm O}_4^{+}(q^2)$ and $L$ is of type ${\rm O}_4^{-}(q) \wr S_2$. In addition, define $y$ as in the proof of Proposition \ref{p:orthplusgen}. Then $\mathcal{M}(y)$ contains a unique reducible subgroup of type $N_2^{-}$, together with two non-conjugate field extension subgroups of type ${\rm GU}_4(q)$ (see \cite[p.767]{GK}), so $y$ is a witness and the result follows from Lemma \ref{l:us2}. 
\end{proof}

\vs

This completes the proof of Theorem \ref{t:main6}.

\end{document}